\documentclass[11pt,reqno]{amsart}

\usepackage{amsmath}
\usepackage{amsfonts}
\usepackage{amssymb}
\usepackage{amsthm}
\usepackage{mathrsfs}
\usepackage{latexsym}

\usepackage{cite} 

\usepackage{esint}

\numberwithin{equation}{section}

\usepackage{graphicx,subfigure}

\usepackage{ifthen} 

\provideboolean{shownotes} 
\setboolean{shownotes}{true} 
\newcommand{\margnote}[1]{
\ifthenelse{\boolean{shownotes}}%
{\marginpar{\raggedright\tiny\texttt{#1}}}%
{}%
}

\newcommand{\hole}[1]{
\ifthenelse{\boolean{shownotes}}%
{\begin{center} \fbox{ \rule {.25cm}{0cm}
\rule[-.1cm]{0cm}{.4cm} \parbox{.85\textwidth}{\begin{center}
\texttt{#1}\end{center}} \rule {.25cm}{0cm}}\end{center}}
{}
}

\usepackage[colorlinks=true,urlcolor=blue,
citecolor=red,linkcolor=blue,linktocpage,pdfpagelabels,
bookmarksnumbered,bookmarksopen,backref=page]{hyperref}

\usepackage{fullpage}

\theoremstyle{plain}

\newtheorem{lemma}{Lemma}[section]
\newtheorem{theorem}[lemma]{Theorem}
\newtheorem{proposition}[lemma]{Proposition}
\newtheorem{corollary}[lemma]{Corollary}

\theoremstyle{definition}
\newtheorem{remark}[lemma]{Remark}
\newtheorem{definition}[lemma]{Definition}

\theoremstyle{remark}

\newcommand{\Id}{\mathrm{Id}}

\newcommand{\A}{\mathbb{A}}
\newcommand{\I}{\mathbb{I}}
\newcommand{\F}{\mathbb{F}}
\newcommand{\R}{\mathbb{R}}
\newcommand{\C}{\mathbb{C}}
\newcommand{\Z}{\mathbb{Z}}
\newcommand{\N}{\mathbb{N}}
\newcommand{\bbS}{\mathbb{S}}
\newcommand{\U}{\mathbb{U}}
\newcommand{\M}{\mathbb{M}}

\newcommand{\tM}{\widetilde{M}}
\newcommand{\tG}{\widetilde{G}}

\newcommand{\bF}{\overline{F}}
\newcommand{\bA}{\overline{A}}
\newcommand{\bU}{\overline{U}}
\newcommand{\bV}{\overline{V}}
\newcommand{\bw}{\overline{w}}

\newcommand{\cT}{{\mathcal{T}}}

\newcommand{\cL}{{\mathcal{L}}}
\newcommand{\cD}{{\mathcal{D}}}

\newcommand{\cS}{{\mathcal{S}}}

\newcommand{\cP}{{\mathcal{P}}}
\newcommand{\cA}{{\mathcal{A}}}
\newcommand{\cB}{{\mathcal{B}}}

\newcommand{\vep}{\varepsilon}

\renewcommand{\Re}{\mathrm{Re}\,} 
\renewcommand{\Im}{\mathrm{Im}\,}

\newcommand{\sgn}{\mathrm{sgn}\,}

\newcommand{\Span}{\mathrm{span}\,}

\newcommand{\ep}{\epsilon}

\newcommand{\ess}{\sigma_\mathrm{\tiny{ess}}}
\newcommand{\ptsp}{\sigma_\mathrm{\tiny{pt}}}

\newcommand{\Ldper}{L^2_\mathrm{\tiny{per}}}
\newcommand{\Huper}{H^1_\mathrm{\tiny{per}}}
\newcommand{\Hdper}{H^2_\mathrm{\tiny{per}}}
\newcommand{\Hmper}{H^m_\mathrm{\tiny{per}}}

\newcommand{\<}{\langle}
\renewcommand{\>}{\rangle}

\begin{document}

\title[Spectral instability of periodic waves for viscous balance laws]{Existence and spectral instability of bounded spatially  periodic traveling waves for scalar viscous balance laws}

\author[E. \'{A}lvarez]{Enrique \'{A}lvarez}
 
\address{{\rm (E. \'{A}lvarez)} Instituto de 
Investigaciones en Matem\'aticas Aplicadas y en Sistemas\\Universidad Nacional Aut\'onoma de 
M\'exico\\ Circuito Escolar s/n, Ciudad Universitaria, C.P. 04510 Cd. de M\'{e}xico (Mexico)}

\email{enrique.alvarez@ciencias.unam.mx}

\author[R. G. Plaza]{Ram\'on G. Plaza}

\address{{\rm (R. G. Plaza)} Instituto de 
Investigaciones en Matem\'aticas Aplicadas y en Sistemas\\Universidad Nacional Aut\'onoma de 
M\'exico\\ Circuito Escolar s/n, Ciudad Universitaria, C.P. 04510 Cd. de M\'{e}xico (Mexico)}

\email{plaza@mym.iimas.unam.mx}

\begin{abstract}
This paper studies both existence and spectral stability properties of bounded spatially periodic traveling wave solutions to a large class of scalar viscous balance laws in one space dimension with a reaction function of monostable or Fisher-KPP type. Under suitable structural assumptions, it is shown that this class of equations underlies two families of periodic waves. The first family consists of small amplitude waves with finite fundamental period which emerge from a Hopf bifurcation around a critical value of the wave speed. The second family pertains to arbitrarily large period waves which arise from a homoclinic bifurcation and tend to a limiting traveling (homoclinic) pulse when their fundamental period tends to infinity. For both families, it is shown that the Floquet (continuous) spectrum of the linearization around the periodic waves intersects the unstable half plane of complex values with positive real part, a property known as spectral instability. For that purpose, in the case of small-amplitude waves it is proved that the spectrum of the linearized operator around the wave can be approximated by that of a constant coefficient operator around the zero solution and determined by a dispersion relation which intersects the unstable complex half plane. In the case of large period waves, we verify that the family satisfies the assumptions of the seminal result by Gardner \cite{Grd2} of convergence of periodic spectra in the infinite-period limit to that of the underlying homoclinic wave, which is unstable. A few examples are discussed.
\end{abstract}



\maketitle

\setcounter{tocdepth}{1}



\section{Introduction}

In this contribution, we consider scalar viscous balance laws in one space dimension of the form,
\begin{equation}
\label{gvbl}
u_t + f(u)_x = \nu u_{xx} + g(u), 
\end{equation}
where $u = u(x,t) \in \R$ and $x \in \R$, $t > 0$. Here $f = f(u)$ denotes a nonlinear flux function and $g = g(u)$ is a balance (or reaction) term expressing production of the quantity $u$. Viscosity (or diffusion) effects are modeled through the Laplace operator applied to $u$ with constant viscosity coefficient, $\nu > 0$. When $f \equiv 0$ the equation reduces to the standard reaction-diffusion equation for which the existence and the stability of traveling waves have been widely investigated (see, e.g., \cite{AronWein75,Fif79,FiM77} and the many references therein).

Scalar viscous balance laws typically arise as parabolic regularizations of hyperbolic balance laws of the form (cf. \cite{Da84,Da4e}),
\begin{equation}
\label{hbl}
u_t + f(u)_x = g(u), 
\end{equation}
also known as \emph{inhomogeneous} conservation laws \cite{FanHa93}, describing idealized inviscid problems in which only reaction and convective effects are taken into consideration (for example, the equation \eqref{hbl} may describe the evolution of a density $u$ of point particles moving with speed $f'(u)$ and reacting at rate $g(u)/u$). In the theory of scalar conservation laws (cf. \cite{Da4e,La57}), it is well known that the convexity of the flux function $f$ plays a key role and determines the structure of entropy solutions. The introduction of the reaction term $g(u)$ (which may describe production/consumption, chemical reactions or combustion, among other interactions) is capable of drastically change the long time behavior of solutions, as was demonstrated by Mascia in both the convex \cite{Masc97a} and non-convex cases \cite{Masc98,Masc00}. Applications of balance laws, although not scalar, include models for roll waves \cite{BJNRZ10,Nob07}, nozzle flow \cite{ChGl96}, or combustion theory \cite{CoMR86}. Thus, scalar and systems of balance laws have been the subject of investigations for a long time (for an abridged list of references, see \cite{Da84,FanHa93,FanHa95,Masc97a,Masc98,Masc00,MasSi97,Snst97}; see also the recent paper \cite{DuRo20} on scalar equations). Since the effects of diffusion are important in many physical applications (such as viscous fluid flow \cite{Bur48} and semiconductor theory \cite{HaSch06,Scholl01}), viscous balance laws have been proposed to account for such effects. In the scalar case, it is common to find viscous balance laws as tools to study viscous profiles as approximations of their inviscid wave counterparts when the viscosity coefficient $\nu$ is small (see, for example, \cite{CroMas07,Hae00,Hae03}). In sum, scalar viscous balance laws represent simplified models that combine diffusion (viscosity), convection and reaction effects into one single equation.

We are interested in analyzing spatially periodic traveling wave solutions to a large class of viscous balance laws of the form \eqref{gvbl}. In the literature, there exist several works addressing the stability \emph{per se} of traveling wave solutions to equation \eqref{gvbl}. For example, the existence and nonlinear (asymptotic) stability of traveling \emph{front} solutions for viscous balance laws have been studied by Wu and Xing \cite{WuXi05}. In particular, they analyze the spectrum of the linearized operator around the fronts on the real line, leading to the concept of spectral stability. Their analysis has been extended to the non-convex case in \cite{Xing05}. Scalar viscous balance laws with degenerate viscosity coefficients have been recently studied by Xu \emph{et al.} \cite{XuJJ16}. Regarding spatially periodic traveling waves, there exist many papers addressing the existence problem (and asymptotic behavior) for specific equations; see, for instance, \cite{LeHa16,LuYJ07,Vall18a,Vall18b,ZLZ13}, among others. Up to our knowledge, the stability of periodic wave solutions to equations in the general form \eqref{gvbl} has not been studied before in the literature.

In this paper, we prove the existence of bounded, spatially periodic traveling wave solutions to a large class of equations of the general form \eqref{gvbl} under suitable structural assumptions. In particular, we suppose that the reaction function is of monostable or Fisher-KPP type (see hypothesis \eqref{H2} below). We prove that this class of equations underlies two families of periodic waves. The first family consists of small amplitude waves with finite fundamental period which emerge from a Hopf bifurcation around a critical value of the wave speed. The Hopf bifurcation can be either subcritical or supercritical. The second family includes arbitrarily large period waves arising from a homoclinic bifurcation around a second critical value of the speed, and which tend to a limiting traveling (homoclinic) pulse when their fundamental period tends to infinity. First, we apply Melnikov's integral method for perturbed homoclinic orbits to show the existence of a traveling pulse for the viscous balance law. The wave speed of the latter is precisely the bifurcation parameter value under which there happens the bifurcation of a limit cycle from the homoclinic loop of a saddle with non-zero saddle quantity, as a consequence of Andronov-Leontovich's theorem. In this fashion, we exhibit the existence of bounded periodic waves with large fundamental period tending to infinity (the ``period" of the homoclinic loop) as the speed of the wave tends to the critical homoclinic speed.

In addition, we also study the stability properties of both families of periodic waves. For that purpose, we linearize the equation around the wave under consideration and study the associated spectral problem. This procedure leads to the concept of \emph{spectral stability}, or, in lay terms, the property that the linearized operator around the wave is ``well-behaved" in the sense that it does not support eigenvalues with positive real part (see Definition \ref{defspectstab} below). In the case of periodic waves, the spectrum of the linearized operator is continuous and we are required to study the \emph{Floquet spectrum}, comprised by curves of spectrum on the complex plane. In this paper, it is shown that both families of periodic waves are spectrally unstable, that is, that the corresponding Floquet spectra intersect the unstable complex half plane of eigenvalues with positive real part. In the case of small amplitude waves, we prove that the spectrum of the linearized operator around the wave can be approximated by that of a constant coefficient operator around the zero solution and determined by a dispersion relation which intersects the unstable complex half plane. Applying standard perturbation theory of linear operators, we show that unstable point eigenvalues of the constant coefficient operator split into neighboring curves of Floquet spectra of the underlying small amplitude waves. In the case of large period waves nearby homoclinic loops, the pioneering work by Gardner \cite{Grd2} characterized the spectrum of the linearized operator around such periodic waves and related it to that of the linearized operator around the homoclinic loop (traveling pulse for the PDE). Gardner proved the convergence of both spectra in the infinite period limit and, under very general conditions, that loops of continuous periodic spectra bifurcate from isolated point spectra of the limiting homoclinic wave. Hence, the typical spectral instability of the traveling pulse determines the spectral instability of the periodic waves under consideration. We then verify the hypotheses of a recent refinement of Gardner's result due to Yang and Zumbrun \cite{YngZ19} to conclude the spectral instability of the family. Finally, we present some examples of viscous balance laws that satisfy the hypotheses in this paper, for which our existence and instability results apply.

\subsection{Equations and assumptions}
\label{seceqsass}

For concreteness, in this paper we consider scalar viscous balance laws in one space dimension of the form
\begin{equation}
\label{vbl}
u_t + f(u)_x = u_{xx} + g(u),
\end{equation}
where $f$ and $g$ are sufficiently regular functions and where the viscosity coefficient has been set to $\nu = 1$. It is assumed that $g$ is of Fisher-KPP type \cite{Fis37,KPP37} having two equilibrium points, one stable and one unstable, and which models growth of logistic type (see assumption \eqref{H2} below). Thus, the model equation under consideration covers a large class of general viscous balance laws as $f$ is not required to be strictly convex. When one analyzes the convergence of viscous waves to their hyperbolic counterparts one cannot get rid of the viscosity parameter by simple scaling due to the presence of the reaction term and, hence, the problem is singularly perturbed (see \cite{Hae00}). In the present context, however, the analysis of existence and stability of the viscous wave does not involve the inviscid limit and it is possible to normalize the space variable and the flux function, $x \to x/\sqrt{\nu}$ and $f \to f/\sqrt{\nu}$, respectively, so that the model \eqref{gvbl} reduces to equation \eqref{vbl} without loss of generality.

In the sequel, we make the following assumptions on the nonlinear functions $f$ and $g$:
\begin{equation}
\label{H1}
\tag{H$_1$}
f \in C^4(\R).
\end{equation}
\begin{equation}
\label{H2}
\tag{H$_2$}
\begin{minipage}[c]{5in}
$g \in C^3(\R)$ and it is of Fisher-KPP type, satisfying
\[
	\begin{aligned}
	&g(0) = g(1) =0,\\
	&g'(0) > 0, \, g'(1) < 0,\\
	&g(u)>0 \, \textrm{ for all } \, u \in(0,1),\\
	&g(u)<0 \, \textrm{ for all } \, u \in (-\infty,0).
	\end{aligned}
\]
\end{minipage}
\end{equation}
\begin{equation}
\label{H3}
\tag{H$_3$}
\begin{minipage}[c]{5in}
There exists $u_* \in (-\infty, 0)$ such that
\[
\int_{u_*}^0 g(s) \, ds + \int_0^1 g(s) \, ds= 0.
\]
\end{minipage}
\end{equation}

\smallskip

We also make some assumptions on the interaction between $f$ and $g$. For instance, we suppose that:
\begin{equation}
\label{H4}
\tag{H$_4$}
\overline{a}_0 := f'''(0) - \frac{f''(0) g''(0)}{\sqrt{g'(0)}} \neq 0, \qquad \text{(genericity condition).}
\end{equation}

Observe that, under \eqref{H2} and \eqref{H3}, $u_* \in (-\infty,0)$ is the unique value such that \eqref{H3} holds and, moreover,
\[
\int_u^1 g(s) \, ds > 0, \qquad \text{for all } \; u \in (u_*,1).
\]
Therefore we can define 
\begin{equation}
\label{defPsi}
\gamma(u):= \sqrt{2\int_u^1 g(s) \, ds}, \qquad u \in (u_*,1),
\end{equation}
as well as the integrals
\begin{equation}
\label{lasIs}
\begin{aligned}
I_0 &:= \int_{u_*}^1 \gamma(s) \, ds > 0,\\
I_1 &:= \int_{u_*}^1 f'(s) \gamma(s) \, ds,\\
J &:= 2 \int_{u_*}^1 f'(s) \sqrt{1+\gamma'(s)^2} \, ds,\\
\text{and} \quad L &:= 2 \int_{u_*}^1 \sqrt{1+\gamma'(s)^2} \, ds.
\end{aligned}
\end{equation}
Notice that $L$ and $J$ are, typically, elliptic integrals; $L$ is simply the length of the curve defined by the function $\gamma$ and it clearly exists; since $f$ is of class $C^4$ this implies that $J$ exists as well. Based on the above definitions we further assume:
%

\begin{align}
I_0 J &\neq L I_1,  &\text{(non-degeneracy condition),} \label{H5} \tag{H$_5$}\\
f'(1) &\neq \displaystyle{\frac{I_1}{I_0}},  &\text{(saddle condition).} \label{H6} \tag{H$_6$}
\end{align}

\begin{remark}
Hypothesis \eqref{H1} is a minimal regularity assumption on $f$ to guarantee the existence of small amplitude periodic waves. We emphasize that we do not require the nonlinear flux $f$ to be strictly convex ($f''(u) \ge \delta > 0$ for all $u$) and that our results apply to general flux functions with inflection points (such as the Buckley-Leverett flux function \eqref{BLf} below) which are useful in the description of non-classical shocks and phase transitions (see LeFloch \cite{LeF02} for further information). Assumption \eqref{H2} specifies a balance (or production) term with logistic response, with an unstable equilibrium point at $u = 0$ and a stable one at $u=1$, plus the required regularity. Hypothesis \eqref{H3} is the balance of forces condition (if we interpret $g$ as the derivative of a potential) necessary for the existence of a homoclinic orbit. Assumptions \eqref{H4}, \eqref{H5} and \eqref{H6} determine the interaction conditions between $f$ and $g$ which are sufficient to guarantee the existence of bounded periodic waves. In particular, hypothesis \eqref{H4} ensures a Hopf bifurcation from which small-amplitude periodic orbits with bounded fundamental period emerge. Assumptions \eqref{H5} and \eqref{H6} make sure that a bifurcation from a limit cycle occurs from a homoclinic loop (traveling pulse type solution) giving rise to bounded periodic waves with large fundamental period.
\end{remark}

\subsection{Main results}
\label{secmainres}

The first theorem pertains to the existence of small-amplitude bounded periodic traveling wave solutions to equations of the form \eqref{vbl} that emerge from a Hopf bifurcation around a critical value of the speed.

\begin{theorem}[existence of small amplitude periodic waves]
\label{thmexbded}
Suppose that conditions \eqref{H1} thru \eqref{H4} hold. Then there exist a critical speed given by
\begin{equation}
\label{defc0}
c_0 := f'(0),
\end{equation}
and $\ep_0 > 0$ sufficiently small such that, for each $0 < \ep < \ep_0$ there exists a unique periodic traveling wave solution to the viscous balance law \eqref{vbl} of the form $u(x,t) = \varphi^\ep(x - c(\ep)t)$, traveling with speed $c(\ep) = c_0 + \ep$ if $\overline{a}_0 > 0$, or $c(\ep) = c_0 - \ep$ if 
$\overline{a}_0 < 0$, and with fundamental period,
\begin{equation}
\label{fundphopf}
T_\ep = \frac{2 \pi}{\sqrt{g'(0)}} + O(\ep), \qquad \text{as } \, \ep \to 0^+.
\end{equation}
The profile function $\varphi^\ep = \varphi^\ep(\cdot)$ is of class $C^3(\R)$, satisfies $\varphi^\ep(z + T_\ep) = \varphi^\ep(z)$ for all $z \in \R$ and is of small amplitude, more precisely,
\begin{equation}
\label{bdsmalla}
|\varphi^\ep(z)|, |(\varphi^\ep)'(z)| \leq C \sqrt{\ep},
\end{equation}
for all $z \in \R$ and some uniform $C > 0$.
\end{theorem}

The second theorem guarantees the existence of large-period, bounded periodic waves that emerge from a homoclinic bifurcation around a critical value for the speed, which is the speed of the homoclinic orbit. The existence of a homoclinic orbit (traveling pulse solution) is a consequence of Melnikov's method (see Theorem \ref{theoexisthomo} below).

\begin{theorem}[existence of large period waves]
\label{thmexlarge}
Under assumptions \eqref{H1} - \eqref{H3}, \eqref{H5} and \eqref{H6}, there is a critical speed given by
\begin{equation}
\label{defc1}
c_1 := \frac{I_1}{I_0},
\end{equation}
such that there exists a traveling pulse solution (homoclinic orbit) to equation \eqref{vbl} of the form  $u(x,t) = \varphi^0(x - c_1 t)$, traveling with speed $c_1$ and satisfying $\varphi^0 \in C^3(\R)$ and
\[
|\varphi^0(z) - 1|, |(\varphi^0)'(z)| \leq C e^{-\kappa |z|},
\]
for all $z \in \R$ and some $\kappa > 0$. In addition, one can find $\ep_1 > 0$ sufficiently small such that, for each $0 < \ep < \ep_1$ there exists a unique periodic traveling wave solution to the viscous balance law \eqref{vbl} of the form $u(x,t) = \varphi^\ep(x - c(\ep)t)$, traveling with speed $c(\ep) = c_1 + \ep$ if $f'(1) < c_1$ or $c(\ep) = c_1 - \ep$ if $f'(1) > c_1$, with fundamental period
\begin{equation}
\label{fundphomo}
T_\ep = O(| \log \ep |) \to \infty, 
\end{equation}
and amplitude
\begin{equation}
\label{bdO1}
|\varphi^\ep(z)|, |(\varphi^\ep)'(z)| = O(1),
\end{equation}
as $\ep \to 0^+$. Moreover, the family of periodic orbits converge to the homoclinic or traveling pulse solution as $\ep \to 0^+$ and satisfy the bounds (after a suitable reparametrization of $z$),
\begin{equation}
\label{boundw1}
\sup_{z \in [-\frac{T_\ep}{2}, \frac{T_\ep}{2}]} \left( |\varphi^0(z) - \varphi^\ep(z)| + |(\varphi^0)'(z) - (\varphi^\ep)'(z)  |\right) \leq C \exp \Big( \!- \kappa \frac{T_\ep}{2}\Big), 
\end{equation}
\begin{equation}
\label{boundc1}
| c_1 - c(\ep)| = \ep \leq C \exp \big( \!- \kappa T_\ep \big),
\end{equation}
for some uniform $C > 0$, the same $\kappa > 0$ and for all $0 < \ep < \ep_1$.
\end{theorem}

The following results pertain to the stability properties of both families of periodic waves as solutions to the viscous balance law. The first theorem establishes the spectral instability of the small-amplitude periodic waves. The main idea behind the proof is that, since the waves have small-amplitude, the spectrum of the linearized operator around the wave  (see Definition \ref{defspectstab} below) can be approximated by that of a constant coefficient operator around the zero solution which is, in turn, determined by a dispersion relation curves intersecting the unstable complex half plane.

\begin{theorem}[spectral instability of small-amplitude waves]
\label{theosmalli}
Under conditions \eqref{H1} thru \eqref{H4}, there exists $0 < \bar{\ep}_0 < \ep_0$ such that every small-amplitude periodic wave $\varphi^\ep$ from Theorem \eqref{thmexbded} with $0 < \ep < \bar{\ep}_0$ is spectrally unstable, that is, the spectrum of linearized operator around the wave intersects the unstable half plane $\C_+ = \{ \lambda \in \C \, : \, \Re \lambda > 0\}$.
\end{theorem}

The second stability result establishes the spectral instability of the family of large-period waves. The main idea behind this behavior is that the family satisfies the assumptions of the classical result by Gardner \cite{Grd2} (see also \cite{SS6,YngZ19}) of convergence of periodic spectra in the infinite-period limit to that of the underlying homoclinic wave, which is spectrally unstable.

\begin{theorem}[spectral instability of large period waves]
\label{theolargei}
Under assumptions \eqref{H1} - \eqref{H3}, \eqref{H5} and \eqref{H6}, there exists $0 < \bar{\ep}_1 < \ep_1$ such that every small-amplitude periodic wave $\varphi^\ep$ from Theorem \eqref{thmexlarge} with $0 < \ep < \bar{\ep}_1$ is spectrally unstable, that is, the spectrum of linearized operator around the wave intersects the unstable half plane $\C_+ = \{ \lambda \in \C \, : \, \Re \lambda > 0\}$.
\end{theorem}

\subsection*{Plan of the paper}

Section \S \ref{secexi} is devoted to proving the existence of the two types of bounded periodic waves for this large class of model equations: those with small amplitude, and those with large period associated to a homoclinic orbit (in particular, Appendix \S \ref{apenon} contains the verification of a non-degeneracy property of the latter). Section \S\ref{secspe} sets up the stability problem and defines spectral stability in terms of the Floquet spectrum. Section \S\ref{secsmall} is devoted to prove the spectral instability of small-amplitude periodic waves, whereas section \S\ref{seclarge} contains the proof of instability of the large period waves. A few model examples satisfying the assumptions of this paper can be found in section \S\ref{secexa}. Finally, we make some concluding remarks in section \S\ref{secdisc}.

\subsection*{On notation}
Linear operators acting on infinite-dimensional spaces are indicated with calligraphic letters (e.g., $\cL$ and $\cT$), except for the identity operator which is indicated by $\Id$. The domain of a linear operator, $\cL : X \to Y$, with $X$, $Y$ Banach spaces, is denoted as $\cD(\cL) \subseteq X$. We denote the real and imaginary parts of a complex number $\lambda \in \C$ by $\Re\lambda$ and $\Im\lambda$, respectively, as well as complex conjugation by ${\lambda}^*$. Complex transposition of matrices are indicated by the symbol $A^*$, whereas simple transposition is denoted by the symbol $A^\top$. For any linear operator $\mathcal{L}$, its formal adjoint is denoted by $\mathcal{L}^*$. Standard Sobolev spaces of complex-valued functions on the real line will be denoted as $L^2(\R;\C)$ and $H^m(\R;\C)$, with $m \in \N$, endowed with the standard inner products,
\[
\langle u,v \rangle_{L^2} = \int_\R u(x) v(x)^* \, dx, \qquad \langle u,v \rangle_{H^m} = \sum_{k=1}^m \langle \partial_x^k u, \partial_x^k v \rangle_{L^2},
\]
and corresponding norms $\| u\|_{L^2}^2 = \< u,u \>_{L^2}$, $\| u\|_{H^m}^2 = \< u,u \>_{H^m}$. For any $T > 0$, we denote by $\Ldper([0,T];\C)$ the Hilbert space of complex $T$-periodic functions in $L^2_\mathrm{\tiny{loc}}(\R)$ satisfying
\[
u(x + T) = u(x), \qquad \text{a.e. in } \; x,
\]
and with inner product and norm
\[
\< u, v \>_{\Ldper} = \int_0^T u(x) v(x)^* \, dx, \qquad \| u \|^2_{\Ldper} = \< u, u \>_{\Ldper}.
\]
For any $m \in \N$, the periodic Sobolev space $\Hmper([0,T],\C)$ will denote the set of all functions $u \in \Ldper([0,T];\C)$ with all weak derivatives up to order $m$ in $\Ldper([0,T];\C)$. By Sobolev's lemma (see, e.g., Iorio and Iorio \cite{IoIo01}), $\Hmper \hookrightarrow C^k_\mathrm{\tiny{per}}$ for $m > k + \tfrac{1}{2}$, $k \in \N$, and we can characterize the spaces $\Hmper$ as
\[
\Hmper([0,T],\C) = \{ u \in H^m([0,T];\C) \, : \, \partial_x^j u (0) = \partial_x^j u(T), \; j=0,1,\ldots, m-1\}.
\]
Their inner product and norm are given by
\[
\< u,v \>_{\Hmper} = \sum_{j=0}^m \< \partial_x^j u, \partial_x^j v\>_{\Ldper}, \quad \|u\|_{\Hmper}^2 = \< u,u\>_{\Hmper}.
\]
We use the standard notation in asymptotic analysis (cf. \cite{Erde56,Mill06}), in which the symbol ``$\sim$" means  ``behaves asymptotically like" as $x \to x_*$; more precisely, $f \sim g$ as $x \to x_*$ if $f - g = o(|g|)$ as $x \to x_*$ (or equivalently, $f/g \to 1$ as $x \to x_*$ if both functions are positive).


\section{Existence of bounded periodic wavetrains}
\label{secexi}

This section contains the proofs of Theorems \ref{thmexbded} and \ref{thmexlarge}. We distinguish between two types of bounded periodic waves: those with small-amplitude and bounded period that emerge from a Hopf bifurcation from the origin, and those with finite amplitude and large fundamental period that bifurcate from a homoclinic loop in the neighborhood of a critical velocity. In both cases the speed $c \in \R$ plays the role of the bifurcation parameter.

Consider a traveling wave solution to equations \eqref{vbl} having the form 
\begin{equation}
\label{tws}
u(x,t) = \varphi(x-ct), 
\end{equation}
where $\varphi : \R \to \R$ is the profile function of the wave and $c \in \R$ is the speed of propagation. Let us denote the Galilean variable of translation as $z = x - ct$. A bounded spatially periodic traveling wave is a solution of the form \eqref{tws} for which the profile function is a periodic function of its argument with fundamental period $T > 0$, satisfying
\[
\varphi(z + T) = \varphi(z), \qquad \text{for all} \;\; z \in \R,
\]
and
\[
|\varphi(z)|, |\varphi'(z)| \leq C, \qquad \text{for all} \;\; z \in \R, \;\; \text{some } C>0.
\]

Substitution of \eqref{tws} into \eqref{vbl} yields the following ODE for the profile function,
\begin{equation}
\label{profileq}
-c \varphi' + f'(\varphi)\varphi' = \varphi'' + g(\varphi).
\end{equation}

In order to analyze the existence of periodic solutions to \eqref{profileq} let us denote $U := \varphi(z)$, $V := \varphi'(z)$, $' = d/dz$ and write \eqref{profileq} as a the first order system in the plane
\begin{equation}
\label{firstos}
\begin{aligned}
U' &= F(U,V,c) := V,&\\
V' &= G(U,V,c) := -cV + f'(U) V - g(U).&
\end{aligned}
\end{equation}

Notice that, from assumptions \eqref{H1} and \eqref{H2} we have $F, G \in C^3(\R^3)$ and that for each parameter value $c\in \R$ system \eqref{firstos} has two equilibria, $P_0 = (0,0)$ and $P_1 = 
(1,0)$, in the $(U,V)$-phase plane. Let us denote the Jacobian with respect to $(U,V)$ of the right hand side of \eqref{firstos} as
\[
A(U,V) := 
\begin{pmatrix} F_U & F_V \\ G_U & G_V \end{pmatrix} 
= \begin{pmatrix} 0 & 1 \\ f''(U) V - g'(U) & -c + f'(U) \end{pmatrix}.
\]
Let $A_0 = A(0,0)$ and $A_1 = A(1,0)$ denote the linearizations of \eqref{firstos} evaluated at the two equilibria, $P_0$ and $P_1$, respectively, so that
\[
 A_0 = \begin{pmatrix}
           0 & 1 \\ -g'(0) & -c + f'(0)
          \end{pmatrix}, \quad \text{and} \quad A_1 = \begin{pmatrix}
           0 & 1 \\ -g'(1) & -c + f'(1)
          \end{pmatrix}.
\]

Note that the eigenvalues of $A_1$ are
\begin{equation}
\label{evaluesA1}
\lambda_1^\pm(c) = \frac{1}{2}\big(f'(1) - c\big) \pm \frac{1}{2} \sqrt{(f'(1)-c)^2 - 4g'(1)},
\end{equation}
and therefore, in view of \eqref{H2}, $g'(1) < 0$ and the equilibrium point $P_1 = (1,0)$ is a hyperbolic saddle for system \eqref{firstos} for each value of $c \in \R$. The eigenvalues of $A_0$ are
\begin{equation}
\label{evaluesA0}
\lambda_0^\pm(c) = \frac{1}{2}\big(f'(0) - c\big) \pm \frac{1}{2} \Big( (f'(0)-c)^2 - 4g'(0) \Big)^{1/2},
\end{equation}
and hence the origin $P_0 = (0,0)$ is 
a node, a focus or a center, depending on the value of $c \in \R$. In the sequel we shall vary $c$ as a bifurcation parameter to establish the conditions under which periodic orbits for \eqref{firstos} do emerge.

\subsection{Small-amplitude periodic waves}

The existence of small-amplitude periodic traveling waves for viscous balance laws of the form \eqref{vbl} is a direct consequence of Andronov-Hopf's bifurcation theorem. Such periodic orbits bifurcate from a local change 
of stability of the origin when the speed $c$ crosses a critical value $c_0$.

\begin{theorem}[Andronov-Hopf]
\label{teoHopf}
Consider the planar system
\begin{equation}
\label{Hopfsyst}
\left\{ \begin{array}{ll}
U' = F(U,V,\mu)\\
V' = G(U,V,\mu)\end{array} \right.  
\end{equation}

\noindent where $F$ and $G$ are functions of class $C^3$ and $\mu \in \R$ is a bifurcation parameter. Suppose $(U,V) = (U_0,V_0)$ is an equilibrium point of system \eqref{Hopfsyst}, which may depend on $\mu$. Let the eigenvalues  of the linearized system around $(U_0,V_0)$ be given by 
\[
 \lambda^\pm (\mu) = \alpha(\mu) \pm i \beta(\mu).
\]
Let us assume that for a certain value $\mu = \mu_0$ the following conditions are satisfied:
\begin{itemize}
 \item[(a)] \emph{(non-hyperbolicity condition)} $\alpha(\mu_0) = 0$, $\beta(\mu_0) = \omega_0 \neq 0$, and 
 \begin{equation}
  \label{signoeq}
  \mathrm{sgn} (\omega_0) = \mathrm{sgn} ((\partial G/\partial U)(U_0,V_0,\mu_0) ).
 \end{equation}
 \item[(b)] \emph{(transversality condition)}
 \[
  \frac{d \alpha}{d \mu}(\mu_0) = d_0 \neq 0.
 \]
\item[(c)] \emph{(genericity condition)} $a_0 \neq 0$, where $a_0$ is the \emph{first Lyapunov exponent},
\begin{equation}
\label{formlyapexp}
\begin{aligned}
 a_0 &:= \frac{1}{16} \big( F_{UUU} + F_{UVV} + G_{UUV} + G_{VVV}\big) +\\ &\; \; + \frac{1}{16 \omega_0} 
\big(F_{UV} ( 
F_{UU} + F_{VV}) - G_{UV}(G_{UU} + G_{VV}) - F_{UU}G_{UU} + F_{VV}G_{VV}\big),
\end{aligned}
\end{equation}
where the partial derivatives of $F$ and $G$ are evaluated at $(U_0,V_0,\mu_0)$.
\end{itemize}
Then there exists $\epsilon > 0$ such that a unique curve of closed periodic orbit solutions bifurcates from 
the equilibrium point into the region 
$\mu \in (\mu_0, \mu_0 + \epsilon)$ if $a_0 d_0 < 0$, or into the region $\mu \in (\mu_0-\epsilon,\mu_0)$ if 
$a_0 d_0 > 0$. The fixed point is stable 
for $\mu > \mu_0$ (respectively, $\mu < \mu_0$) if $d_0 < 0$ (respectively, $d_0 > 0$). Consequently, the 
periodic orbits are unstable (respectively, stable) if the equilibrium point is stable (respectively, 
unstable) on the region where the periodic orbits exist. Moreover, the amplitude of the periodic orbits grows 
like $\sqrt{|\mu-\mu_0|}$ and their fundamental periods behave like
\[
 T(\mu) = \frac{2 \pi}{|\omega_0|} + O(|\mu - \mu_0|),
\]
as $\mu \to \mu_0$. The bifurcation is called \emph{supercritical} if the bifurcating periodic orbits 
are stable and \emph{subcritical} if they are unstable.
\end{theorem}

\begin{remark}
Theorem \ref{teoHopf} is the classical result first proved by Andronov 
\cite{Andro29} in the plane and extended to arbitrary finite dimensions by Hopf \cite{Hop42}. The reason to include its precise 
statement here is that most of its versions in the standard literature (see, for example, 
\cite{GuHo83,HaKo91,Kuz982e}) are expressed in terms of the normal form of a generic system \eqref{Hopfsyst}, for which
the sign condition \eqref{signoeq} is usually implicitly assumed. But for a system not necessarily written in normal form, the sign condition has 
to be verified in order to determine on which side of the bifurcation parameter value periodic orbits do
emerge. 
The formula for the first Lyapunov exponent \eqref{formlyapexp} is well-known and can be found in \cite{GuHo83}, p. 152 (see also \cite{HaKo91}). The expression for the period can be found in the version of the same theorem by Marsden and McCracken \cite{MaMcC76} (see Theorem 3.1, p. 65).
\end{remark}

\begin{remark}
\label{remstability}
The notion of a stable periodic orbit in the statement of Andronov-Hopf's theorem refers to the standard concept from dynamical systems theory: the orbit is stable as a solution to system \eqref{firstos} for a specific (and constant) value of $c$ if any other nearby solution (to the system with the same $c$) tends to the orbit under consideration. This notion is completely unrelated to the concept of \emph{spectrally stable periodic wave} from Definition \ref{defspectstab} which is motivated from the dynamical stability of the traveling wave solution as a solution to the evolution PDE. For example, the small amplitude waves for the modified generalized Burgers-Fisher equation (see equation \eqref{mBF} in section \S \ref{secmBF} below) are stable as orbit solutions to the first order system \eqref{firstos}, but spectrally unstable as solutions to the viscous balance law according to Theorem \ref{theosmalli}.
\end{remark}

\begin{theorem}[existence of small-amplitude periodic orbits]
\label{theosmallw}
Under the assumptions \eqref{H1}, \eqref{H2} and \eqref{H4}, there exist, a critical speed $c_0 = f'(0) \in \R$ and $\epsilon_0 > 0$ sufficiently small, such that 
a unique family of closed periodic orbit solutions $(\overline{U},\overline{V})(z)$ for system \eqref{firstos} bifurcates from 
the origin $P_0 = (0,0)$. The family is parametrized by speed values
$c \in (c_0, c_0 + \epsilon_0)$ if $\overline{a}_0 > 0$, or by $c \in (c_0-\epsilon_0,c_0)$ if 
$\overline{a}_0 < 0$. Moreover, the amplitude of the periodic orbits and their fundamental periods behave like
\[
|\overline{U}|, |\overline{V}| = O(\sqrt{|c-c_0|}),
\]
and
\[
 T(c) = \frac{2 \pi}{\sqrt{g'(0)}} + O(|c - c_0|),
\]
respectively, as $c \to c_0$.
\end{theorem}
\begin{proof}
It is a direct consequence of Andronov-Hopf's bifurcation theorem upon verification of conditions (a) thru (c). Let us first write the eigenvalues \eqref{evaluesA0} of $A_0$ as
\[
\lambda_0^\pm = \alpha(c) \mp i \beta(c),
\]
where 
\[
\alpha(c) := \frac{1}{2}(f'(0) - c), \qquad \beta(c) := - \frac{1}{2} \sqrt{4g'(0) - (f'(0) -c)^2},
\]
defined for $c \sim f'(0)$. Note that, under hypothesis \eqref{H2}, $\beta(c) \in \R$ for $c \sim f'(0)$. Hence we have a bifurcation value for the speed given by $c_0 = f'(0)$ for which $\alpha(c_0) = 0$ and the origin is a center for system \eqref{firstos} with eigenvalues
\[
\lambda_0^+(c_0) = - i \sqrt{g'(0)}, \qquad \lambda_0^-(c_0) =  i \sqrt{g'(0)}.
\]
Notice that $\omega_0 := \beta(c_0) = - \sqrt{g'(0)} \neq 0$ and, since $G_U = f''(U)V - g'(U)$, we obtain 
\[
(G_U)|_{(0,0,c_0)} = - g'(0),
\]
yielding $\sgn(\omega_0) = \sgn ((G_U)|_{(0,0,c_0)}) = -1$, that is, the non-hyperbolicity condition (a). Likewise, the transversality condition (b) is satisfied inasmuch as
\[
\frac{d\alpha}{dc}(c_0) = - \frac{1}{2} =: d_0 \neq 0.
\]
Finally, to compute the first Lyapunov exponent, notice that $F(U,V,c) = V$ and hence all second derivatives of $F$ are zero. The Lyapunov exponent \eqref{formlyapexp} then reduces to
\[
a_0 = \frac{1}{16} \big( G_{UUV} + G_{VVV} \big)|_{(0,0,c_0)} - \frac{1}{16 \omega_0} \big( G_{UV}(G_{UU} + G_{VV}) \big)|_{(0,0,c_0)}.
\]
Upon calculation and evaluation of the derivatives,
\[
\begin{aligned}
G_{UV}|_{(0,0,c_0)} = f''(0), \quad &G_{UU}|_{(0,0,c_0)} = - g''(0), &G_{VV}|_{(0,0,c_0)} = 0, \\&G_{UUV}|_{(0,0,c_0)} = f'''(0), &G_{VVV}|_{(0,0,c_0)} = 0 &,
\end{aligned}
\]
we arrive at
\[
a_0 = \frac{1}{16} \Big( f'''(0) - \frac{f''(0)g''(0)}{\sqrt{g'(0)}} \Big) = \frac{\overline{a}_0}{16} \neq 0,
\]
in view of \eqref{H4}. This verifies the genericity condition (c). Since $d_0 < 0$ and $\sgn(a_0) = \sgn(\overline{a}_0)$ we obtain the result. 
\end{proof}

\begin{proof}[Proof of Theorem \ref{thmexbded}]
In view of Theorem \ref{theosmallw}, there exists a family of small amplitude periodic orbits parametrized by $\ep := |c-c_0|$ such that, for all $0 < \ep < \ep_0$ there exists a unique periodic orbit, which we denote as $(\bU^\ep,\bV^\ep) (z) =: (\varphi^\ep, (\varphi^\ep)')(z)$, $z \in \R$, solution to system \eqref{firstos} with speed $c(\ep) = c_0 - \ep$ if $\overline{a}_0 < 0$ or $c(\ep) = c_0 + \ep$ if $\overline{a}_0 > 0$, with fundamental period
\[
T_\ep = T_0 + O(\ep) = \frac{2\pi}{\sqrt{g'(0)}} + O(\ep),
\]
and such that $(\varphi^\ep, (\varphi^\ep)') \to P_0 = (0,0)$ as $\ep \to 0^+$ with amplitudes
\[
|\varphi^\ep(z)|, |(\varphi^\ep)'(z)| \leq C \sqrt{\ep},
\]
for some uniform constant $C > 0$.

Each of these orbits is associated to a periodic traveling wave solution to the viscous balance law \eqref{vbl} of the form $u^\ep(x,t) = \varphi^\ep( x - c(\ep) t)$, traveling with speed $c(\ep) = c_0 \pm \ep \gtrless c_0$, depending on the sign of $\overline{a}_0$. Moreover, from standard ODE theory and from the regularity assumptions on $f$ and $g$, it can be easily verified that the orbit is a $C^3$ function of $z \in \R$ and of the bifurcation parameter $c$. The theorem is proved.
\end{proof}

\subsection{Existence of an homoclinic loop: the traveling pulse}

Prior to the existence analysis for large-period traveling waves, we first need to establish the existence of a homoclinic orbit for system \eqref{firstos}. For that purpose we apply Melnikov's integral method. Let us consider the auxiliary system
\begin{equation}
\label{unpert}
\begin{aligned}
U' &= V,&\\
V' &=  -cV + af'(U) V - g(U),
\end{aligned}
\end{equation}
where $a \in \R$ is an auxiliary parameter, and write it as a \emph{near-Hamiltonian system} of the form (cf. \cite{HaYu12}),
\begin{equation}
\label{unpertH}
\begin{aligned}
U' &= \partial_V H + \vep R(U,V,\vep,\mu),\\
V' &=  - \partial_U H + \vep Q(U,V,\vep,\mu),
\end{aligned}
\end{equation}
where
\begin{equation}
\label{Hamil}
H(U,V) := \frac{1}{2} V^2 + \int_0^U g(s) \, ds,
\end{equation}
is the Hamiltonian, and
\[
\begin{aligned}
R(U,V,\vep,\mu) &\equiv 0,\\
Q(U,V,\vep,\mu) &:= \mu_1 f'(U)V - \mu_2 V,\\
\mu = (\mu_1,\mu_2) \in \R^2, & \quad a =: \vep \mu_1, \; c =: \vep \mu_2.
\end{aligned}
\]

Here $\mu = (\mu_1,\mu_2) \in \R^2$ is a vector parameter and $0 < \vep \ll 1$ is small. The associated Hamiltonian (unperturbed) system reads
\begin{equation}
\label{sistH}
\begin{aligned}
U' &= \partial_V H = V, \\
V' &=  - \partial_U H = - g(U).
\end{aligned}
\end{equation}

\begin{remark}
\label{remsamesad}
First, it is to be observed that $P_0 = (0,0)$ and $P_1 = (1,0)$ are equilibrium points for both the Hamiltonian system \eqref{sistH} and the perturbed system \eqref{unpertH}. If we linearize system \eqref{sistH} around the origin, the corresponding Jacobian reads
\[
\widetilde{A}(0,0) = \begin{pmatrix} 0 & - g'(0) \\ 1 & 0\end{pmatrix},
\]
with eigenvalues $\lambda = \pm i \sqrt{g'(0)}$ and henceforth $P_0 = (0,0)$ is a center for system \eqref{sistH}. Likewise, the linearization around $P_1 =(1,0)$ yields
\[
\widetilde{A}(0,0) = \begin{pmatrix} 0 & - g'(1) \\ 1 & 0\end{pmatrix},
\]
with eigenvalues $\lambda = \pm \sqrt{-g'(1)} \in \R$, and hence $P_1 = (1,0)$ is a hyperbolic saddle for the Hamiltonian system \eqref{sistH}.  The stable and unstable eigendirections at $P_1$ are given by
\[
r^- = \begin{pmatrix} - \sqrt{-g'(1)} \\ 1 \end{pmatrix} \quad \text{and} \quad r^+ = \begin{pmatrix} \sqrt{-g'(1)} \\ 1\end{pmatrix},
\]
respectively. 

On the other hand, notice that $P_1 = (1,0)$ is also a hyperbolic saddle for the perturbed system \eqref{unpertH} for any parameter values $a$ and $c$ (equivalently, for any $\vep$, $\mu_1$ and $\mu_2$). Indeed, the linearization of \eqref{unpertH} around $P_1 = (1,0)$ is
\[
\widetilde{A}^\vep(1,0) = \begin{pmatrix} 0 & - g'(1) \\ 1 & a f'(1) -c\end{pmatrix},
\]
having eigenvalues
\[
\lambda^{\vep}_\pm = \frac{1}{2} \left( a f'(1) -c \pm \sqrt{(a f'(1) -c)^2 - 4 g'(1)} \right),
\]
and in view of \eqref{H2}, we have $\lambda^{\vep}_- < 0 < \lambda^{\vep}_+$ for all values of $a$ and $c$, yielding a hyperbolic saddle, independently of the parameter values. In the same fashion, if we linearize \eqref{unpertH} around $P_0 = (0,0)$ the resulting Jacobian is
\[
\widetilde{A}^\vep(0,0) = \begin{pmatrix} 0 & - g'(0) \\ 1 & a f'(0) -c\end{pmatrix},
\]
with associated eigenvalues
\[
\lambda^{\vep}_\pm = \frac{1}{2} \left( a f'(0) -c \pm \sqrt{(a f'(0) -c)^2 - 4 g'(0)} \right).
\]
Thus, $P_0 = (0,0)$ is a center whenever $a = 1$.
\end{remark}

\smallskip

The energy levels at $P_0 = (0,0)$ and $P_1 = (1,0)$ as equilibria of the Hamiltonian system \eqref{sistH} are 
\begin{equation}
\label{defbeta}
\beta := H(1,0) = \int_0^1 g(s) \, ds > 0,
\end{equation}
and $H(0,0) = 0$, respectively. Now let us make the following observations about the Hamiltonian system \eqref{sistH}:

\smallskip

\noindent (a) First, notice that the set
\begin{equation}
\label{defGammab}
\Gamma^{\beta} := \{ (U,V) \in \R^2 \, : \, H(U,V) = \beta \},
\end{equation}
is a homoclinic loop for the Hamiltonian system \eqref{sistH} joining the hyperbolic saddle $P_1 = (1,0)$ with itself. The homoclinic orbit is given explicitly by the graph
\begin{equation}
\label{homoloop}
V(U) = \pm \bV^{\beta}(U) := \pm \sqrt{2 \Big( \beta - \int_0^U g(s) \, ds \Big) } = \pm \gamma(U), \quad U \in (u_*,1),
\end{equation}
where the function $\gamma = \gamma(\cdot)$ is defined in \eqref{defPsi}. See Figure \ref{FigOmegashi}.

\smallskip

\noindent (b) There exists a family of periodic orbits for system \eqref{sistH},
\begin{equation}
\label{defGammah}
\Gamma^h := \{ (U,V) \in \R^2 \, : \, H(U,V) = h \}, \qquad h \in (0,\beta),
\end{equation}
such that
\begin{itemize}
\item[(i)] $\Gamma^h \to P_0 = (0,0)$ as $h \to 0^+$, and
\item[(ii)] $\Gamma^h \to \Gamma^\beta$ as $h \to \beta^-$.
\end{itemize}
Indeed, if we define 
\[
\tG(u) = \int_0^u g(s) \, ds,
\]
then under assumptions \eqref{H2} and \eqref{H3} it is clear that $\tG(0) = 0$, $\tG(1) = \tG(u_*) = \beta$ and $\tG'(u) > 0$ if $u \in (0,1)$, $\tG'(u) < 0$ if $u \in (u_*,0)$. Therefore, for each energy level $h  \in (0,\beta)$ there exist unique values $u_1(h) \in (u_*,0)$ and  $u_2(h) \in(0, 1)$ such that
\[
\tG(u_1(h)) = \tG(u_2(h)) = h,
\]
and the periodic orbits are given explicitly by the graphs
\begin{equation}
\label{pergraphs}
V(U) = \pm \bV^{h}(U) := \pm \sqrt{2 \Big( h - \int_0^U g(s) \, ds \Big) }, \quad U \in (u_1(h), u_2(h)), \;\; h \in (0,\beta).
\end{equation}
It is also clear that $\bV^{h} \to \bV^{\beta}$ as $h \to \beta^-$ and that the orbits shrink to the origin as $h \to 0^+$. See Figure \ref{FigOmegashi}.

\smallskip

\noindent (c) If $T = T(h)$ denotes the fundamental period of the periodic orbit $\Gamma^h$, $h \in (0,\beta)$, then $T(h) \to \infty$ as $h \to \beta^-$. The period can be computed explicitly by the elliptic integral
\[
T(h) = \sqrt{2} \int_{u_1(h)}^{u_2(h)} \frac{dy}{\sqrt{h - \int_0^y g(s) \, ds}}, \qquad h \in (0,\beta).
\]
From standard properties of Hamiltonian systems (see, e.g., \cite{Scha80,Sba12}), $0 < T(h) < \infty$ for each $h \in (0,\beta)$ and $T(h) \to \infty$ as $h \to \beta^-$, which is the infinite period of the homoclinic loop $\Gamma^\beta$.


\begin{figure}[t]
\begin{center}
\includegraphics[scale=.6, clip=true]{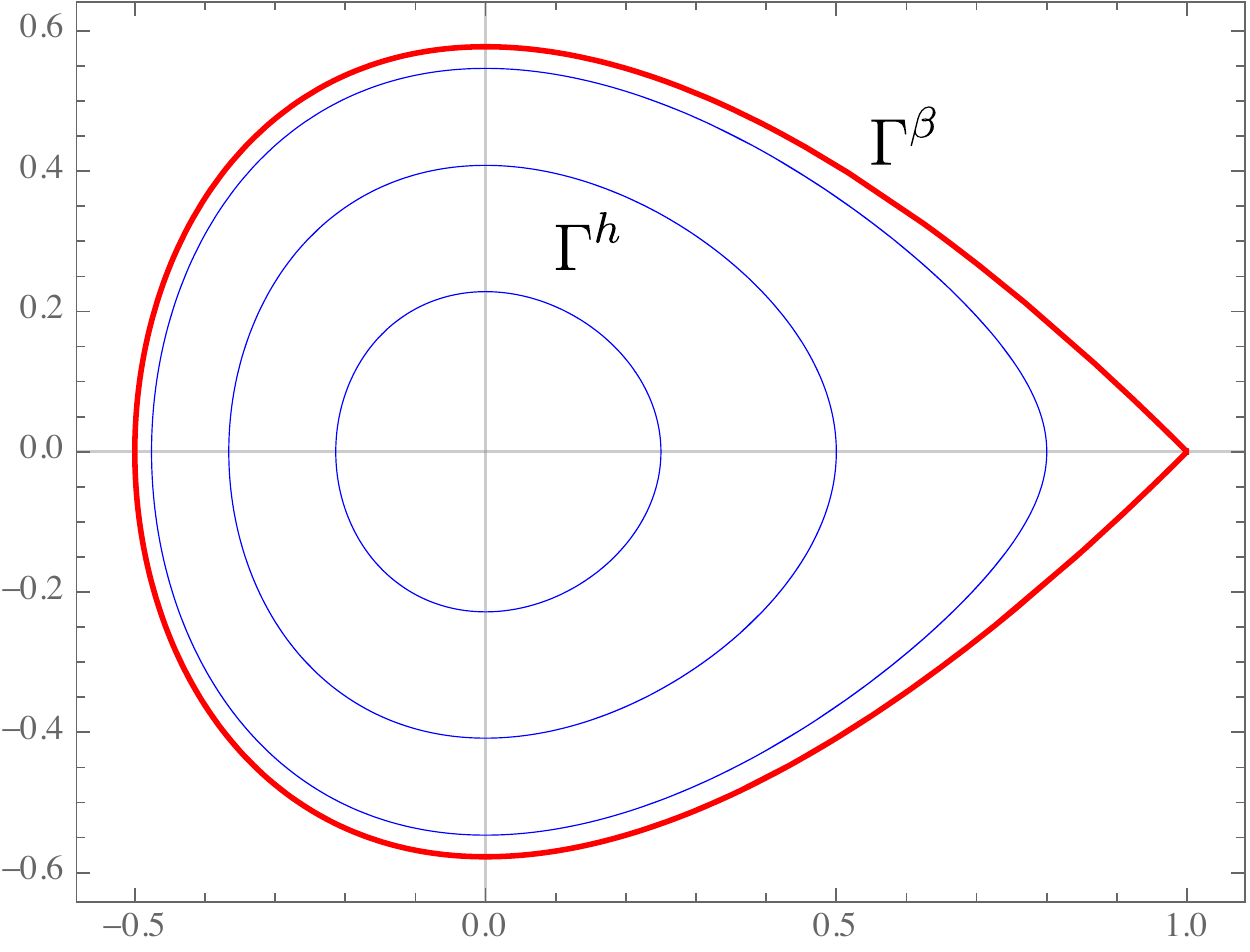}
\end{center}
\caption{\small{Homoclinic loop $\Gamma^\beta$ (in red) in the $(U,V)$-phase plane and periodic orbits $\Gamma^h$, $h \in (0,\beta)$ (in blue), for the Hamiltonian system \eqref{sistH} in the case of the reaction function of logistic type, $g(u) = u(1-u)$ (color online).}}\label{FigOmegashi}
\end{figure}

\smallskip

In view of observations (a) thru (c) above, let us now define the open sets
\begin{equation}
\label{defOmegah}
\begin{aligned}
\Omega_h :=  \text{int} \, \Gamma^h &= \{ (U,V) \in \R^2 \, : \, 0 < H(U,V) < h \} \\ 
&= \left\{ (U,V) \in \R^2 \, : \, u_1(h) < U < u_2(h)), \; - \bV^h(U) < V < \bV^h(U) \right\}, \\
\end{aligned}
\end{equation}
for each $h \in (0,\beta)$ (see Figure \ref{FigOmegashi}). In the same fashion, let us define
\begin{equation}
\label{defOmegab}
\begin{aligned}
\Omega_{\beta} := \text{int} \, \Gamma^{\beta} &= \{ (U,V) \in \R^2 \, : \, 0 < H(U,V) < \beta \} \\ 
&= \left\{ (U,V) \in \R^2 \, : \, u_* < U < 1, \; - \gamma(U) < V < \gamma(U) \right\}.
\end{aligned}
\end{equation}
Then one defines the associated \emph{Melnikov integrals} (cf. \cite{HaYu12}, p. 268) as
\begin{equation}
\label{defMelniInt}
\tM(h, \mu) := \int_{\Omega_h} \left( \partial_U R + \partial_V Q \right) \, dU \, dV.
\end{equation}
They satisfy (see \cite{HaYu12}, Lemmata 6.2 and 6.8):
\begin{itemize}
\item $\tM \in C^\infty$ for $|\vep| + |h - h_0| \ll 1$ small for any $h_0 \in (0,\beta)$, all $\mu \in \R^2$.
\item The derivative with respect to $h$ is determined by
\begin{equation}
\label{derMt}
\partial_h \tM(h,\mu) = \oint_{\Gamma^h} \left( \partial_U R + \partial_V Q \right) \, d\sigma_h, \qquad h \in (0,\beta), \;\, \mu \in \R^2.
\end{equation}
\end{itemize}

Moreover, we define the Melnikov integrals precisely at $h = \beta$ by
\begin{equation}
\label{Melnii}
\begin{aligned}
M(\mu) &:= \tM(\beta, \mu) = \int_{\Omega_\beta} \left( \partial_U R + \partial_V Q \right) \, dU \, dV,\\
M_1(\mu) &:= \partial_h \tM(\beta,\mu) = \oint_{\Gamma^{\beta}} \left( \partial_U R + \partial_V Q \right) \, d\sigma_\beta.
\end{aligned}
\end{equation}

Then we can apply the following classical theorem due to Melinkov \cite{Melni63} (see also \cite{CHM-P80}), which establishes the conditions under which the perturbed system \eqref{unpertH} underlies a homoclinic loop emerging from the homoclinic orbit for the Hamiltonian system \eqref{sistH} (see, e.g., Theorem 6.8, p. 466, in \cite{Chi99}, Theorem 6.4, p. 266 in \cite{HaYu12}, as well as Lemma 4.5.1 and Theorem 4.5.4 in \cite{GuHo83}.)

\begin{theorem}[Melnikov's method for perturbed homoclinic orbits \cite{Melni63}]
\label{theoMelni}
Suppose that $P_1$ is a hyperbolic saddle equilibrium point for the unperturbed Hamiltonian system \eqref{sistH} possessing a homoclinic loop $\Gamma^\beta$. If $\vep > 0$ is sufficiently small then the perturbed system \eqref{unpertH} has a unique hyperbolic equilibrium point $P_1(\vep) = P_1 + O(\vep)$. Moreover, if $M(\mu_0) = 0$ and $M_1(\mu_0) \neq 0$ (that is, if the Melnikov integral has a simple zero at $\mu = \mu_0$ at the energy level $h = \beta$ on the homoclinic loop) then the perturbed system \eqref{unpertH} with $\mu = \mu_0$ has a unique hyperbolic homoclinic loop $\Gamma^\beta_\vep$ for each $\vep > 0$ sufficiently small, relative to the stable and unstable manifolds of the hyperbolic equilibrium point $P_1(\vep)$. If, on the other hand, $M(\mu)$ has no zeroes and $|\vep| \neq 0$ is small, then the stable and unstable manifolds of $P_1(\vep)$ do not intersect.
\end{theorem}


As a consequence we have the following
\begin{theorem}
\label{theoexisthomo}
Under assumptions \eqref{H1}, \eqref{H2}, \eqref{H3} and \eqref{H5}, system \eqref{unpert} has a unique homoclinic orbit joining the hyperbolic saddle point $P_1 = (1,0)$ with itself for the parameter values $a = 1$ and $c = c_1 = I_1/I_0$.
\end{theorem}
\begin{proof}
Follows upon application of Melnikov's method. First notice that in view that $R(U,V) = 0$ and $Q(U,V) = \mu_1 f'(U)V - \mu_2 V$ we can evaluate the Melnikov integrals \eqref{Melnii}. Since $\partial_U R = 0$ and $\partial_V Q = \mu_1 f'(U) - \mu_2$ we then have
\[
\begin{aligned}
M(\mu) &= \int_{\Omega_\beta} ( \mu_1 f'(U) - \mu_2 ) \, dV \, dU\\
&= \int_{u_*}^1 \int_{-\gamma(U)}^{\gamma(U)} ( \mu_1 f'(U) - \mu_2 ) \, dV \, dU\\
&= 2 \Big( \mu_1 \int_{u_*}^1 f'(U) \gamma(U) \, dU - \mu_2  \int_{u_*}^1 \gamma(U) \, dU \Big)\\
&= 2 \big( \mu_1 I_1 - \mu_2 I_0 \big).
\end{aligned}
\]
Hence, $M(\mu) = 0$ only when
\begin{equation}
\label{mu1mu2}
\mu_2 = \Big( \frac{I_1}{I_0} \Big) \mu_1.
\end{equation}
Now let us evaluate $M_1$ at any $\mu \in \R^2$ satisfying \eqref{mu1mu2}. From \eqref{Melnii} we obtain
\[
\begin{aligned}
M_1(\mu) &= \oint_{\Gamma^\beta} (\mu_1 f'(U) - \mu_2) \, d\sigma_\beta \\
&= \mu_1 \oint_{\Gamma^\beta} f'(U) \, d\sigma_\beta - \mu_2 \oint_{\Gamma^\beta} \, d\sigma_\beta \\
&= 2 \mu_1 \int_{u_*}^1 f'(U) \sqrt{1 + \gamma'(U)^2} \, dU - \mu_2 |\partial \Omega_\beta| \\
&= \mu_1 J - \mu_2 L\\
&= \mu_1 \left( J - \Big( \frac{I_1}{I_0} \Big)  L \right) \neq 0
\end{aligned}
\]
if $\mu_1 \neq 0$ and in view of \eqref{H5}. Therefore, choose $\vep > 0$ sufficiently small and set
\[
\mu_1 := \frac{1}{\vep} > 0, \qquad \mu_2 = \Big( \frac{I_1}{I_0} \Big) \mu_1,
\]
so that 
\[
c = \vep \mu_2 = \frac{I_1}{I_0}, \qquad a = \vep \mu_1 = 1.
\]
Hence $\mu_0 := (\mu_1,\mu_2) = (1/\vep, (I_0/\vep I_1)) \in \R^2$ is the bifurcation value for which the Melnikov integral has a simple zero. In this case the critical value for the speed is
\[
c = c_1 = \frac{I_1}{I_0}.
\]

Upon application of Theorem \ref{theoMelni}, if $\vep$ is sufficiently small then the perturbed system \eqref{unpertH} has a unique hyperbolic point $P_1(\vep) = P_1 + O(\vep)$. But since $P_1 = (1,0)$ is a hyperbolic saddle for system \eqref{unpertH} for any parameter values we observe that
\[
P_1(\vep) \equiv P_1 = (1,0), \qquad \text{for any } \; 0 < \vep \ll 1.
\]
Now, since $M (\mu_0) = 0$ and $M_1(\mu_0) \neq 0$ we conclude that the perturbed system has a unique homoclinic loop $\Gamma^\beta_\vep$ relative to the stable and unstable manifolds at $P_1$, for parameter values $a = 1$ and $c = c_1$. This yields the result.
\end{proof}

\begin{corollary}[existence of a traveling pulse]
\label{corpulse}
The system \eqref{firstos} has a homoclinic loop for the speed value $c = c_1 = I_1/I_0$, which we denote as
\[
\Gamma_0 := \{ (\psi, \psi')(z) \, : \, z \in \R\},
\]
with $\psi \in C^3(\R)$ and such that $(\psi, \psi')(z) \to (1,0)$ as $z \to \pm \infty$. Moreover, the convergence is exponential, that is, there exist constants $C, \kappa > 0$ such that
\begin{equation}
\label{expdecay}
|\psi(z) - 1|, |\psi'(z)| \leq C e^{-\kappa |z|}, \qquad \text{as } \; |z| \to \infty.
\end{equation}
This homoclinic orbit is associated to a traveling pulse solution to the viscous balance law \eqref{vbl} of the form $u(x,t) = \psi(x - c_1 t)$ and traveling with speed $c = c_1$.
\end{corollary}
\begin{proof}
Let us denote the homoclinic orbit fromTheorem \ref{theoexisthomo} as $(\psi,\psi')(z)$. This orbit is a solution to system \eqref{firstos} with speed value $c = c_1$. Since $F,G \in C^3$ it is clear that $\psi \in C^2$. Upon differentiation of \eqref{firstos} we obtain
\begin{equation}
\label{psi3}
\psi''' = (-c_1 + f'(\psi)) \psi'' + f''(\psi)(\psi')^2 - g'(\psi) \psi'.
\end{equation}
Then by a boostrapping argument we conclude that $\psi \in C^3(\R)$. The exponential decay follows from standard ODE estimates and the fact that $P_1 = (1,0)$ is a hyperbolic saddle for system \eqref{firstos} for the speed value $c = c_1$. More precisely, the stable and unstable eigenvalues are given by
\begin{equation}
\label{evaluessaddle}
\begin{aligned}
\lambda_1(c_1) &= \frac{1}{2}(f'(1) - c_1) - \frac{1}{2} \sqrt{(f'(1) - c_1)^2 - 4 g'(1)} < 0,\\
\lambda_2(c_1) &= \frac{1}{2}(f'(1) - c_1) + \frac{1}{2} \sqrt{(f'(1) - c_1)^2 - 4 g'(1)} > 0,
\end{aligned}
\end{equation}
so that
\[
\begin{aligned}
|\psi(z) - 1|, |\psi'(z)| &\leq C e^{\lambda_1(c_1)z}, \quad \text{as } \; z \to \infty,\\
|\psi(z) - 1|, |\psi'(z)| &\leq C e^{\lambda_2(c_1)z}, \quad \text{as } \; z \to -\infty.
\end{aligned}
\]
Thus, we can take
\begin{equation}
\label{defkappa}
\kappa := \min \{ \lambda_2(c_1), -\lambda_1(c_1) \} > 0,
\end{equation}
to obtain \eqref{expdecay}, as claimed.
\end{proof}

\subsection{Periodic wavetrains with large fundamental period}
\label{seclongp}

The following theorem is the classical result by Andronov and Leontovich \cite{AnLeo37,ALGM73}, which guarantees the bifurcation of a limit cycle from a homoclinic loop of a saddle with non-zero saddle quantity (see, for example, Theorem 13.4 in \cite{SSTC01}, p. 700).

\begin{theorem}[Andronov-Leontovich \cite{AnLeo37,ALGM73}]
\label{themAL}
Consider a planar system of the form
\[
\begin{aligned}
U' &= F(U,V,\mu)\\
V' &= G(U,V,\mu)
\end{aligned}
\]
where $F$ and $G$ are smooth functions and $\mu \in \R$. Suppose that $(U_0,V_0)$ is a hyperbolic saddle point for all values $\mu$ near $\mu_0$; that at $\mu = \mu_0$ the associated eigenvalues are $\lambda_1(\mu_0) < 0 < \lambda_2(\mu_0)$; and that the system possesses a homoclinic orbit $\Gamma_0$ at the saddle. Let us define the \emph{saddle quantity} as,
\[
\Sigma_0 := \lambda_1(\mu_0) + \lambda_2(\mu_0),
\]
and suppose that $\Sigma_0 \neq 0$. Then:
\begin{itemize}
\item[(a)] If $\Sigma_0 < 0$ then for sufficiently small $\mu - \mu_0 > 0$ there exists a unique stable limit cycle $\Gamma(\mu)$ bifurcating from $\Gamma_0$ which as $\mu \to \mu_0^+$ gets closer to the homoclinic loop at $\mu = \mu_0$. When $\mu < \mu_0$ there are no limit cycles.
\item[(b)] If $\Sigma_0 > 0$ then for sufficiently small $\mu - \mu_0 < 0$ there exists a unique unstable limit cycle $\Gamma(\mu)$ bifurcating from $\Gamma_0$ which as $\mu \to \mu_0^-$ becomes the homoclinic loop at $\mu = \mu_0$. When $\mu > \mu_0$  there are no limit cycles.
\end{itemize}
\end{theorem}

The existence of large-period, bounded periodic orbits is a consequence of both the existence of a homoclinic loop and Andronov-Leontovich's theorem.

\begin{theorem}[existence of large period orbits]
\label{theoexistlp}
Under assumptions \eqref{H1}, \eqref{H2}, \eqref{H3}, \eqref{H5} and \eqref{H6}, there exist a critical speed $c_1$ given by $c_1 = I_1/I_0$ and $0 < \tilde{\ep}_1 \ll 1$ sufficiently small such that, for each value $c \in (c_1 - \tilde{\ep}_1, c_1)$ if $f'(1) > c_1$ (respectively, for each value $c \in (c_1, c_1 + \tilde{\ep}_1)$ if $f'(1) < c_1$) system \eqref{firstos} has a unique closed periodic orbit solution $(\overline{U},\overline{V})(z)$ which becomes the homoclinic loop at $P_1 = (1,0)$ from Theorem \ref{theoexisthomo} as $c \to c_1^-$ (respectively, as $c \to c_1^+$). Moreover, the amplitude of the periodic orbits and their fundamental periods behave like
\[
|\overline{U}|, |\overline{V}| = O(1), 
\]
and
\[
T(c) = O(|\log (|c - c_1|)|) \to \infty,
\]
respectively, as $c \to c_1$.
\end{theorem}
\begin{proof}
Under the assumptions, it is clear that $P_1 = (1,0)$ is a hyperbolic saddle for system \eqref{firstos} with $c = c_1$ (see Remark \ref{remsamesad}). Also, from Theorem \ref{theoexisthomo} this system underlies a homoclinic orbit joining $P_1$ with itself for this critical value for the speed. The eigenvalues of the linearization of \eqref{firstos} at $P_1$ evaluated at the critical bifurcation parameter $c = c_1$ satisfy
\[
\lambda_1(c_1) < 0 <\lambda_2(c_1),
\]
where $\lambda_1(c_1)$ and $\lambda_2(c_1)$ are given by \eqref{evaluessaddle}. Hence the saddle quantity is non-zero,
\[
\Sigma_0 = f'(1) - c_1 \neq 0,
\]
in view of \eqref{H6}.  From Andronov-Leontovich's theorem we conclude the existence of $\tilde{\ep}_1 > 0$ sufficiently small such that, if $f'(1) > c_1$ (respectively, $f'(1) < c_1$) then for each $c \in (c_1- \tilde{\ep}_1,c_1)$ (respectively, $c \in (c_1, c_1 + \tilde{\ep}_1)$) there exists a unique closed periodic orbit for system \eqref{firstos} with large fundamental period $T(c)$. That the amplitude of the family of periodic orbits (for each $c$ near $c_1$) is of order $O(1)$ follows directly from the fact that they belong to a neighborhood of the homoclinic loop. That the fundamental period behaves like $O(|\log (|c - c_1|)|)$ follows from a direct estimation of the time required for a trajectory to pass by a saddle point (see Gaspard \cite{Gasp90}, or Exercise 8.4.12 in \cite{Strg94}). The theorem is proved.
\end{proof}

\subsection*{Proof of Theorem \ref{thmexlarge}}

Assume \eqref{H1}, \eqref{H2}, \eqref{H3}, \eqref{H5} and \eqref{H6}. Then apply Theorem \ref{theoexistlp} to conclude the existence of a family of periodic orbits parametrized by
\[
0 < \ep := | c - c_1 | \in (0, \tilde{\ep}_1),
\]
which we denote as $(\bU^\ep,\bV^\ep) (z) =: (\varphi^\ep, (\varphi^\ep)')(z)$, $z \in \R$, solutions to system \eqref{firstos} with speed value $c(\ep) = c_1 + \ep$ if $f'(1) < c_1$ or $c(\ep) = c_1 - \ep$ if $f'(1) > c_1$, with fundamental period
\[
T_\ep = O(| \log \ep |) \to \infty,
\]
and amplitude
\[
|\varphi^\ep(z)|, |(\varphi^\ep)'(z)| = O(1),
\]
as $\ep \to 0^+$. Moreover, the family of orbits converge to the homoclinic loop of Theorem \ref{theoexisthomo}, relative to the saddle point $P_1 = (1,0)$ as $\ep \to 0^+$, which we denote as
\[
(\varphi^0,(\varphi^0)')(z) := (\psi, \psi')(z), \qquad z \in \R,
\]
with $(\varphi^0, (\varphi^0)')(z) \to (1,0)$ exponentially fast as $z \to \pm \infty$. Thanks to this convergence of the family we know that there exists $\widetilde{\delta}(\ep) > 0$ such that $\widetilde{\delta}(\ep) \to 0$ as $\ep \to 0^+$ and 
\[
|\varphi^0(z) - \varphi^\ep(z)| \leq \widetilde{\delta}(\ep), \qquad \text{for all } \; |z| \leq \frac{T_\ep}{2}.
\]

Since the homoclinic loop $(\varphi^0,(\varphi^0)') = (\psi, \psi')$ is non-degenerate in the sense of Beyn \cite{Beyn90a,Beyn90b} (see Definition \ref{defnondeg} and Lemma \ref{lemnondefgen} in Appendix \ref{apenon} below) then we can apply Corollary 3.2 in Beyn \cite{Beyn90a} (p. 178) to conclude that there exists $0 < {\ep}_1 < \tilde{\ep}_1$ sufficiently small and an appropriate reparametrization of the phase $z$ such that
\[
\begin{aligned}
&\sup_{z \in [-\frac{T_\ep}{2}, \frac{T_\ep}{2}]} \left( |\varphi^0(z) - \varphi^\ep(z)| + |(\varphi^0)'(z) - (\varphi^\ep)'(z)  |\right) \leq \\ & \qquad \qquad \leq C \exp \Big(\!-\!\big( \min \{ \lambda_2(c_1), |\lambda_1(c_1)|\} \big) \frac{T_\ep}{2} \Big),
\end{aligned}
\]
and
\[
\ep \leq C \exp \left( -  \big( \min \{ \lambda_2(c_1), |\lambda_1(c_1)|\} \big) T_\ep \right),
\]
for each $0 < \ep < {\ep}_1$, where $\lambda_2(c_1)$, $\lambda_2(c_1)$ are the spectral bounds of the homoclinic orbit given by \eqref{evaluessaddle}. Set $\kappa = \min \{ \lambda_2(c_1), |\lambda_1(c_1)|\} > 0$ (like in \eqref{defkappa}). This shows the bounds \eqref{boundw1} and \eqref{boundc1}. Finally, the family of orbits $\varphi^\ep$ is of class $C^3$ in $z \in \R$ and in the bifurcation parameter $c$ thanks to the regularity of $f$ and $g$, and to standard ODE results. The theorem is proved.
\qed

\begin{remark}
Actually, the bound \eqref{boundc1} readily implies
\[
T_\ep \leq C_1 | \log \ep |,
\]
for some $C_1 > 0$, that is, $T_\ep = O(|\log \ep|)$ as stated independently in Theorem \ref{theoexistlp} as a consequence of Andronov-Leontovich's theorem.
\end{remark}

\section{The spectral stability problem}
\label{secspe}

\subsection{Perturbation equations and definition of spectrum}

To begin our stability analysis, let us consider a perturbation $v$ of a bounded periodic traveling wave $\varphi = \varphi(z)$ with fundamental period $T > 0$. Substituting $v + \varphi$ into the viscous balance law \eqref{vbl} written in the Galilean frame associated with the independent variables $(z,t) = (x-ct,t)$, one finds that the perturbation $v = v(z,t)$ is a solution to the nonlinear equation
\[
v_t - cv_z + f(v+\varphi)_z - f(\varphi)_z = v_{zz} + g(v+\varphi) - g(\varphi),
\]
where we have substituted the equation satisfied by the profile \eqref{profileq}. For nearby perturbations the leading approximation is given by the linearization of this equation around $\varphi$, yielding
\[
v_t = v_{zz} + (c - f'(\varphi))v_z + (g'(\varphi) - f'(\varphi)_z) v.
\]
Specializing to perturbations of the form $v(z,t) = e^{\lambda t} w(z)$, where $\lambda \in \C$ and $w$ lies in an appropriate Banach space $X$ we arrive at the eigenvalue problem
\begin{equation}
\label{primeraev}
\lambda w = w_{zz} + (c - f'(\varphi))w_z + (g'(\varphi) - f'(\varphi)_z) w,
\end{equation}
in which the complex growth rate appears as the eigenvalue. Motivated by the notion of spatially localized, finite energy perturbations in the Galilean coordinate frame in which the periodic wave is stationary, we consider $X = L^2(\R;\C)$ and define the linearized operator around the wave as
\begin{equation}
\label{linop}
\left\{
\begin{aligned}
\cL \, &: \, L^2(\R;\C) \longrightarrow L^2(\R;\C),\\
\cL \, &: = \, \partial_z^2 + a_1(z) \partial_z + a_0(z) \Id, 
\end{aligned}
\right.
\end{equation}
with dense domain $\cD(\cL) = H^2(\R;\C)$, and where the coefficients,
\begin{equation}
\label{defas}
\begin{aligned}
a_1(z) &:= c - f'(\varphi),\\
a_0(z) &:= g'(\varphi) - f'(\varphi)_z,
\end{aligned}
\end{equation}
are bounded and periodic, satisfying $a_j(z + T) = a_j(z)$ for all $z \in \R$, $j = 0,1$. Notice that $\cL$ is a densely defined, closed operator acting on $L^2$ with domain $\cD(\cL) = H^2(\R;\C)$. Hence, the eigenvalue problem \eqref{primeraev} is recast as
\begin{equation}
\label{specprobi}
\cL w = \lambda w,
\end{equation}
for some $\lambda \in \C$ and $w \in \cD(\cL) = H^2(\R;\C)$. The following definition is standard (cf. \cite{EE87,KaPro13,Kat80}).
\begin{definition}[resolvent and spectra]
Let $\cL : X \to Y$ be a closed linear operator, with $X, Y$ Banach spaces and dense domain $\cD(\cL) \subset X$. The \emph{resolvent} of $\cL$, denoted as $\rho(\cL)$, is the set of all complex numbers $\lambda \in \C$ such that $\cL - \lambda$ is injective and onto, and $(\cL - \lambda )^{-1}$ is a bounded operator. The \emph{point spectrum} of $\cL$, denoted as $\ptsp(\cL)$, is the set of $\lambda \in \C$ such that $\cL - \lambda$ is a Fredholm operator with index equal to zero and non-trivial kernel. The \emph{essential spectrum} of $\cL$, denoted as $\ess(\cL)$, is the set of all $\lambda \in \C$ such that either $\cL -\lambda$ is not Fredholm, or it is Fredholm with non-zero index. The \emph{spectrum} of $\cL$ is defined as $\sigma(\cL) = \ess(\cL) \cup \ptsp(\cL)$.
\end{definition}
\begin{remark}
Since $\cL$ is closed then $\rho(\cL) = \C \backslash \sigma(\cL)$. Moreover, $\ptsp(\cL)$ consists of isolated eigenvalues with finite multiplicity (cf. \cite{KaPro13,Kat80}). We denote the spectrum of an operator $\cL$ when computed with respect to the space $X$ as $\sigma(\cL)_{|X}$.
\end{remark}

\begin{definition}[spectral stability]
\label{defspectstab}
We say that a bounded periodic wave $\varphi$ is \emph{spectrally stable} as a solution to the viscous balance law \eqref{vbl} if the $L^2$-spectrum of the linearized operator around the wave defined in \eqref{linop} satisfies
\[
\sigma(\cL)_{|L^2} \cap \{\lambda \in \C \, : \, \Re \lambda > 0\} = \varnothing.
\]
Otherwise we say that it is \emph{spectrally unstable}.
\end{definition}

\subsection{Floquet characterization of the spectrum}

It is well-known that, as the coefficients of the operator $\cL$ in \eqref{linop} are periodic in $z$, Floquet theory implies that its $L^2$-spectrum is purely essential or ``continuous", that is, $\sigma(\cL)_{|L^2} = \ess(\cL)_{|L^2}$, which means that there are no isolated eigenvalues (see, e.g., Lemma 9 in \cite{DunSch2}, p. 1487, or Lemma 3.3 in \cite{JMMP2}). 

Moreover, it is possible to parametrize the spectrum in terms of Floquet multipliers of the form $e^{i\theta}\in \bbS^1$, or 
equivalently, by $\theta\in\mathbb{R}\pmod{2\pi}$. Let us define the set $\sigma_\theta$ as the set of
complex numbers $\lambda$ for which there exists a bounded, non-trivial solution $w \in L^\infty(\R;\C)$ to the boundary value problem
\begin{equation}
\label{spectw}
\left\{
\begin{aligned}
\lambda w &=  w_{zz} + \big(c - f'(\varphi)\big) w_z + \big(g'(\varphi) - f'(\varphi)_z\big) w,\\
w(T) &= e^{i\theta} w(0),\\
w_z(T) &= e^{i \theta} w_z(0),
\end{aligned}
\right.
\end{equation}
for some $\theta \in (-\pi,\pi]$. We thus define the 
\emph{Floquet  spectrum} $\sigma_F$ as:
\[
 \sigma_F := \!\!\bigcup_{-\pi<\theta \leq \pi}\sigma_\theta.
\]

\begin{lemma}[Floquet characterization of the spectrum] $\sigma(\cL)_{|L^2} = \sigma_F$.
\end{lemma}
\begin{proof}
See the proof of Proposition 3.4 in \cite{JMMP2}.
\end{proof}

In this fashion, the purely essential spectrum $\sigma(\cL)_{|L^2}$ can be written as the union of partial spectral $\sigma_\theta$. Moreover, each set $\sigma_\theta$ is \emph{discrete} as it is the zero set of an analytic function (see \cite{JMMP2}, p. 4649). It is to be noticed that if $\theta = 0$ then the boundary conditions in \eqref{spectw} become periodic and $\sigma_0$ detects perturbations 
which are co-periodic. By a symmetric argument, the set $\sigma_{\pi}$ detects anti-periodic perturbations.

To remove the $\theta$-dependence associated with the boundary conditions in \eqref{spectw} one can pose the problem in a proper periodic space independently of (but indexed by) $\theta \in (-\pi, \pi]$ by means of a \emph{Bloch-wave decomposition}. Define
\[
u(z):= e^{- i\theta z/T} w(z).
\]
Then the non-separated boundary conditions in \eqref{spectw} transform into periodic ones, $\partial_z^j u(T) = \partial_z^j u(0)$, $j = 0,1$, and the spectral problem \eqref{spectw} is recast as 
\[
\cL_\theta u = \lambda u,
\]
for a one-parameter family of Bloch operators
\[
\left\{
\begin{aligned}
\cL_\theta &:= (\partial_z + i\theta/T)^2 + a_1(z) (\partial_z + i \theta/T) + a_0(z) \Id,\\
\cL_\theta &: \Ldper([0,T];\C) \to \Ldper([0,T];\C),
\end{aligned}
\right.
\]
with domain $\cD(\cL_\theta) = \Hdper([0,T]; \C)$, parametrized by $\theta \in (-\pi,\pi]$. Since the family has compactly embedded domains in $\Ldper([0,T];\C)$ then their spectrum consists entirely of isolated eigenvalues, $\sigma(\cL_\theta)_{|\Ldper} = \ptsp(\cL_\theta)_{|\Ldper} $. Moreover, they depend continuously on the Bloch parameter $\theta$, which is typically a local coordinate for the spectrum $\sigma(\cL)_{|L^2}$, explaining the intuition that the former is purely ``continuous" and consisting of curves of spectrum in the complex plane (see Proposition 3.7 in \cite{JMMP2}) meaning that $\lambda \in \sigma(\cL)_{|L^2}$ if and only if $\lambda \in \ptsp(\cL_\theta)_{\Ldper}$ for some $\theta \in (-\pi,\pi]$. Consequently, we also have the spectral representation (see \cite{Grd1,KaPro13} for details),
\[
\sigma(\cL)_{|L^2} =  \!\!\bigcup_{-\pi<\theta \leq \pi}\ptsp(\cL_\theta)_{|\Ldper}.
\]

\subsection{The spectral problem as a first order system. Evans functions.}

Finally, following the seminal work of Alexander, Gardner and Jones \cite{AGJ90} (see also \cite{San02, KaPro13}), one can write the spectral problem \eqref{specprobi} for any second order differential linearized operator $\cL$ around a traveling wave $\varphi$ of the form \eqref{linop} and with bounded coefficients (not necessarily periodic) as a first order system, namely,
\begin{equation}
\label{firstordersyst}
W_z = \A(z,\lambda) W,
\end{equation}
where
\[
W = \begin{pmatrix} w \\ w_z \end{pmatrix} \in H^2(\R;\C) \times H^1(\R;\C),
\]
and with coefficients
\[
\A(z,\lambda) := \begin{pmatrix} 0 & 1 \\ \lambda - a_0(z) & - a_1(z)
\end{pmatrix} = \begin{pmatrix}
0 & 1 \\ \lambda - (g'(\varphi) - f'(\varphi)_z) & -c + f'(\varphi)
\end{pmatrix},
\]
which are analytic in $\lambda \in \C$ and bounded functions of $z \in \R$ of class $C^1(\R;\C^{2 \times 2})$. Then one can consider the family of closed, densely defined operators $\cT(\lambda)$ given by
\[
\left\{
\begin{aligned}
\cT(\lambda) W &:= W_z - \A(z,\lambda) W,\\
\cT(\lambda) \, &: \, L^2(\R;\C) \times L^2(\R;\C)  \to L^2(\R;\C) \times L^2(\R;\C),
\end{aligned}
\right.
\]
with dense domain $\cD(\cT(\lambda)) = H^2(\R;\C) \times H^1(\R;\C)$ and parametrized by $\lambda \in \C$. It is well-known that the set of complex numbers such that $\cT(\lambda)$ is Fredholm with index zero coincides in location and multiplicity with the set $\ptsp(\cL)_{|L^2}$ of isolated eigenvalues of $\cL$. Similar definitions an equivalences apply to the resolvent and the essential spectrum (see Definition 3.1 in \cite{JMMP2}, as well as \cite{AGJ90,San02,KaPro13} for further details). Another advantage is that we can relate the spectrum of the linearized operator to dynamical systems tools associated to \eqref{firstordersyst} such as Evans functions.

\subsubsection{The homoclinic Evans function}

Let us suppose that the traveling wave under consideration, $\varphi = \varphi(z)$, is a traveling pulse (or homoclinic orbit for system \eqref{firstos}) with finite limits $\lim_{z \to \pm \infty} \varphi(z) = u_\infty$ and $\lim_{z \to \pm \infty} \varphi_z(z) = 0$ (in our case with system \eqref{firstos} the homoclinic orbit is relative to the hyperbolic saddle $P_1 = (1,0)$, so that $u_\infty = 1$). We can then define the asymptotic coefficients $\A_\infty(\lambda) := \lim_{z \to \pm \infty} \A(z,\lambda)$ and a finite set of dispersion curves in the complex plane, $\lambda_j(\xi)$, $1 \leq j \leq m$, $\xi \in \R$, determined by $\det (\A_\infty(\lambda) - i \xi \I) = 0$, where $\I$ denotes the identity matrix. By definition and continuity of the coefficients, for values of $\lambda$ away from these curves the eigenvalues of $\A_\infty(\lambda)$ have non-vanishing real part (domain of hyperbolicity, in the sense of dynamical systems, of the coefficients $\A_\infty$). In fact, and thanks to hypotheses \eqref{H3}, it can be shown that on the open set $\Omega_\infty = \{\lambda \in \C \, : \, \Re \lambda > g'(1)\}$ (the set of consistent splitting; see \eqref{defOmegai} below), the coefficients $\A_\infty(\lambda)$ have exactly one eigenvalue with positive real part, exactly one with negative real part, and that $\Omega_\infty$ contains the unstable complex half plane $\C_+ = \{ \Re \lambda > 0\}$. More importantly, due to exponential decay of the traveling pulse, exponential dichotomies theory and invariance of Morse indices, the conclusion can be extrapolated to the variable coefficients system \eqref{firstordersyst}: there exists one solution $W^+(z,\lambda)$ spanning the stable space of \eqref{firstordersyst} decaying as $z \to +\infty$, and one solution $W^-(z,\lambda)$ decaying as $z \to -\infty$ for each $\lambda \in \Omega_\infty$. The homoclinic Evans function is defined as the Wronskian
\begin{equation}
\label{defhomEv}
D(\lambda) := \det ( W^-(z,\lambda), \, W^+(z,\lambda))_{| z=0},
\end{equation}
for $\lambda \in \Omega_\infty$. The Evans function is not unique but they all differ by appropriate vanishing factors and are endowed with the following properties: $D$ is analytic on $\Omega_\infty$ and vanishes at $\lambda \in \Omega_\infty$ if and only if $\lambda \in \ptsp(\cL)_{|L^2}$, with the order of the zero being the algebraic multiplicity of the eigenvalue (see \cite{KaPro13,San02} and the many references therein).

\subsubsection{The periodic Evans function}

Let us now go back to the case of a periodic traveling wave $\varphi = \varphi(z)$ with fundamental period $T$. In such case, the matrix $\A(z,\lambda)$ is $T$-periodic in $z$ and we may apply Floquet theory for ODEs. Let $\F = \F(z,\lambda)$ denote the identity-normalized fundamental solution matrix for system \eqref{firstordersyst}, that is, the unique solution to $\partial_z \F = \A(z,\lambda) \F$ with initial condition $\F(0,\lambda) = \I$ for every $\lambda \in \C$. The $T$-periodicity in $z$ of the coefficients $\A$ then implies that $\F(z+T,\lambda) = \F(z,\lambda) \M(\lambda)$ for all $z \in R$, where $\M(\lambda) := \F(T,\lambda)$ is the \emph{monodromy matrix} for system \eqref{firstordersyst} and it is an entire function of $\lambda \in \C$ (see, e.g., \cite{JMMP2,JMMP3,KaPro13}). It can be shown (see Proposition 3.4 in \cite{JMMP2}) that $\lambda \in \sigma(\cL)_{|L^2}$ if and only if there exists $\mu \in \C$ with $|\mu| = 1$ such that
\[
\det (\M(\lambda) - \mu \I) = 0.
\]
This is, at least one of the eigenvalues of the monodromy matrix, also known as \emph{Floquet multipliers}, lies in complex unit circle. Gardner \cite{Grd1,Grd2} defines the \emph{periodic Evans function} as the restriction of the above determinant to $\mu$ in the unit circle $\bbS^1 \subset \C$,
\begin{equation}
\label{defperEv}
D(\lambda,\theta) := \det (\M(\lambda) - e^{i\theta} \I).
\end{equation}
For each $\theta \in \R$ (mod $2 \pi$), the periodic Evans function is an entire function of $\lambda \in \C$ whose isolated zeroes are particular points of the (continuous) spectrum $\lambda \in \sigma(\cL)_{|L^2}$. Of course, each $\theta \in (-\pi,\pi]$ is precisely the Bloch parameter in \eqref{spectw} associated to a Floquet multiplier of the form $e^{i\theta}$. For each $\theta$ fixed, the zeroes of the analytic function $D(\lambda,\theta)$ are discrete and coincide in order (multiplicity) and location with the discrete Bloch spectrum, $\ptsp(\cL_\theta)_{|\Ldper}$. When the periodic wave under consideration is a large period wave emerging from a homoclinic bifurcation (like the ones described in Theorem \ref{thmexlarge}), the homoclinic Evans function associated to the traveling pulse and the periodic Evans function of the family are closely related to each other (see \cite{Grd2,SS6,YngZ19,Z16}).

\section{Spectral instability of small-amplitude waves}
\label{secsmall}

This section is devoted to prove Theorem \ref{theosmalli}. The main idea is that the spectral stability of small amplitude waves can be studied as a perturbation of the zero-amplitude case, that is, from a dispersion relation of the PDE linearized around the zero solution. (For a related analysis in a scalar Hamiltonian context, see K\'{o}llar \emph{et al.} \cite{KDT19}.)

\subsection{Overview of spectral perturbation theory}

For convenience of the reader, in this section we review the basic perturbation theory for a linear family of operators of the form $\cL(\vep) = \cL^0 + \vep \cA$. We describe a simple criterion in the case of $\cL^0$ self-adjoint to establish when an eigenvale $\lambda_0$ of $\cL^0$ persists for $\vep \neq 0$ and small, as described in the book by Hislop and Sigal \cite{HiSi96} (chapter 15). For the more general theory the reader is referred to Kato \cite{Kat80}. We first recall some basic definitions.

\begin{definition}
Let $\cA, \cS : X \to Y$ be linear operators with $X, Y$ Banach spaces. We say that $\cA$ is relatively bounded with respect to $\cS$, or simply $\cS$-bounded, provided that $\cD(\cS) \subset \cD(\cA)$ and that there exist $\alpha, \beta \geq 0$ such that
\[
\| \cA u \| \leq \alpha \| u \| + \beta \| \cS u \|,
\]
for all $u \in \cD(\cS)$.
\end{definition}

\begin{definition}
Let $\cL : X \to Y$ be a closed operator with $X, Y$ Banach spaces. Suppose $\Gamma \subset \rho(\cL)$ is a closed rectifiable contour around a discrete eigenvalue of $\cL$, $\lambda_0 \in \ptsp(\cL)$. The Riesz projection for $\cL$ and $\lambda_0$ is defined as
\[
\cP = \frac{1}{2 \pi i} \oint_{\Gamma} (\cL - \lambda)^{-1} \, d\lambda.
\] 
The algebraic multiplicity of $\lambda_0$ is the dimension of the range of $\cP$, $\overline{m}(\lambda_0) = \dim R(\cP)$, whereas the geometric multiplicity of $\lambda_0$ is the nullity of $\cL - \lambda_0$, $\underline{m}(\lambda_0) = \dim \ker (\cL - \lambda_0)$. Clearly $\underline{m}(\lambda_0) \leq \overline{m}(\lambda_0)$.
\end{definition}

Consider a family of operators 
\begin{equation}
\label{familia}
\cL(\vep) = \cL^0 + \vep \cA,
\end{equation}
defined on a Hilbert space $H$ such that $\cD(\cL^0) \subset \cD(\cA) \subset H$, so that $\cL(\vep) : \cD(\cL^0) \subset H \to H$ for all $\vep$ small.

\begin{definition}
A discrete eigenvalue $\lambda_0 \in \ptsp(\cL^0)$ is \emph{stable with respect to the family} $\cL(\vep)$ if
\begin{itemize}
\item[(i)] there exists $r > 0$ such that $\Gamma_r = \{ \lambda \in \C \, :  \, |\lambda - \lambda_0| = r  \} \subset \rho(\cL(\vep))$ for all small $|\vep| \ll 1$, and
\item[(ii)] if $\cP_\vep$ denotes the Riesz projection for $\cL(\vep)$ and $\lambda_0$ corresponding to the contour $\Gamma_r$ then $\cP_\vep \to \cP_0$ in norm as $\vep \to 0$.
\end{itemize}
\end{definition}

The following proposition provides a simple criterion for the persistence of a discrete eigenvalue $\lambda_0$ of $\cL^0$ under the family $\cL(\vep)$ in the particular case when $\cL^0$ is self-adjoint.

\begin{proposition}
\label{proppert}
Suppose that $\cL^0$ is a self adjoint operator and that $\cA$ is $\cL^0$-bounded. Then all discrete eigenvalues of $\cL^0$ are stable with respect to the family $\cL(\vep)$. Moreover, for all $|\vep|$ sufficiently small the operator $\cL(\vep)$ has discrete eigenvalues $\lambda_j(\vep)$ in a neighborhood of $\lambda_0$ of total algebraic multiplicities equal to the algebraic multiplicity of $\lambda_0$. Each eigenvalue $\lambda_j(\vep)$ admits an analytic series expansion (or Rayleigh-Schr\"odinger expansion) of the form
\[
\lambda_j(\vep) = \lambda_0 + \sum_{k=1}^\infty \alpha_k^j \vep^k,
\]
for some $\alpha_k^j \in \C$ with non-zero radius of convergence.
\end{proposition}
\begin{proof}
See Proposition 15.3, Theorem 15.7 and formulae (15.6) - (15.8) in Hislop and Sigal \cite{HiSi96} (chapter 15, pp. 149--157).
\end{proof}

\begin{remark}
It is important to observe that $\cA$ does not need to be self-adjoint (not even symmetric). Proposition \ref{proppert} guarantees that the eigenvalue $\lambda_0$ splits into discrete eigenvalues $\lambda_j(\vep)$ of $\cL(\vep)$ with same total multiplicity in a $\vep$-neighborhood of $\lambda_0$.
\end{remark}

\subsection{Emergence of unstable eigenvalues: proof of Theorem \ref{theosmalli}}

Let us consider the family of periodic, small-amplitude waves from Theorem \ref{thmexbded} which are parametrized by $\ep := |c - c_0| \in (0, \ep_0)$, where $c_0 = f'(0)$ and $c = c(\ep)$ is the wave speed of each element of the family. These waves have amplitude of order
\[
|\varphi^\ep|, |\varphi^\ep_z| = O(\sqrt{\ep}),
\]
and fundamental period
\[
T_\ep 
= \frac{2 \pi}{\sqrt{g'(0)}} + O(\ep) =: T_0 + O(\ep).
\]
The associated spectral problem \eqref{spectw} 
\begin{equation}
\label{probw}
\left\{
\begin{aligned}
\lambda w &= w_{zz} + \big(c(\ep) - f'(\varphi^\ep)\big) w_z + \big(g'(\varphi^\ep) - f'(\varphi^\ep)_z\big) w,\\
w(T_\ep) &= e^{i\theta} w(0),\\
w_z(T_\ep) &= e^{i \theta} w_z(0),  \qquad \text{some } \, \theta \in (-\pi, \pi],
\end{aligned}
\right.
\end{equation}
can be recast as an equivalent spectral problem in a periodic space. Consider the following Bloch transformation
\[
y:= \frac{\pi z}{T_\ep}, \quad u(y) := e^{-i \theta y/\pi} w \Big( \frac{T_\ep y}{\pi} \Big),
\]
for given $\theta \in (-\pi, \pi]$. Then the spectral problem \eqref{probw} transforms into
\[
\lambda u = \frac{1}{T_\ep^2} \big(i\theta + \pi \partial_y\big)^2 u + \frac{\bar{a}^\ep_1(y)}{T_\ep} \big(i\theta + \pi \partial_y\big) u + \bar{a}^\ep_1(y) u,
\]
where the coefficients
\[
\label{coeffsaep}
\begin{aligned}
\bar{a}^\ep_1(y) &:= c(\ep) - f'( \varphi^\ep( T_\ep y/ \pi)), \\
\bar{a}^\ep_0(y) &:= g'( \varphi^\ep( T_\ep y/ \pi)) - f'' ( \varphi^\ep( T_\ep y/ \pi)) \varphi_z^\ep( T_\ep y/ \pi ),
\end{aligned}
\]
are clearly $\pi$-periodic in the $y$ variable and where $u \in \Hdper([0,\pi];\C)$ is subject to $\pi$-periodic boundary conditions,
\[
u(0) = u(\pi), \quad u_y(0) = u_y(\pi).
\]
Multiply by $T_\ep^2$ (constant) to obtain the following equivalent spectral problem
\begin{equation}
\label{newL}
\cL_\theta u = \widetilde{\lambda} u,
\end{equation}
for the operator
\[
\left\{
\begin{aligned}
\cL_\theta  &:= \big(i\theta + \pi \partial_y\big)^2  + a^\ep_1(y) \big(i\theta + \pi \partial_y \big)  + a^\ep_0(y) \Id,\\
\cL_\theta  &: \cD(\cL_\theta) = \Hdper([0,\pi];\C) \subset \Ldper([0,\pi];\C) \longrightarrow \Ldper([0,\pi];\C), \\
\end{aligned}
\right.
\]
for any given $\theta \in (-\pi, \pi]$ and where
\[
\begin{aligned}
\widetilde{\lambda} &:= T_\ep^2 \, \lambda,\\
a^\ep_1(y) &:= T_\ep \, \bar{a}^\ep_1(y),\\
a^\ep_0(y) &:= T_\ep^2 \, \bar{a}^\ep_0(y).
\end{aligned}
\]
Let us write \eqref{newL} as a perturbation problem. The coefficients can be written as
\[
\begin{aligned}
a^\ep_1(y) &= \big( T_0 + O(\ep) \big)\big( c(\ep) - f'( \varphi^\ep( T_\ep y/ \pi)) \big)\\
&= \big( T_0 + O(\ep) \big)\big( c_0 + O(\ep) - f'(0) + O(| \varphi^\ep|) \big)\\
&= \big( T_0 + O(\ep) \big)\big( O(\ep) + O(\sqrt{\ep}) \big)\\
&= \sqrt{\ep} \, b_1(y),
\end{aligned}
\]
where
\[
b_1(y) := \frac{1}{\sqrt{\ep}} \, a^\ep_1(y) = O(1), \quad y \in [0,\pi].
\]
Likewise
\[
\begin{aligned}
a^\ep_0(y) &= \big( T_0 + O(\ep) \big)^2 \Big( g'( \varphi^\ep( T_\ep y/ \pi)) - f''(\varphi^\ep( T_\ep y/ \pi)) \varphi^\ep_z ( T_\ep y/ \pi)\Big)\\
&= \big( T_0^2 + O(\ep) \big) \big( g'(0) + O(|\varphi^\ep|) + O(|\varphi^\ep_z|) )\big)\\
&= \big( T_0^2 + O(\ep) \big)\big( g'(0) + O(\sqrt{\ep}) \big)\\
&= T_0^2 g'(0) + O(\sqrt{\ep})\\
&= 4\pi^2 + O(\sqrt{\ep}).
\end{aligned}
\]
Thus we write
\[
 b_0(y) := \frac{a_0^\ep(y) - 4\pi^2}{\sqrt{\ep}} = O(1), \quad y \in [0,\pi].
\]
Now, if we denote $\vep := \sqrt{\ep} \in (0, \sqrt{\ep_0})$ we obtain 
\[
\cL_\theta u = \big(i\theta + \pi \partial_y\big)^2 u + 4 \pi^2 u + \vep b_1(y) \big(i\theta + \pi \partial_y\big) u + \vep b_0(y) u = \cL_\theta^0 u + \vep \cL_\theta^1 u,
\]
where the operators $\cL_\theta^0$ and $\cL_\theta^1$ are defined as
\[
\left\{
\begin{aligned}
\cL_\theta^0 &:= \big(i\theta + \pi \partial_y\big)^2 + 4 \pi^2 \Id,\\
\cL_\theta^0 &: \cD(\cL_\theta^0) = \Hdper([0,\pi],\C) \subset \Ldper([0,\pi],\C) \longrightarrow \Ldper([0,\pi],\C),
\end{aligned}
\right.
\]
and as
\[
\left\{
\begin{aligned}
\cL_\theta^1 &:= b_1(y) \big(i\theta + \pi \partial_y\big) + b_0(y) \Id,\\
\cL_\theta^1 &: \cD(\cL_\theta^1) = \Huper([0,\pi],\C) \subset \Ldper([0,\pi],\C) \longrightarrow \Ldper([0,\pi],\C),
\end{aligned}
\right.
\]
respectively. Note that here $b_j(y) = O(1)$, $y \in [0,\pi]$, $j=0,1$. Therefore, the spectral problem \eqref{newL} is recast as a perturbed spectral problem of the form
\begin{equation}
\label{pertevL}
\cL_\theta u = \cL_\theta^0 u + \vep \cL_\theta^1 u = \widetilde{\lambda} u, \qquad u \in \Hdper([0,\pi],\C).
\end{equation}
Here the densely defined operators $\cL_\theta^l$, $l = 0,1$,  act on $\Ldper([0,\pi];\C)$ with norm
\[
\| u \|_{\Ldper} = \left( \int_0^\pi |u(z)|^2 \, dz \right)^{1/2}.
\]
\begin{lemma}
\label{L0bded}
For each $\theta \in (-\pi,\pi]$, $\cL_\theta^1$ is $\cL_\theta^0$-bounded.
\end{lemma}
\begin{proof}
First, notice that $\cD(\cL_\theta^0) = \Hdper([0,\pi],\C) \subset \cD(\cL_\theta^1) = \Huper([0,\pi],\C)$. We need to show that there exist uniform constants $\alpha, \beta \geq 0$ such that
\[
\| \cL_\theta^1 u \|_{\Ldper} \leq \alpha \|u \|_{\Ldper} + \beta \| \cL_\theta^0 u \|_{\Ldper},
\]
for all $u \in \Hdper([0,\pi];\C)$. Let us estimate
\begin{equation}
\label{nvak1}
\begin{aligned}
\| \cL_\theta^1 u \|_{\Ldper} &= \| b_1(y) \big( i \theta + \partial_y \big) u + b_0(y) u \|_{\Ldper} \\
&\leq \pi \| b_1(y) \|_{L^\infty} \| u_y \|_{\Ldper} + \Big[ |\theta| \|b_1(y) \|_{L^\infty} + \|b_0(y) \|_{L^\infty} \Big] \| u \|_{\Ldper}\\
&\leq \pi K_1 \| u_y \|_{\Ldper} + (\pi K_1 + K_0) \| u \|_{\Ldper},
\end{aligned}
\end{equation}
since $|\theta| \leq \pi$ and where $0 < K_1 := \|b_1 \|_{L^\infty}$, $0 < K_0 := \|b_0 \|_{L^\infty}$. Now, it is known (see Kato \cite{Kat80}, p. 192) that for all $u \in H^2([0,\pi]; \C)$ there holds the estimate
\begin{equation}
\label{stark}
\| u_y \|_{L^2(0,\pi)} \leq \frac{\pi}{N-1} \| u_{yy} \|_{L^2(0,\pi)}  + \frac{2N(N+1)}{\pi (N-1)} \| u \|_{L^2(0,\pi)},
\end{equation}
where $N$ is any positive number with $N > 1$. Substitute \eqref{stark} into \eqref{nvak1} to obtain
\begin{equation}
\label{nvak2}
\| \cL_\theta^1 u \|_{\Ldper} \leq C_1 (N) \| u_{yy} \|_{\Ldper} + C_0(N) \| u \|_{\Ldper},
\end{equation}
where
\[
\begin{aligned}
C_1(N) &= \frac{\pi^2 K_1}{N-1} > 0,\\
C_0(N) &= K_0 + \frac{K_1}{N-1} \Big( \pi (N-1) + 2N(N+1) \Big) > 0.
\end{aligned}
\]
On the other hand, the estimate
\[
\| \cL_\theta^0 u \|_{\Ldper} = \|  (i\theta + \pi \partial_y )^2 u + 4 \pi^2 u\|_{\Ldper} \geq \pi^2 \| u_{yy} \|_{\Ldper} - \| 2i\theta \pi u_y + (4\pi^2 - \theta^2)u \|_{\Ldper},
\]
together with \eqref{stark} and $|\theta| \leq \pi$, yield
\[
\begin{aligned}
\pi^2 \| u_{yy} \|_{\Ldper} &\leq \| \cL_\theta^0 u \|_{\Ldper} + 2 \pi |\theta|  \|u_y \|_{\Ldper} + (4\pi^2 - \theta^2) \|u \|_{\Ldper}\\
&\leq \| \cL_\theta^0 u \|_{\Ldper} + 2 \pi^2 \Big( \frac{\pi}{N-1} \| u_{yy} \|_{\Ldper}  + \frac{2N(N+1)}{\pi (N-1)} \| u \|_{\Ldper} \Big) + 4 \pi^2 \| u \|_{\Ldper}\\
&\leq \| \cL_\theta^0 u \|_{\Ldper}  + \frac{2\pi^2}{N-1} \| u_{yy} \|_{\Ldper} + \frac{4 \pi}{N-1} \Big( \pi(N-1) + N(N+1) \Big) \| u \|_{\Ldper}.
\end{aligned}
\]
Let us choose $N$ sufficiently large, say $N \geq 1 + 4 \pi$, so that
\[
1 - \frac{2\pi}{N-1} \geq \frac{1}{2},
\]
and therefore
\[
\| u_{yy} \|_{\Ldper} \leq \frac{2}{\pi^2} \| \cL_\theta^0 u \|_{\Ldper} + \frac{8}{\pi (N-1)} \Big( \pi(N-1) + N(N+1) \Big) \| u \|_{\Ldper}.
\]
Upon substitution into \eqref{nvak2} we arrive at
\[
\| \cL_\theta^1 u \|_{\Ldper} \leq \alpha \|u \|_{\Ldper} + \beta \| \cL_\theta^0 u \|_{\Ldper},
\]
with uniform constants
\[
\begin{aligned}
\alpha &:= \frac{8 C_1(N)}{\pi (N-1)} \Big( \pi(N-1) + N(N+1) \Big) + C_0(N) >0,\\
\beta &:= \frac{2C_1(N)}{\pi^2} > 0.
\end{aligned}
\]
This yields the result.
\end{proof}

Now, let us take a look at the spectral problem \eqref{pertevL} specialized to the case of the Floquet exponent (or Bloch parameter) with $\theta = 0$, namely
\[
\cL_0 u = \cL_0^0 + \vep \cL_0^1 u = \widetilde{\lambda} u, \qquad u \in \Hdper([0,\pi];\C).
\]
First, it is to be observed that the operator
\[
\left\{
\begin{aligned}
\cL_0^0 &= \pi^2 \partial_y^2 + 4 \pi^2 \Id, \\
\cL_0^0 &:  \Ldper([0,\pi];\C)  \to \Ldper([0,\pi];\C),
\end{aligned}
\right.
\]
with domain $\cD(\cL_0^0) = \Hdper([0,\pi];\C)$, is clearly self-adjoint with a positive eigenvalue $\widetilde{\lambda}_0 = 4\pi^2$ associated to the constant eigenfunction $u_0(y) = 1/\sqrt{\pi} \in \Hdper([0,\pi];\C)$ satisfying $\| u_0 \|_{\Ldper} = 1$ and $u_0 \in \ker (\partial_y^2) \subset \Hdper([0,\pi];\C)$. Since $\cL_0^1$ is $\cL_0^0$-bounded because of Lemma \ref{L0bded}, by Proposition \ref{proppert} the operator $\cL_0 = \cL_0^0 + \vep \cL_0^1$ has discrete eigenvalues $\widetilde{\lambda}_j(\vep)$ in a $\vep$-neighborhood of $\widetilde{\lambda}_0 = 4\pi^2$ with multiplicities adding up to $m_0$ if $\vep$ is sufficiently small. Moreover, since $\widetilde{\lambda}_0 >0$ there holds
\[
\Re \lambda_j(\vep) > 0, \qquad |\vep| \ll 1.
\]
Hence, we have proved the following
\begin{lemma}
\label{lemunev}
For each $0 < \vep \ll 1$ sufficiently small there holds
\[
\ptsp( \cL_0^0 + \vep \cL_0^1)_{|\Ldper} \cap \{ \lambda \in \C \, : \, |\lambda - 4\pi^2| < r(\vep) \} \neq \varnothing,
\]
for some $r(\ep) = O(\vep) > 0$.
\end{lemma}

\subsection*{Proof of Theorem \ref{theosmalli}}

Now, since $\vep = \sqrt{\ep}$, from Lemma \ref{lemunev} we know that for $0 < \ep \ll 1$ sufficiently small there exist discrete eigenvalues $\lambda(\ep) \in \ptsp(\cL_0^0 + \sqrt{\ep} \cL_0^1)$ such that $|\lambda - 4 \pi^2| \leq C\sqrt{\ep}$ for some $C > 0$. Transforming back into the original problem, this implies that there exist eigenvalues $\lambda = \lambda(\ep)$ and bounded solutions $w$ of \eqref{probw} with $\theta = 0$ that satisfy
\[
|(T_0^2 + O(\ep)) \lambda(\ep) - 4 \pi^2 | = O(\sqrt{\ep}),
\]
or equivalently (in view that $T_0 = 4\pi^2/g'(0)$),
\begin{equation}
\label{aproxev}
|\lambda(\ep) - g'(0)| = O(\sqrt{\ep}), \qquad 0 < \ep \ll1.
\end{equation}

This implies that for $\ep > 0$ small enough (in a possibly smaller neighborhood, $0 < \ep < \bar{\ep}_0 < \ep_0$) there exist unstable eigenvalues $\lambda(\ep)$ with $\Re \lambda(\ep) > 0$ of the spectral problem \eqref{probw} with $\theta = 0$, for some appropriate eigenfunctions $w$. If we let $\theta$ vary within $(-\pi,\pi]$ we obtain curves of spectrum that locally remain in the unstable half plane (see Figure \ref{figunstableev}). We conclude that  
\[
\sigma(\cL^\ep)_{|L^2} = \bigcup_{-\pi < \theta \leq \pi} \sigma(\cL_\theta)_{|\Ldper} \cap \{\lambda \in \C \, : \, \Re \lambda > 0\} \neq \varnothing,
\]
for $\ep > 0$ sufficiently small. This completes the proof of Theorem \ref{theosmalli}.
\qed

\begin{figure}[h]
\begin{center}
\includegraphics[scale=0.26, clip=true]{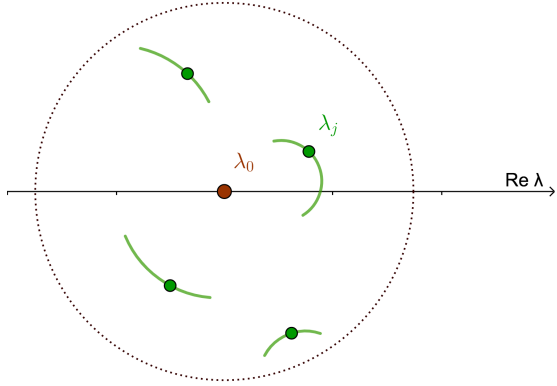}
\end{center}
\caption{\small{Cartoon representation of the unstable real eigenvalue $\lambda_0 = g'(0) > 0$ (in red) and of the neighboring unstable eigenvalues $\lambda_j(\vep)$ (in green) near $\lambda_0$ for $0 < \vep \ll 1$ small for the case of a Floquet exponent $\theta = 0$. By letting $\theta$ vary within $(-\pi,\pi]$ we obtain unstable curves of spectrum (in green) of the linearized operator around the periodic wave (color online).}}\label{figunstableev}
\end{figure}

\begin{remark}
It is to be noticed that formula \eqref{aproxev} readily implies spectral instability in view of the sign of $g'(0)$. Thus, the instablility of $u=0$ as equilibrium point of the reaction function (in the sense that $u=0$ is a local maximum of the potential $\int^u g(s) \, ds$))  is responsible for the spectral instability of the small amplitude waves bifurcating from the equilibrium. Heuristically, this result can be interpreted as follows: when $\ep \to 0^+$ the small-amplitude periodic waves collapse to the origin and the linearized operator tends (formally) to a constant coefficient linearized operator around zero, whose spectrum is determined by a dispersion relation that invades the unstable half plane thanks to the sign of $g'(0)$. Notice as well that the positive sign of $g'(0)$ is also responsible of the existence of the periodic waves bifurcating from the origin. The dedicated reader may easily verify that there is no Hopf bifurcation when $g'(0)<0$, as no change of stability of the origin occurs when we vary $c$ in a neighborhood of $c_0$. 
\end{remark}

\section{Spectral instability of large period wavetrains}
\label{seclarge}

\subsection{Spectral instability of the traveling pulse}

Let us consider the traveling pulse solution to equation \eqref{vbl} from Theorem \ref{thmexlarge} (or Corollary \ref{corpulse}),
\[
u(x,t) = \varphi^0(x-c_1t), \qquad x \in \R, \;\; t > 0,
\]
traveling with speed $c_1 = I_1/I_0$. Denoting as before the Galilean variable of translation as $z = x - c_1 t$, let us consider a solution to \eqref{vbl} of the form $\varphi^0(z) + e^{\bar{\lambda} t} w(z)$ with some $w \in H^2(\R;\C)$ and some $\bar{\lambda} \in \C$. Upon substitution and linearization we arrive at the following eigenvalue problem
\begin{equation}
\label{LbarR}
\begin{aligned}
\bar{\lambda} = \bar{\cL}^0 w &:= w_{zz} + (c_1 - f'(\varphi^0(z)) w_z + (g'(\varphi^0(z)) - f'(\varphi^0(z))_z) w,\\
\bar{\cL}^0 &: \, \cD(\bar{\cL}^0) = H^2(\R;\C) \subset L^2(\R;\C) \longrightarrow L^2(\R;\C).
\end{aligned}
\end{equation}

$\bar{\cL}^0$ is a closed, densely defined operator in $L^2(\R;\C)$, that is, on the whole real line. Moreover, $\bar{\cL}^0$ is of \emph{Sturmian type} (see, e.g., Kapitula and Promislow \cite{KaPro13}, section 2.3),
\[
\bar{\cL}^0 = \partial_z^2 + \bar{a}_1^0 \partial_z + \bar{a}_0^0 \, \Id,
\]
with smooth coefficients
\[
\begin{aligned}
 \bar{a}_1^0(z) &= c_1 -  f'(\varphi^0(z)),\\
 \bar{a}_0^0(z) &= g'(\varphi^0(z)) - f'(\varphi^0(z))_z,
\end{aligned}
\]
which decay exponentially to finite limits as $z \to \pm \infty$ in view of \eqref{expdecay}, more precisely,
\begin{equation}
\label{coeffexpd}
| \bar{a}_1^0(z) - \bar{a}_1^{\infty} | + | \bar{a}_0^0(z) - \bar{a}_0^{\infty} | \leq C e^{- \kappa |z|}, \qquad z \to \pm \infty,
\end{equation}
with,
\[
\bar{a}_1^{\infty} := c_1 - f'(1), \qquad \bar{a}_0^{\infty} := g'(1).
\]

The operator $\bar{\cL}^0$ is not self-adjoint but can be made self-adjoint under the $\omega$-inner product
\[
\< u,v \>_{L^2_\omega} := \int_{\R} u(z) v(z)^* \omega(z) \, dz,
\]
where the weight function $\omega(\cdot)$ is defined as
\[
\omega(z) := \exp \left( \int_0^z \bar{a}_1^0(s) \, ds \right),
\]
and has finite asymptotic values $\omega_{\pm} := \lim_{z \to \pm \infty} e^{- \bar{a}_1^{\infty} z} \omega(z)$.


The instability of the traveling pulse is therefore a direct consequence of standard Sturm-Liouville theory.

\begin{theorem}
\label{theoinspulse}
The traveling pulse solution is spectrally unstable, more precisely, there exists $\bar{\lambda}_0 > 0$ (real and strictly positive) such that $\bar{\lambda}_0 \in \ptsp(\bar{\cL}^0)$. Moreover, this eigenvalue is simple.
\end{theorem}
\begin{proof}
Since $\bar{\cL}^0 : L^2 \to L^2$ is of Sturmian type and its coefficients satisfy \eqref{coeffexpd} we can apply Theorem 2.3.3 in \cite{KaPro13}, p. 33, to conclude that the point spectrum of $\bar{\cL}^0$ consists of a finite number of simple real eigenvalues which can be enumerated in a strictly decreasing order
\[
\bar{\lambda}_0 > \bar{\lambda}_1 > \ldots > \bar{\lambda}_N > \bar{a}_0^{\infty},
\]
with $N \in \N$, and for any $j = 1, \ldots, N$, the eigenfunction $q_j \in H^2$ associated to $\bar{\lambda}_j$ can be normalized such that $q_j$ has exactly $j$ zeroes. Moreover, the ground state eigenvalue $\bar{\lambda}_0$ is determined by
\[
\bar{\lambda}_0 = \sup_{\| u \|_{L^2_\omega} =1} \< \bar{\cL}^0 u, u \>_{L^2_\omega},
\]
where the supremum is achieved precisely at $u = q_0$, which has no zeroes.

Now, it is to be observed that $\lambda = 0$ belongs to $\ptsp(\bar{\cL}^0)$ because the derivative of $\varphi^0$ is the associated eigenfunction. Indeed, equation \eqref{psi3} (with $\psi = \varphi^0$, the traveling pulse) is equivalent to
\[
\bar{\cL}^0 (\partial_z \varphi^0) = \partial_z^3 \varphi^0 +  \bar{a}_1^0(z) \partial_z^2 \varphi^0 +  \bar{a}_0^0(z) \partial_z \varphi^0 = 0.
\]
Moreover, since $\varphi^0 \in C^3(\R)$ and by exponential decay, it is clear that $\partial_z \varphi^0 \in H^2(\R;\C)$. Thus, $\lambda = 0 \in \ptsp(\bar{\cL}^0)$ with associated eigenfunction $\partial_z \varphi^0$.

Notice, however, that from the phase plane construction $\partial_z \varphi^0$ has exactly one zero (located at $(u_*,0)$ in the phase plane; see Theorem \ref{theoexisthomo}). Hence we deduce that $\bar{\lambda}_1 = 0$ is the second largest eigenvalue, associated to the eigenfuction $q_1 = \alpha_1 \partial_z \varphi^0$ (where $\alpha_1 \neq 0$ is a normalizing constant), which has exactly one zero. Therefore, there exists one positive eigenvalue $\bar{\lambda}_0 > 0$, the ground state, with eigenfunction $q_0 \in H^2$, which has no zeroes.
\end{proof}

The spectral problem for the traveling pulse \eqref{LbarR} can be recast as a first order system on the whole real line of the form
\begin{equation}
\label{firstordA0}
W_z = \A^0(z,\lambda) W,
\end{equation}
where
\begin{equation}
\label{A0coef}
\A^0(z,\lambda) := \begin{pmatrix} 0 & 1 \\ \lambda - \bar{a}_0^0(z) & - \bar{a}_1^0(z) \end{pmatrix}
= \begin{pmatrix} 0 & 1 \\ \lambda - (g'(\varphi^0)-f'(\varphi^0)_z) & - c_1 + f'(\varphi^0) \end{pmatrix}.
\end{equation}
These coefficients are clearly analytic in $\lambda \in \C$ and of class $C^1(\R;\C^{2 \times 2})$ as functions of $z \in \R$. Moreover, they have asymptotic limits given by
\begin{equation}
\label{A0coefasymp}
\A^0_\infty(\lambda) := \lim_{z \to \pm \infty} \A^0(z,\lambda) = \begin{pmatrix} 0 & 1 \\ \lambda - \bar{a}_0^{\infty} & - \bar{a}_1^{\infty} \end{pmatrix} = \begin{pmatrix} 0 & 1 \\ \lambda - g'(1) & - c_1 + f'(1) \end{pmatrix}.
\end{equation}
Thanks to exponential decay \eqref{expdecay} of the homoclinic orbit, we have
\[
|\varphi^0(z) - 1| + |(\varphi^0)'(z)| \leq C e^{- \kappa |z|}, \qquad z \in \R.
\] 
Therefore, from continuity of the coefficients and for any $|\lambda| \leq M$, with some $M > 0$, there exists a constant $C(M) > 0$ such that
\begin{equation}
\label{laH2K}
| \A^0(z,\lambda) - \A^0_\infty(\lambda)| \leq C(M) e^{-\kappa |z|},
\end{equation}
for all $z \in \R$. Hence,
\begin{equation}
\label{defOmegai}
\Omega_{\infty} = \{ \lambda \in \C \, : \, \Re \lambda > g'(1) \},
\end{equation}
is an open connected subset in the complex plane. Then, from assumption \eqref{H2}, it is clear that the unstable half plane, namely $\C_+ := \{ \lambda \in \C \, : \, \Re \lambda > 0\}$, is properly contained in $\Omega_{\infty}$. Moreover, it is easy to verify that for every $\lambda \in \Omega_{\infty}$ the coefficient matrix $\A^0_{\infty}(\lambda)$ has no center eigenspace and that its stable, $\bbS^0_{\infty}(\lambda)$, and unstable, $\U^0_{\infty}(\lambda)$, eigenspaces satisfy
\[
\dim \U^0_{\infty}(\lambda) = \dim \bbS^0_{\infty}(\lambda) = 1, \qquad \text{for all } \; \lambda \in \Omega_{\infty}.
\]
$\Omega_{\infty}$ is called the set of consistent splitting (or domain of hyperbolicity) of $\A^0_{\infty}(\lambda)$. 

One can define the family of operators 
\[
\cT^0(\lambda) : L^2(\R;\C) \times L^2(\R;\C) \longrightarrow L^2(\R;\C) \times L^2(\R;\C),
\]
parametrized by $\lambda \in \C$, densely defined with domain $\cD(\cT^0) = H^2(\R;\C) \times H^1(\R;\C)$ and given by
\[
\cT^0(\lambda) = \frac{d}{dz} - \A^0(z,\lambda).
\]
It is well-known (cf. \cite{San02,KaPro13}) that $\ptsp(\bar{\cL}^0)_{|L^2}$ coincides with the set of complex numbers $\lambda \in \C$ such that $\cT^0(\lambda)$ is a Fredholm operator with index equal to zero. Therefore, from Theorem \ref{theoinspulse} we reckon the existence of an unstable real and simple eigenvalue, $\bar{\lambda}_0 > 0$, for which there exists a bounded solution
\[
W_0 = \begin{pmatrix} q_0 \\ \partial_z q_0
\end{pmatrix} \in H^2(\R;\C) \times H^1(\R;\C),
\]
to the equation
\[
\cT^0(\bar{\lambda}_0) W_0 = \partial_z W_0 - \A^0(\bar{\lambda}_0,z) W_0 = 0,
\]
for all $z \in \R$. As a corollary of Theorem \ref{theoinspulse}, the homoclinic Evans function associated to the traveling pulse is non-vanishing in the open set $\Omega_\infty$, except for a single, real, unstable and simple zero at $\lambda = \bar{\lambda}_0 > 0$. More precisely,
\[
\begin{aligned}
D^0(\lambda) &\neq 0, \quad \text{for all } \;\; \lambda \in \Omega_\infty \backslash \{\bar{\lambda}_0\},\\
D^0(\bar{\lambda}_0) &= 0,  \quad \frac{d D^0}{d \lambda} (\bar{\lambda}_0) \neq 0,
\end{aligned}
\]
where $D^0 = D^0(\lambda)$ denotes the homoclinic Evans funcion defined in \eqref{defhomEv} for the traveling pulse $\varphi^0 = \varphi^0(z)$.

\subsection{Approximation theorem for large spatial period}

In order to establish the spectral instability of the large period waves from Theorem \ref{thmexlarge}, we need to verify that the family of waves satisfies the structural assumptions of the seminal result of Gardner \cite{Grd2} on convergence of spectra of periodic traveling waves in the infinite-period (homoclinic) limit to the isolated point spectrum of the underlying homoclinic orbit. The proof of Gardner is of topological nature. We thus refer to the (more analytical) works of Sandstede and Scheel \cite{SS6} and Yang and Zumbrun \cite{YngZ19}.

Under assumptions \eqref{H1} - \eqref{H3}, \eqref{H5} and \eqref{H6}, consider the family of periodic traveling waves from Theorem \ref{thmexlarge},
\[
\begin{aligned}
u(x,t) &= \varphi^\ep (x - c(\ep)t), \\
 \varphi^\ep(z) &= \varphi^\ep( z + T_\ep), \qquad \text{for all } \; z \in \R,
\end{aligned}
\]
traveling with speed $c = c(\ep)$ and parametrized by $\ep = | c_1 - c(\ep)|$, with $0 < \ep < {\ep}_1 \ll 1$ sufficiently small. The family converges as $\ep \to 0^+$ to the solitary wave (traveling pulse) solution $\varphi^0(x - c_1 t)$ traveling with speed $c_1 = I_1/I_0$, which is associated to a homoclinic orbit for system \eqref{firstos} with $c = c_1$. The fundamental period of the family of periodic waves, $T_\ep$, converges to $\infty$ as $\ep \to 0^+$ at order $O(|\log \ep|)$.

From the discussion in Section \S \ref{secspe}, we know that the spectral problem for each member of the family $\varphi^\ep$, $0 < \ep < {\ep}_1$, can be written as a first order system of the form \eqref{firstordersyst}
\begin{equation}
\label{firstordAe}
W_z = \A^\ep(z,\lambda) W,
\end{equation}
where the coefficients,
\begin{equation}
\label{Aecoef}
\A^\ep(z,\lambda) = \begin{pmatrix} 0 & 1 \\ \lambda - \bar{a}_0^{\ep}(z) & - \bar{a}_1^{\ep}(z)\end{pmatrix},
\end{equation}
are analytic in $\lambda \in \C$, continuous in $\ep > 0$ and of class $C^1(\R;\C^{2 \times 2})$ as functions of $z \in \R$. Here, the scalar coefficients,
\[
\begin{aligned}
 \bar{a}_1^{\ep}(z) &:= c(\ep) -  f'(\varphi^{\ep}(z)),\\
 \bar{a}_0^{\ep}(z) &:= g'(\varphi^{\ep}(z)) - f'(\varphi^{\ep}(z))_z,
\end{aligned}
\]
are bounded, sufficiently smooth functions of $z$. The the family of operators 
\[
\left\{
\begin{aligned}
\cT^\ep(\lambda) &: L^2(\R;\C) \times L^2(\R;\C) \longrightarrow L^2(\R;\C) \times L^2(\R;\C),\\
\cT^\ep(\lambda) &= \partial_z - \A^\ep(z,\lambda),
\end{aligned}
\right.
\]
parametrized by $\lambda \in \C$, has the property that its spectrum is purely essential. Notice as well that
\[
\begin{aligned}
\A^\ep(z,\lambda) &- \A^0(z,\lambda) = \begin{pmatrix} 
0 & 0 \\ \bar{a}_0^0(z) - \bar{a}_0^{\ep}(z) &  \bar{a}_1^0(z) - \bar{a}_1^{\ep}(z)
\end{pmatrix} \\
&= \begin{pmatrix} 
0 & 0 \\ g'(\varphi^0) - g'(\varphi^\ep) - f''(\varphi^0) (\varphi^0)' + f''(\varphi^\ep)(\varphi^\ep)' & c_1 - c(\ep) - f'(\varphi^0) + f'(\varphi^\ep)
\end{pmatrix}.
\end{aligned}
\]
Hence, since the coefficients are smooth and bounded and from estimates \eqref{boundw1} and \eqref{boundc1} we have, for $|\lambda |\leq M$,
\[
\begin{aligned}
|\A^\ep(z, \lambda) - \A^0(z,\lambda) | &\leq \overline{C}(M) \Big( |\varphi^0(z) - \varphi^\ep(z)| + | (\varphi^0)'(z) - (\varphi^\ep)'(z)| + |c_1 - c(\ep)| \Big)\\
&\leq C(M) e^{- \kappa T_\ep/2}.
\end{aligned}
\]

Consequently, from the estimate above, Theorem \ref{thmexlarge} and \eqref{laH2K} we conclude that, for every $|\lambda| \leq M$, there holds
\begin{equation}
\label{hypsYZ}
\begin{aligned}
T_\ep = O(|\log \ep|) &\to \infty, \quad \text{as } \, \ep \to 0^+,\\
|\A^0(z,\lambda) - \A^0_{\infty}| &\leq C(M) e^{- \bar{\theta} |z|}, \quad \text{for all } \, z \in \R,\\
|\A^0(z,\lambda) - \A^\ep(z,\lambda)| &\leq C(M) e^{-\kappa T_\ep/2}, \quad \text{for all } \, |z| \leq \frac{T_\ep}{2},
\end{aligned}
\end{equation}
for some uniform constants $C(M), \kappa > 0$. Here $\bar{\theta} = \kappa$ in view of \eqref{laH2K}.

Conditions \eqref{hypsYZ} are the structural assumptions (H1) - (H3) in \cite{YngZ19} (p. 30). Thus we have the following
\begin{theorem}[Gardner \cite{Grd2}; Yang and Zumbrun \cite{YngZ19}]
\label{thmYZall}
Assume \eqref{hypsYZ}. Then on a compact set $K \subset \Omega_{\infty}$ such that the homoclinic Evans function $D^0 = D^0(\lambda)$ does not vanish on $\partial K$, the spectra of $\cL^\ep$ for $T_\ep$ sufficiently large (or equivalently, for any $0 < \ep < \ep_2$ with $0 < \ep_2 \ll 1$ sufficiently small) consists of loops of spectra $\Lambda_{k,j}^\ep \subset \C$, $k = 1, \ldots, m_j$, in a neighborhood of order $O(e^{-\eta T_\ep/(2m_j)})$ of the eigenvalues $\lambda_j$ of $\bar{\cL}^0$, where $m_j$ denotes the algebraic multiplicity of $\lambda_j$ and $0 < \eta < \min \{\kappa, \bar{\theta}\}$.
\end{theorem}
\begin{proof}
See Corollary 4.1 and Proposition 4.2 in \cite{YngZ19}.
\end{proof}

\begin{remark}
The conclusion of Theorem \ref{thmYZall} is a refinement of the classical Gardner's result (Theorem 1.2 in \cite{Grd2}) due to Yang and Zumbrun \cite{YngZ19}, who prove the convergence of the periodic Evans function, $D^\epsilon(\lambda,\theta)$, associated to the periodic waves for each value of $\epsilon$ to the corresponding homoclinic Evans function $D^0(\lambda)$, as $\epsilon \to 0^+$. For that purpose, they rescale the periodic Evans function as a Jost-function type determinant, involving the difference of two matrix-valued functions (see also \cite{Z16}). The third equation in \eqref{hypsYZ} (exponential bound) is an additional hypothesis to those of Gardner, but it holds true in many situations (like ours) where the vertex of the homoclinic loop is a hyperbolic rest point of the traveling wave ODE, under the (typically true) transversality condition regarding the associated Melnikov separation function with full rank with respect to the bifurcation parameter (for an extensive discussion on this issue, see Sandstede and Scheel \cite{SS6}, Proposition 5.1 and hypotheses (G1) and (G2)).
\end{remark}

\subsection*{Proof of Theorem \ref{theolargei}}

Under assumptions \eqref{H1}, \eqref{H2}, \eqref{H3}, \eqref{H5} and \eqref{H6}, it is clear that the family of periodic waves with large period, $\varphi^\ep$, as well as the traveling pulse $\varphi^0$ from Theorem \ref{thmexlarge}, satisfy hypotheses \eqref{hypsYZ}. Let $\bar{\lambda}_0 > 0$ be the real, simple and positive (homoclinic) eigenvalue of the linearized operator, $\bar{\cL}^0$, around the traveling pulse (see Theorem \ref{theoinspulse}). Since $\C_+ \subset \Omega_\infty$ and $\bar{\lambda}_0 > 0$ is an isolated eigenvalue, then we can take a closed contour $\Gamma$ around $\bar{\lambda}_0$ such that $K = \overline{\Gamma} \cup (\text{int} \, \Gamma)$ is a small compact set contained in $\Omega_\infty$ with no eigenvalues of $\bar{\cL}^0$ on $\partial K = \Gamma$. Then from Theorem \ref{thmYZall} we conclude that there exists $\bar{\ep}_1:= \min \{\ep_1,\ep_2\} > 0$ sufficiently small such that for all $0 < \ep < \bar{\ep}_1$ there exists a loop of spectrum $\Lambda^\ep \subset \C$ in a small neighborhood around $\bar{\lambda}_0$ of order $O(e^{-\kappa T_\ep/2}) = O(\ep)$ of eigenvalues of the linearized operator $\cL^\ep$ around $\varphi^\ep$. Moreover, since the unstable homoclinic eigenvalue $\bar{\lambda}_0$ is simple, then for each $0 < \ep < \bar{\ep}_1$ there exists one single closed loop of spectrum $\Lambda^\ep$. This loop does not necessarily contain $\bar{\lambda}_0$ but belongs to a $O(\ep)$-neighborhood of it (see Figure \ref{figunstableev-homo}). Hence, we conclude that the spectrum of the linearized operator $\cL^\ep$ around each periodic wave $\varphi^\ep$ with $0 < \ep < \bar{\ep}_1$ is contained in the unstable half plane. The theorem is proved.
\qed

\begin{figure}[h]
\begin{center}
\includegraphics[scale=0.26, clip=true]{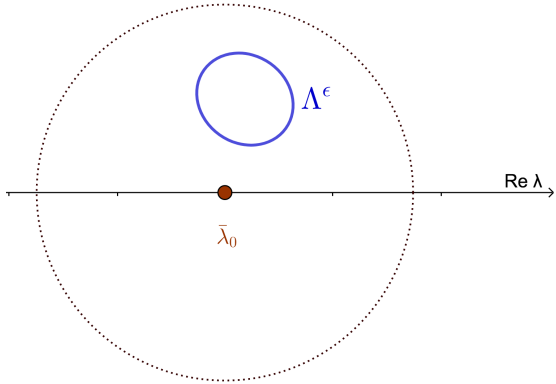}
\end{center}
\caption{\small{Cartoon representation of the unstable, simple, real eigenvalue, $\bar{\lambda}_0 > 0$ (in red), of the linearized operator $\bar{\cL}^0$ around the homoclinic loop. For $0 < \ep \ll 1$ sufficiently small there exists a unique loop of spectra, $\Lambda^\ep$ (in blue), of the linearized operator $\cL^\ep$ around the periodic wave inside an unstable $O(\ep)$-neighborhood of $\bar{\lambda}_0$ (color online).}}\label{figunstableev-homo}
\end{figure}

\section{Examples}
\label{secexa}

In this section we present some examples of viscous balance laws of the from \eqref{gvbl} which satisfy the hypotheses of this paper.

%
%
%

\subsection{Burgers-Fisher equation}

The Burgers-Fisher equation (cf. \cite{LeHa16,MiGu02,Vall18b}),
\begin{equation}
\label{eqBF}
u_t + uu_x = u_{xx} + u(1-u), \qquad x \in \R, \; t >0,
\end{equation}
is a viscous balance law of the form \eqref{vbl} where the nonlinear flux function is given by the classical Burgers' flux \cite{Bur48,La57},
\begin{equation}
\label{BFf}
f(u) = \frac{1}{2} u^2,
\end{equation}
and the reaction term is the well-known Fisher-KPP \cite{Fis37,KPP37} (or logistic) function,
\begin{equation}
\label{BFg}
g(u) = u(1-u).
\end{equation}

The Burgers-Fisher equation is perhaps the simplest scalar model combining nonlinear advection effects together with viscosity and a production rate of logistic type. It is the paradigm of a scalar viscous balance law. 
Clearly, $f$ and $g$ satisfy assumptions \eqref{H1} and \eqref{H2}. It is also easy to verify that there exists a unique value $u_* = -1/2$ such that
\[
\int_{u_*}^1 g(s) \, ds = \int_{-\tfrac{1}{2}}^1 s(1-s) \, ds = 0,
\]
and assumption \eqref{H3} is also satisfied. Moreover, since $g'(u) = 1-2u$, $g''(u) = -2$, $f'(u) = u$, $f''(u) = 1$ for all $u$, the parameter $\overline{a}_0$ in \eqref{H4} is 
\[
\overline{a}_0 = - \frac{f''(0) g''(0)}{\sqrt{g'(0)}} = 2 > 0,
\]
verifiying, in this fashion, the genericity condition \eqref{H4}. Notice that, since $\overline{a}_0 > 0$ and $c_0 = f'(0) = 0$, then Theorem \ref{theosmallw} guarantees the emergence of a family of small-amplitude periodic waves for each speed value $c \in (0, \ep_0)$ with $0 <\ep_0 \ll 1$ sufficiently small. The periodic orbits are unstable with respect to the dynamical system \eqref{firstos} with $c = c(\ep)$ and this corresponds to a \emph{subcritical} Hopf bifurcation. Their fundamental period is $T = 2\pi + O(c)$, for $c \sim 0^+$. The emergence of small-amplitude waves for the Burgers-Fisher equation \eqref{eqBF} is illustrated in Figure \ref{figHopfBF}. Figure \ref{figHopfBFneg} shows the phase portrait of system \eqref{firstos} for the speed value $c = -0.05$; the origin is a repulsive node and all nearby solutions (in light green color) move away from it. Figure \ref{figHopfBF0} shows the case when $c = 0$, the parameter value where the subcritical Hopf bifurcation occurs; the origin is a center and solutions move away if they start far enough from the origin and locally rotate around a linearized center otherwise. Figure \ref{figHopfBFpos} shows the case where $c = 0.005$: the orbit in red is a numerical approximation of the unique small amplitude periodic wave for this speed value, the origin is an attractive node and nearby solutions inside the periodic orbit approach zero, whereas solutions outside the periodic orbit move away since the orbit is unstable as a solution to system \eqref{firstos}. Panel \ref{figHopfBFwave} shows the graph (in red) of the periodic wave $\varphi$ as a function of the Galilean variable $z = x -ct$. This is a family of spectrally unstable periodic waves according to Theorem \ref{theosmalli}.

\begin{figure}[ht]
\begin{center}
\subfigure[$c = -0.05$]{\label{figHopfBFneg}\includegraphics[scale=.45, clip=true]{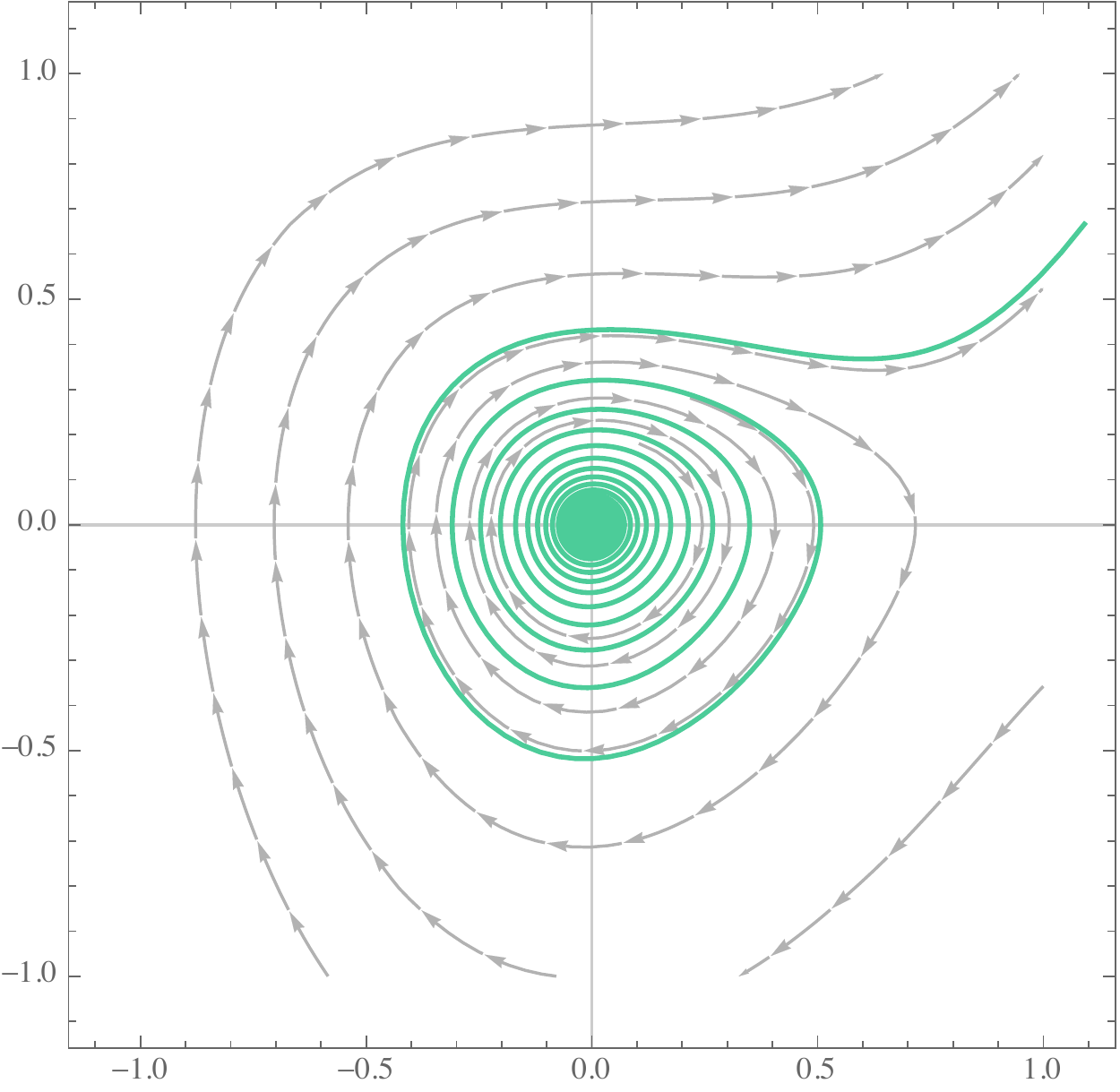}}
\subfigure[$c=0$]{\label{figHopfBF0}\includegraphics[scale=.45, clip=true]{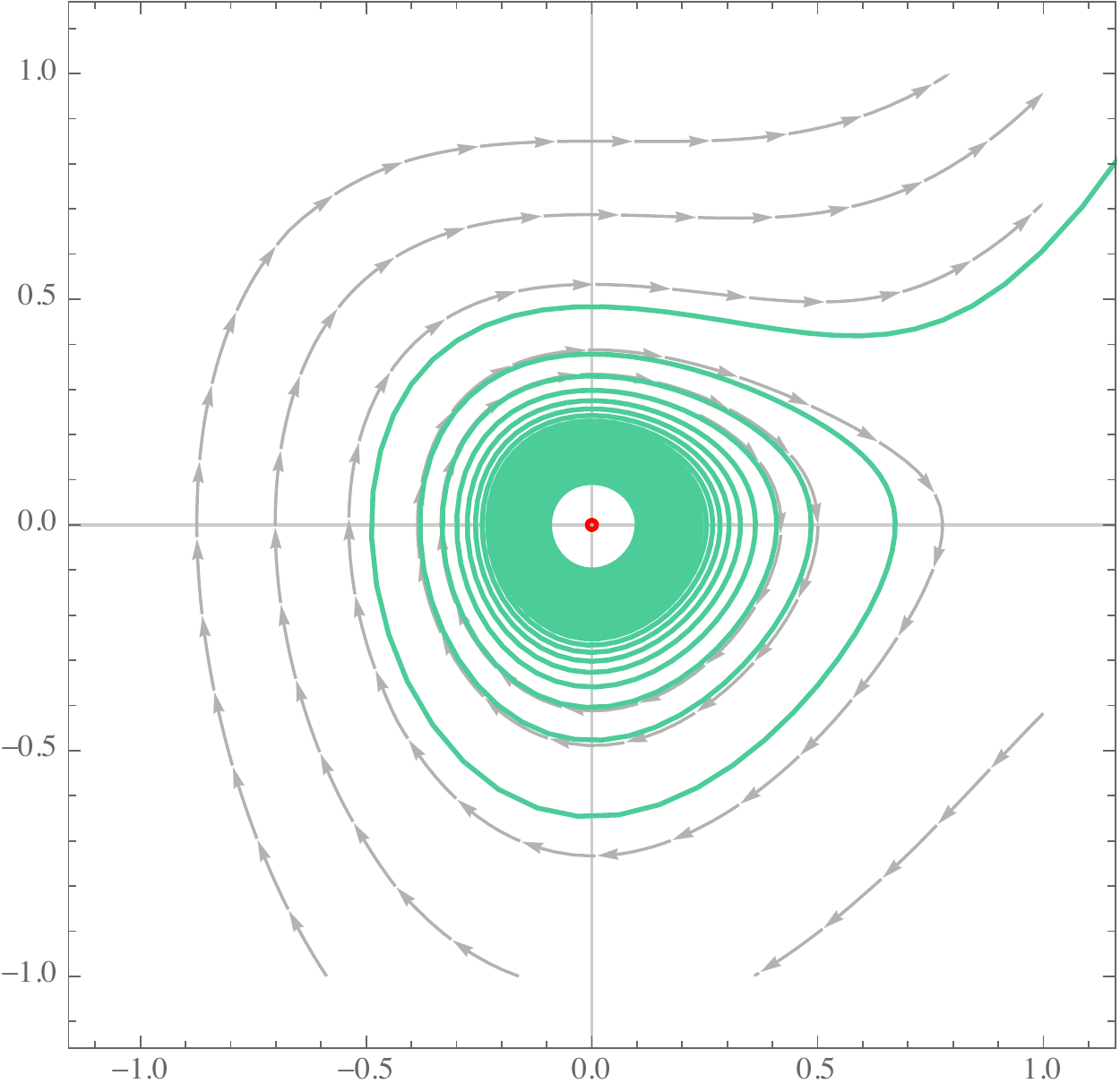}}
\subfigure[$c= 0.005$]{\label{figHopfBFpos}\includegraphics[scale=.45, clip=true]{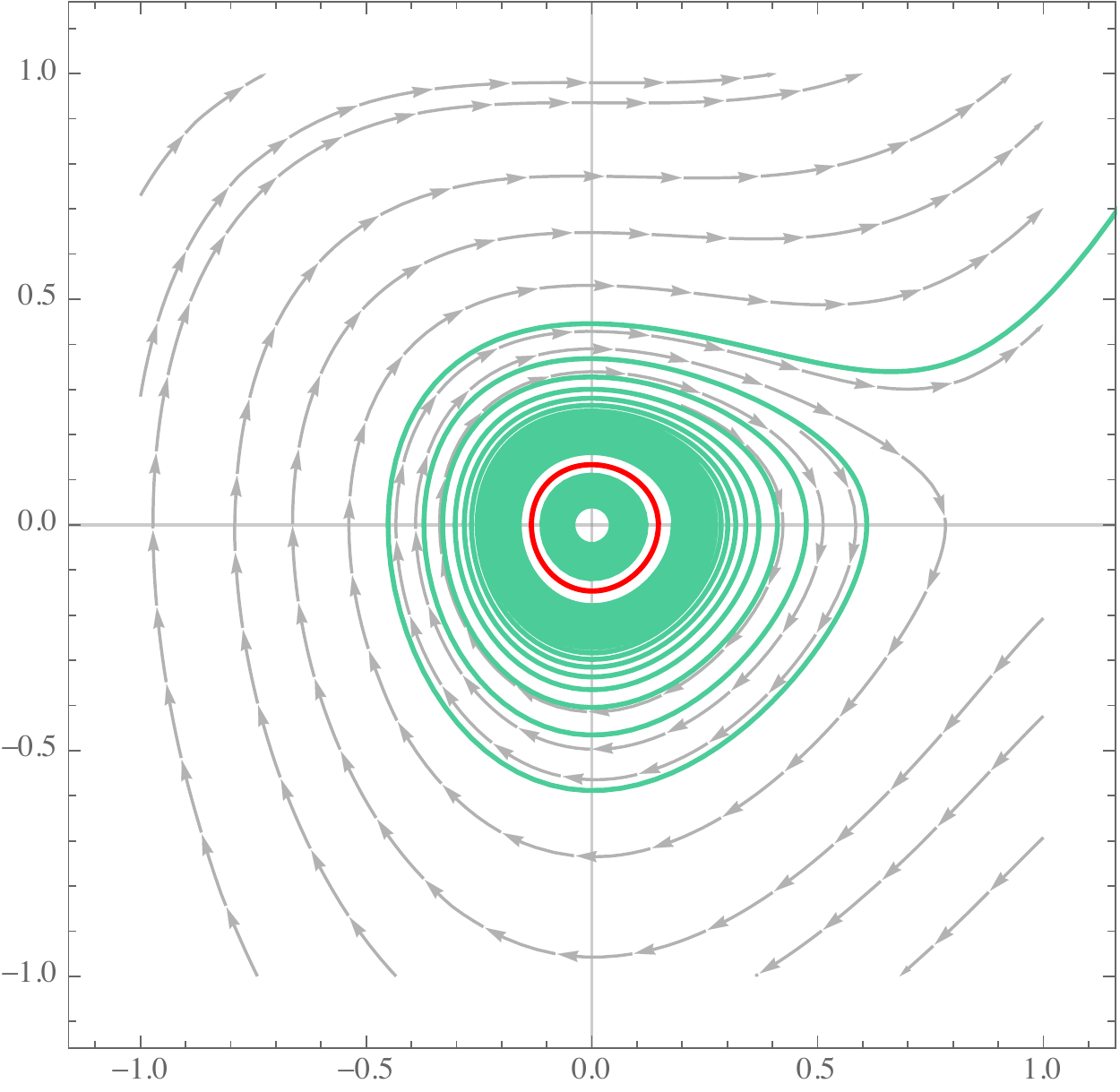}}
\subfigure[$c= 0.005$]{\label{figHopfBFwave}\includegraphics[scale=.45, clip=true]{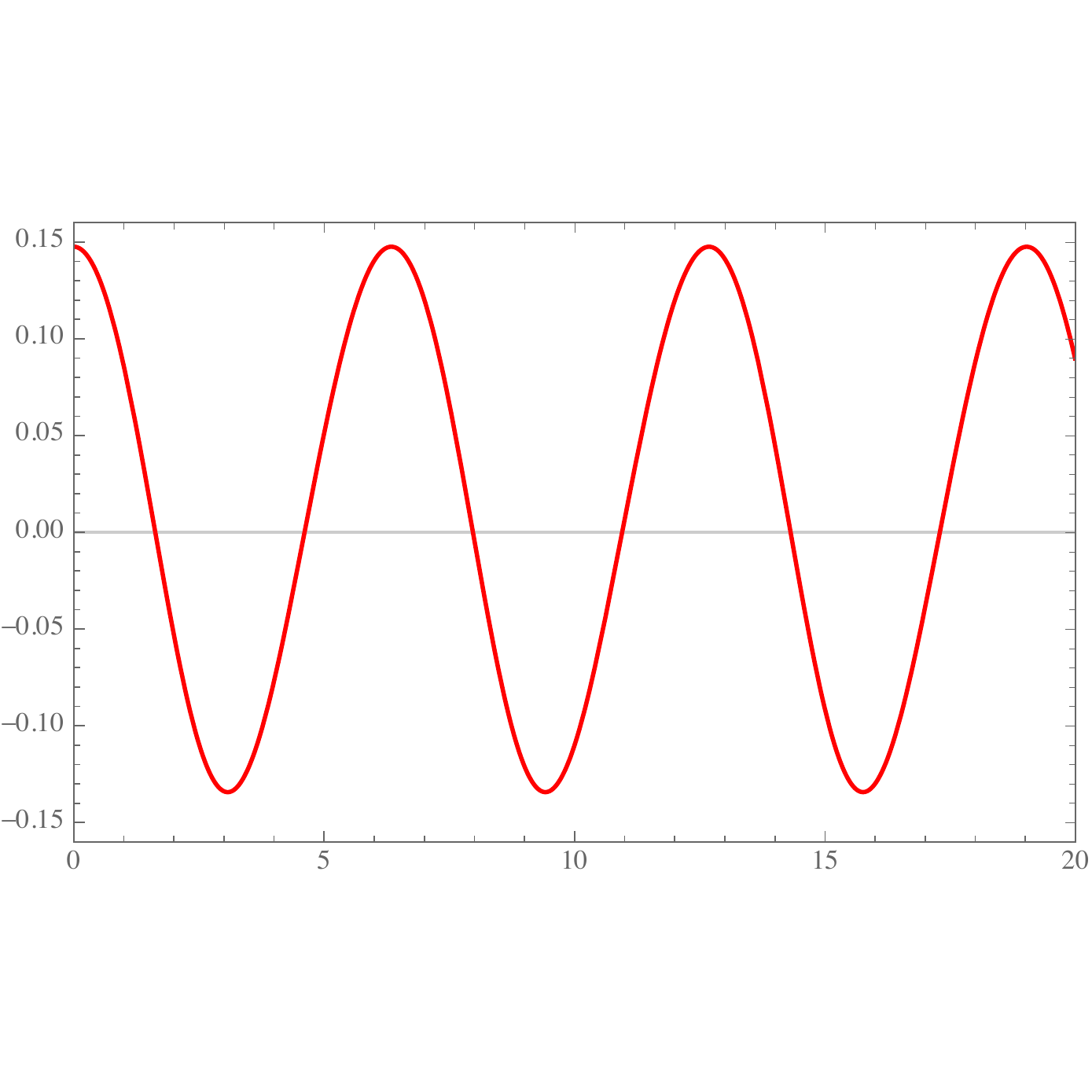}}
\end{center}
\caption{\small{Emergence of small-amplitude waves for the Burgers-Fisher equation \eqref{eqBF}. Panel (a) shows the phase portrait of system \eqref{firstos} for the speed value $c = -0.05$. Numerical solutions are shown in light green color. Panel (b) shows the case when $c = 0$, the parameter value where a subcritical Hopf bifurcation occurs. Panel (c) shows the case where $c = 0.005$: the orbit in red is a numerical approximation of the unique small amplitude periodic wave for this speed value. Panel (d) shows the graph (in red) of the approximated periodic wave $\varphi$ as a function of the Galilean variable $z = x -ct$ (color online).}}\label{figHopfBF}
\end{figure}

For the logistic reaction function \eqref{BFg} under consideration, the function $\gamma = \gamma(u)$ defined in \eqref{defPsi} is given by
\begin{equation}
\label{gamFisher}
\gamma(u) = \frac{1}{\sqrt{3}} \sqrt{1-3u^2 + 2 u^3}, \qquad u \in (-\tfrac{1}{2},1),
\end{equation}
whereupon we can compute the integrals defined in \eqref{lasIs}. For instance, $I_0$ and $I_1$ can be computed exactly:
\begin{equation}
\label{I0Fisher}
I_0 = \int_{-\tfrac{1}{2}}^1 \gamma(s) \, ds = \frac{1}{\sqrt{3}} \int_{-\tfrac{1}{2}}^1 \sqrt{1-3s^2+2s^3} \; ds 
= \frac{3}{5},
\end{equation}
\[
I_1 = \int_{-\tfrac{1}{2}}^1 f'(s) \gamma(s) \, ds = \frac{1}{\sqrt{3}}  \int_{-\tfrac{1}{2}}^1 s \sqrt{1-3s^2+2s^3} \, ds = \frac{3}{35}.
\]
The values of $L$ and $J$ are given by the following elliptic integrals, which are approximated numerically,
\begin{equation}
\label{LFisher}
L = 2 \int_{-\tfrac{1}{2}}^1 \sqrt{\frac{1-4s^3+3s^4}{1-3s^2+2s^3}} \,\; ds \approx 4.07339,
\end{equation}
\[
J = 2 \int_{-\tfrac{1}{2}}^1 s \, \sqrt{\frac{1-4s^3+3s^4}{1-3s^2+2s^3}} \,\; ds \approx 0.69062.
\]
Thus, the non-degeneracy condition \eqref{H5} holds as $I_0 J \approx 0.415237 \neq L I_1 \approx 0.349148$. Notice that the value of the speed of the homoclinic orbit from which the periodic loops with large period bifurcate is
\[
c_1 = \frac{I_1}{I_0} = \frac{1}{7}.
\]
This shows, for instance, that the saddle condition \eqref{H6} holds, inasmuch as $f'(1) = 1 \neq c_1$. Therefore, the Burgers-Fisher equation satisfies hypotheses \eqref{H1} thru \eqref{H6} of this paper.

Finally observe that, since $f'(1) =1 > c_1 = 1/7$, then from Theorem \ref{theoexistlp} the family of periodic waves with large period emerge for speed values in a neighborhood \emph{below} the value $c_1 = 1/7$, that is, for $c \in (\tfrac{1}{7} - \ep_1, \tfrac{1}{7})$ with $\ep_1 > 0$ small. Figure \ref{FigHomoBF} shows a numerical approximation of the homoclinic loop to system \eqref{firstos} with speed $c_1$ based on the saddle $P_1 = (1,0)$ (dashed line in blue), and a large-period wave from the family with speed $c \approx c_1 - 0.05$ (continuous line in orange). This is a family of spectrally unstable periodic waves in view of Theorem \ref{theolargei}.

\begin{figure}[t]
\begin{center}
\includegraphics[scale=.5, clip=true]{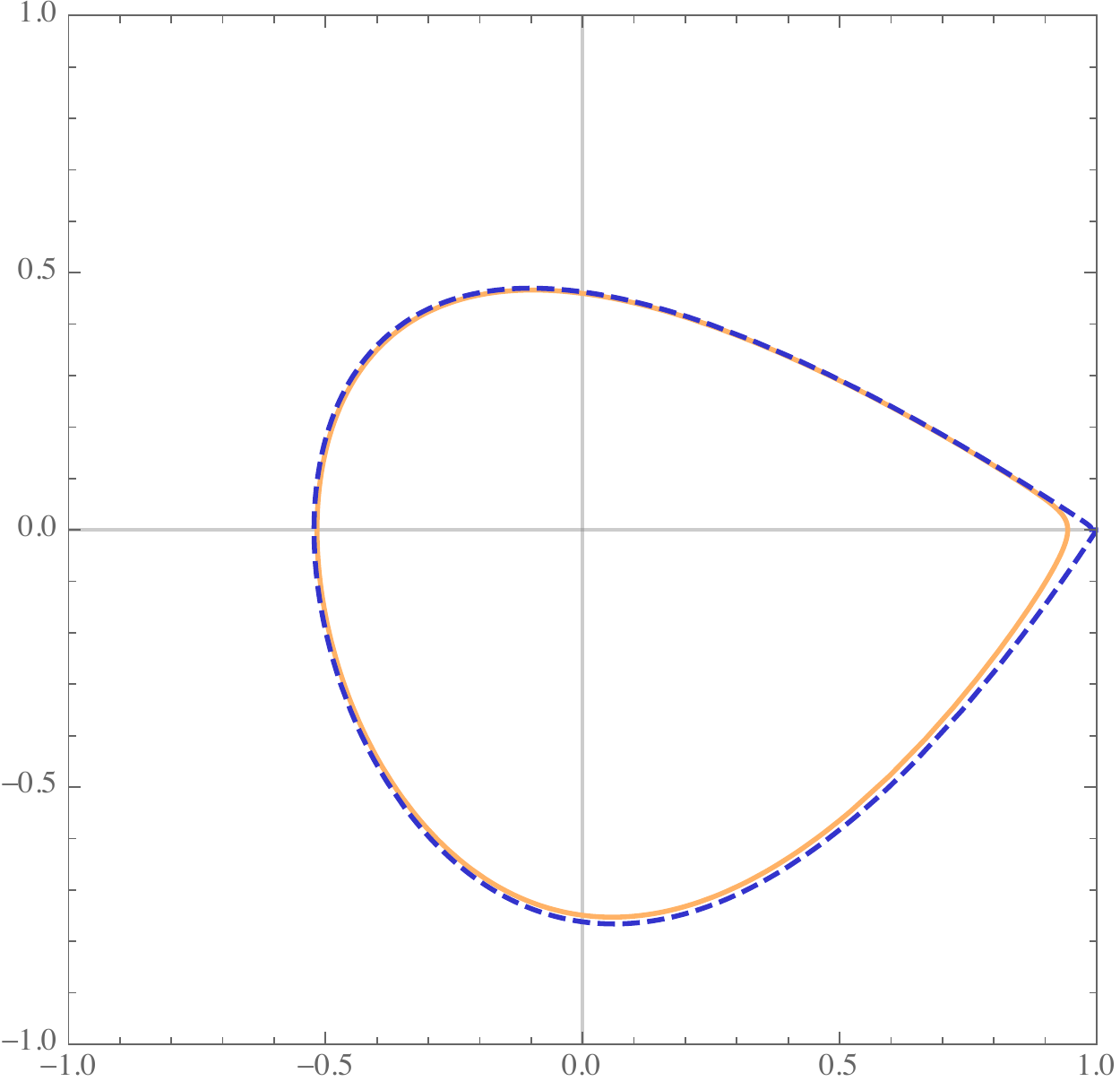}
\end{center}
\caption{\small{Numerical approximation of the homoclinic loop for the Burgers-Fisher equation \eqref{eqBF} with speed value $c_1 = 1/7$ (in blue, dashed line) and the periodic wave nearby with speed value $c_1  - \epsilon$, $\epsilon \approx 0.05$ (solid, orange line; color online).}}\label{FigHomoBF}
\end{figure}

\subsection{Logistic Buckley-Leverett model}
\label{secLogBL}

Consider the following viscous balance law
\begin{equation}
\label{LogBLmodel}
u_t + \partial_x \left( \frac{u^2}{u^2 + \tfrac{1}{2}(1-u)^2}\right) = u_{xx} + u(1-u), \qquad x \in \R, \; \; t >0.
\end{equation}
The nonlinear flux function is the well-known Buckley-Leverett function \cite{BuLe42},
\begin{equation}
\label{BLf}
f(u) = \frac{u^2}{u^2 + \tfrac{1}{2}(1-u)^2},
\end{equation}
which is a relatively simple scalar model that captures the main features of two phase fluid flow in a porous medium. Given that $f$ is not uniformly convex, it allows the emergence of non-classical wave solutions to the Riemann problem for the associated conservation law (see, e.g., \cite{LeF02,Lv02}). When applied to model oil recovery, the two phases correspond to pure oil ($u=0$) and pure water ($u = 1$). Hence, in typical applications the values of $u$ range in $[0,1]$ and, in addition, there is no production term. In this case, we allow values of $u \in \R$ in order to capture the emergence of periodic waves. The production term is, as in the previous example, the logistic reaction function \eqref{BFg}.





Clearly, the functions $f$ and $g$ in \eqref{BLf} and \eqref{BFg} satisfy assumptions \eqref{H1}, \eqref{H2} and \eqref{H3}. Moreover, $g'(u) = 1-2u$, $g''(u) = -2$ and computing the derivatives of $f$ yields
\[
\begin{aligned}
f'(u) &= \frac{u(1-u)}{\big(u^2 + \tfrac{1}{2}(1-u)^2\big)^2}, &  f''(u) = \frac{4(1-9u^2+6u^3)}{\big( 1-2u + 3u^2\big)^3},\\
f'''(u)  &= -\frac{24(-1+6u-18u^3+9u^4)}{\big( 1-2u + 3u^2\big)^4}.
\end{aligned}
\]
Whence the value of $\overline{a}_0$ is given by
\[
\overline{a}_0 = f'''(0) - \frac{f''(0) g''(0)}{\sqrt{g'(0)}} = 32,
\]
and the genericity condition \eqref{H4} holds. Since $\overline{a}_0 > 0$ then from Theorem \ref{theosmallw} we know there exist a family of small amplitude periodic waves for each speed value $c \in (0, \ep_0)$ for some small $0 < \ep_0 \ll 1$, because $c_0 = f'(0) = 0$ in this case, and this corresponds to a subcritical Hopf bifurcation in which the periodic waves are unstable as solutions to the dynamical system \eqref{firstos} with $c = c(\ep)$. Their fundamental period is approximately $2 \pi$ in view of formula \eqref{fundphopf}. Figure \ref{figHopfBL} shows the phase portraits of system \eqref{firstos} for equation \eqref{LogBLmodel} and different values of $c \sim 0$. Figure \ref{figHopfBLneg} shows the phase plane for $c = -0.05$, in which the origin is a repulsive node; Figure \ref{figHopfBL0} shows the case with the bifurcation value of the speed, $c= 0$; and Figure \ref{figHopfBLpos} shows the case with $c = 0.0025$ and the orbit in red color is a numerical approximation of the unique small amplitude periodic wave for this speed value, the origin is an attractive node and nearby solutions inside the periodic orbit approach zero. Panel \ref{figHopfBLwave} shows the graph (in red) of the periodic wave $\varphi$ as a function of the Galilean variable $z = x -ct$.

\begin{figure}[t]
\begin{center}
\subfigure[$c = -0.05$]{\label{figHopfBLneg}\includegraphics[scale=.45, clip=true]{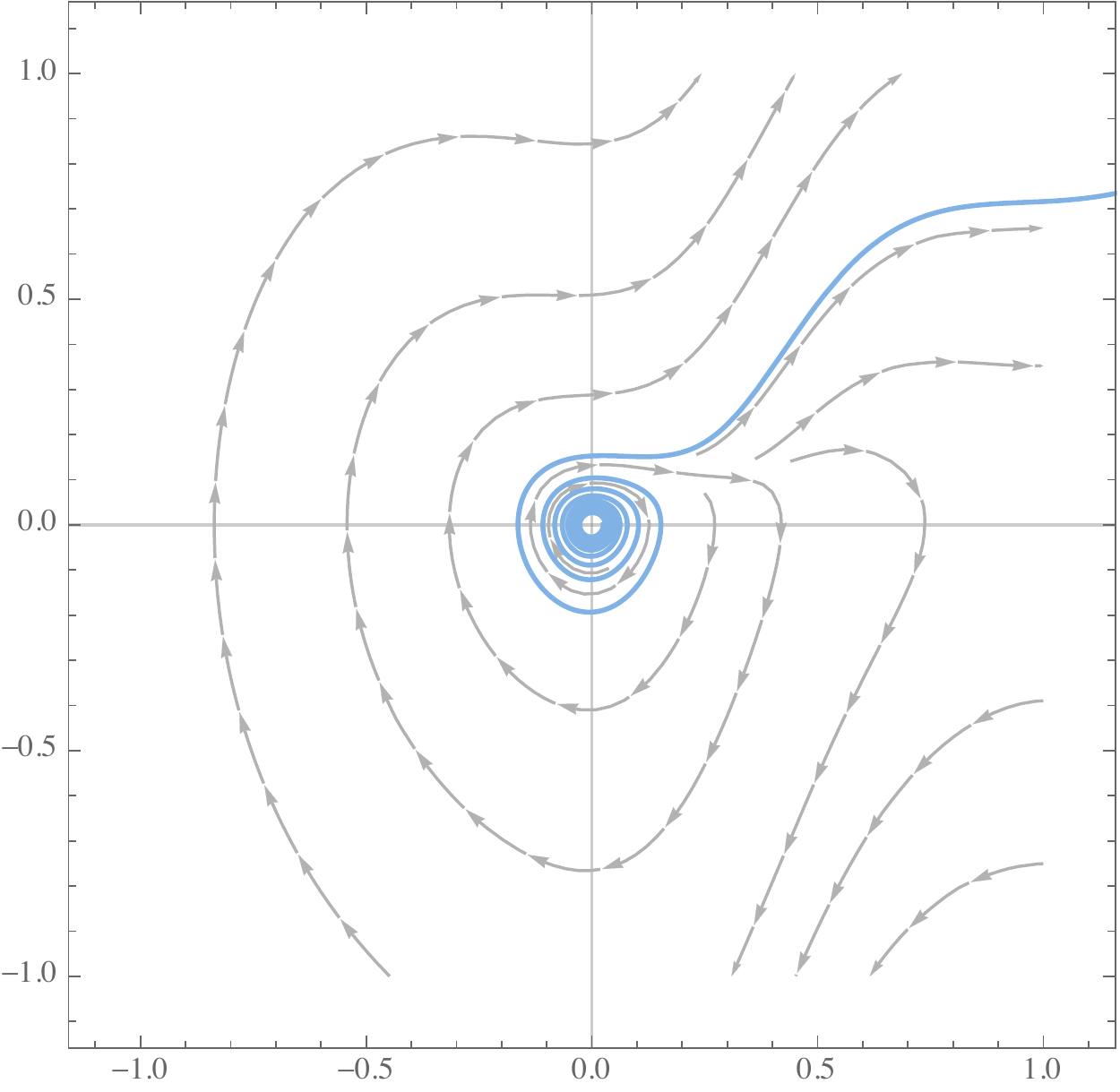}}
\subfigure[$c=0$]{\label{figHopfBL0}\includegraphics[scale=.45, clip=true]{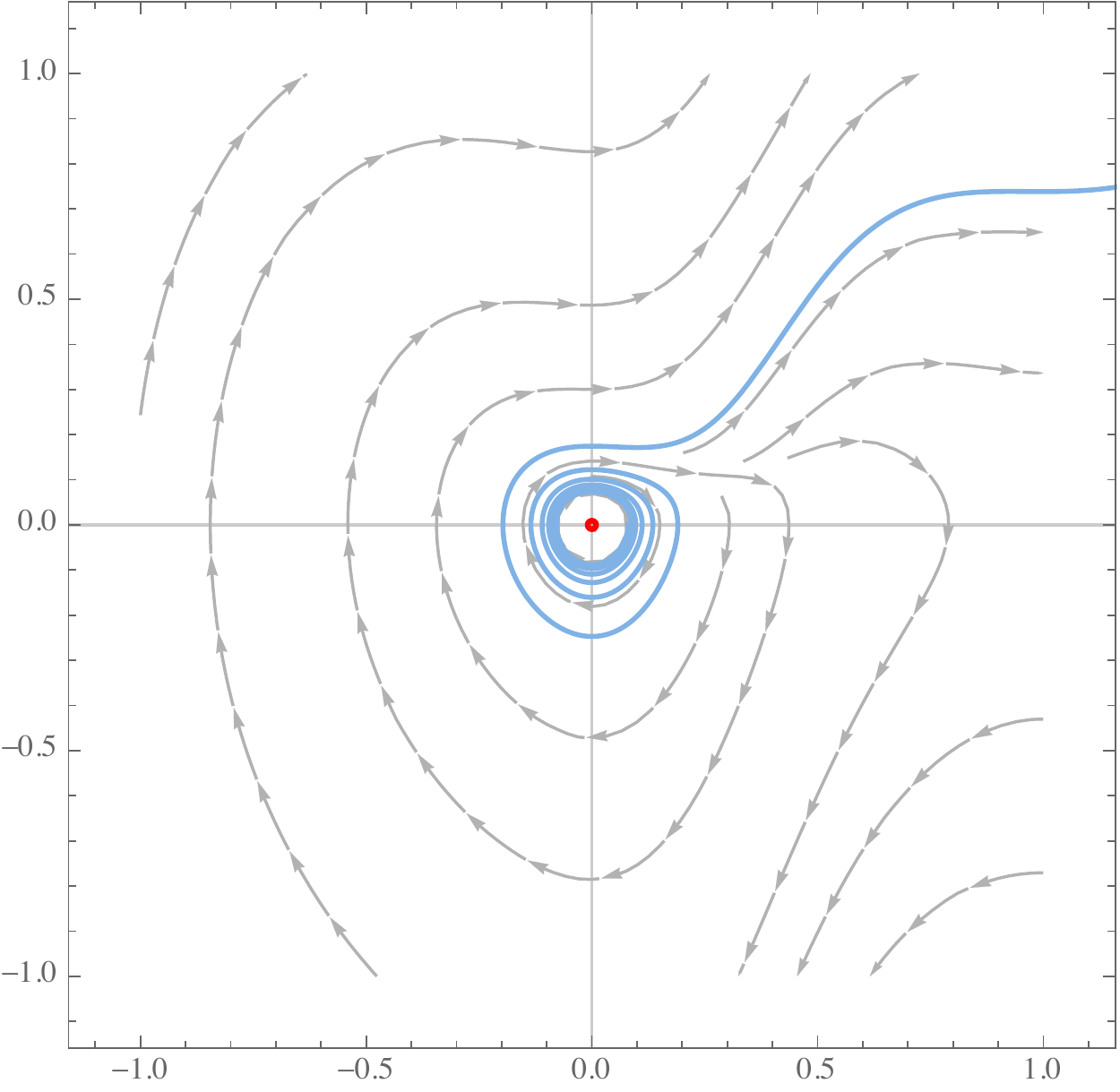}}
\subfigure[$c= 0.0025$]{\label{figHopfBLpos}\includegraphics[scale=.45, clip=true]{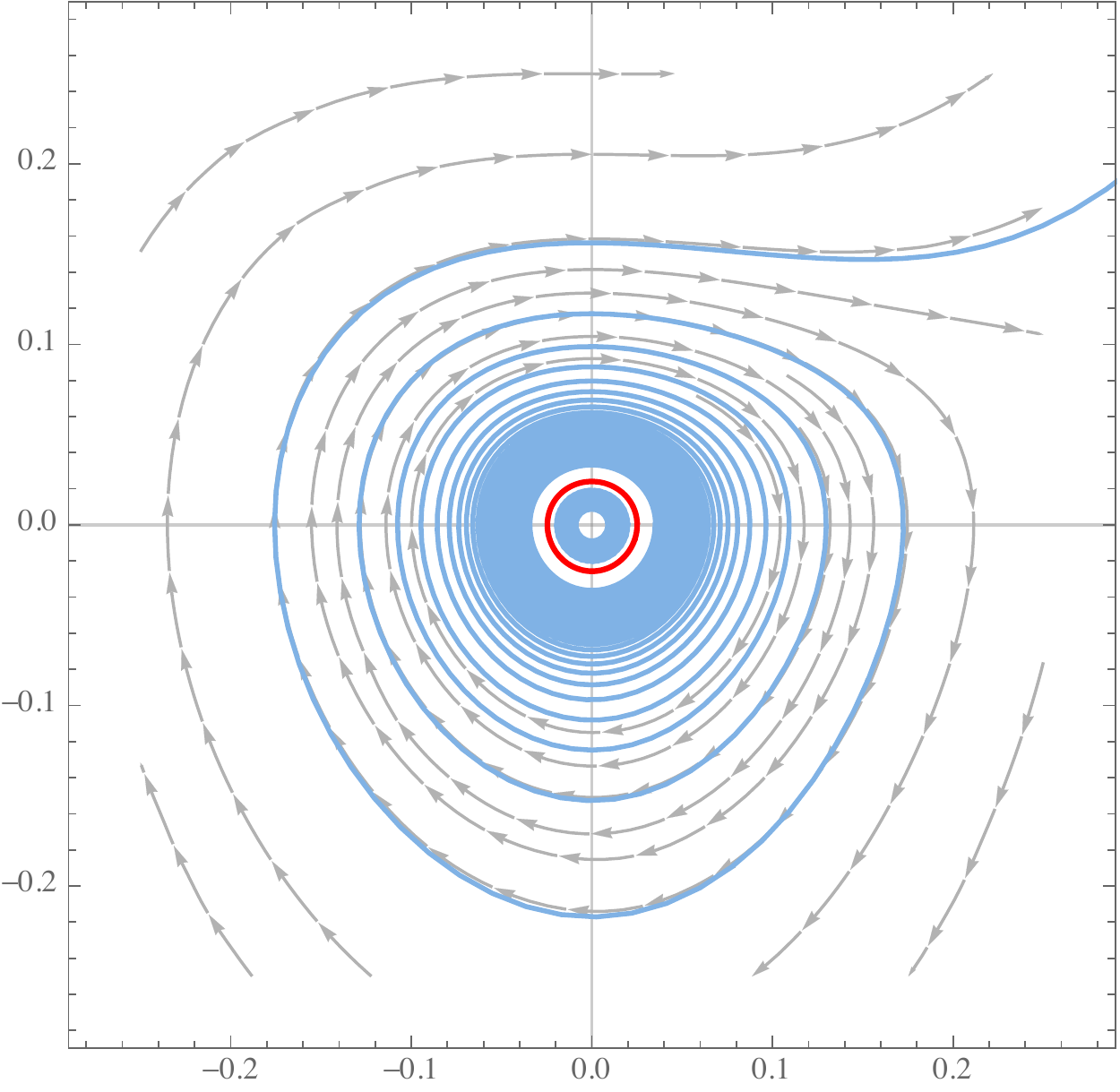}}
\subfigure[$c= 0.0025$]{\label{figHopfBLwave}\includegraphics[scale=.45, clip=true]{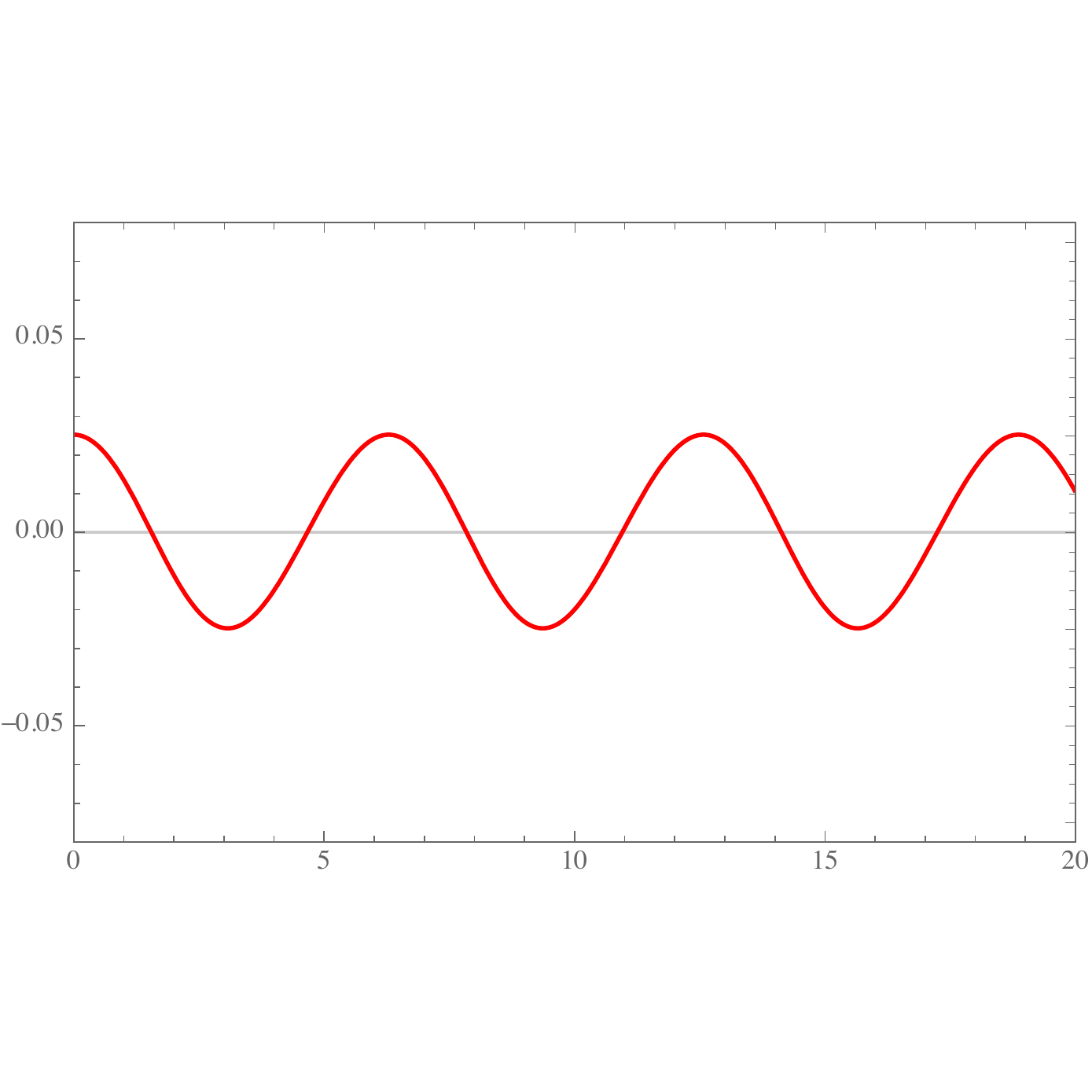}}
\end{center}
\caption{\small{Emergence of small-amplitude waves for the logistic Buckley-Leverett model \eqref{LogBLmodel}. Numerical solutions are shown in light blue color. Panel (a) shows the phase portrait of system \eqref{firstos} for the speed value $c = -0.05$; the origin is a repulsive node and all nearby solutions move away. Panel (b) shows the case when $c = 0$, the parameter value where a subcritical Hopf bifurcation occurs. Panel (c) shows the case where $c = 0.0025$: the orbit in red is a numerical approximation of the unique small amplitude periodic wave for this speed value. Panel (d) shows the graph (in red) of the periodic wave $\varphi$ as a function of the Galilean variable $z = x -ct$ (color online).}}\label{figHopfBL}
\end{figure}

Since the production term is the logistic function \eqref{BFg} as in the previous example, the function $\gamma = \gamma(u)$ is also given by \eqref{gamFisher} and, from \eqref{I0Fisher}, we also have that $I_0 = 3/5$. Upon substitution of the flux function \eqref{BLf} we obtain
\[
I_1 = \int_{-1/2}^1 f'(s) \gamma(s) \, ds = \int_{-1/2}^1 \frac{s(1-s)\sqrt{1-3s^2+2s^3}}{\big(s^2 + \tfrac{1}{2}(1-s)^2\big)^2} \, ds = 0.353458.
\]
Hence, the value of the speed of the homoclinic orbit from which the periodic loops with large period bifurcate is
\[
c_1 = \frac{I_1}{I_0} = 0.589097.
\]
This shows, for instance, that the saddle condition \eqref{H6} holds, inasmuch as $f'(1) = 0$. The values of $L$ and $J$ in \eqref{lasIs} are given by the following elliptic integrals, whose values are approximated numerically,
\begin{equation}
\label{LFisher}
L = 2 \int_{-1/2}^1 \sqrt{\frac{1-4s^3+3s^4}{1-3s^2+2s^3}} \,\; ds \approx 4.07339,
\end{equation}
\[
J = 2 \int_{-1/2}^1 \frac{s(1-s)}{\big(s^2 + \tfrac{1}{2}(1-s)^2\big)^2} \sqrt{\frac{1-4s^3+3s^4}{1-3s^2+2s^3}} \;\, ds \approx 1.62723.
\]
Thus, the non-degeneracy condition \eqref{H5} holds as $I_0 J \approx 0.976335 \neq L I_1 \approx 1.43977$. These calculations show that the logistic Buckley-Leverett model \eqref{LogBLmodel} satisfies the hypotheses of this paper.

In view that $c_1 > f'(1) = 0$, Theorem \ref{theoexistlp} then implies that the family of periodic waves with large period emerge for speed values in a neighborhood \emph{above} the value $c_1$, that is, for $c \in (0.5891, 0.5891 + \ep_1)$ with $\ep_1 > 0$ small. Figure \ref{FigHomoBL} shows a numerical approximation of the homoclinic loop to system \eqref{firstos} with speed $c_1$ (dashed line in blue) and a large-period wave from the family with speed $c \approx c_1 + 0.025$ (continuous line in orange). This is a family of spectrally unstable periodic waves in view of Theorem \ref{theolargei}.

\begin{figure}[t]
\begin{center}
\includegraphics[scale=.5, clip=true]{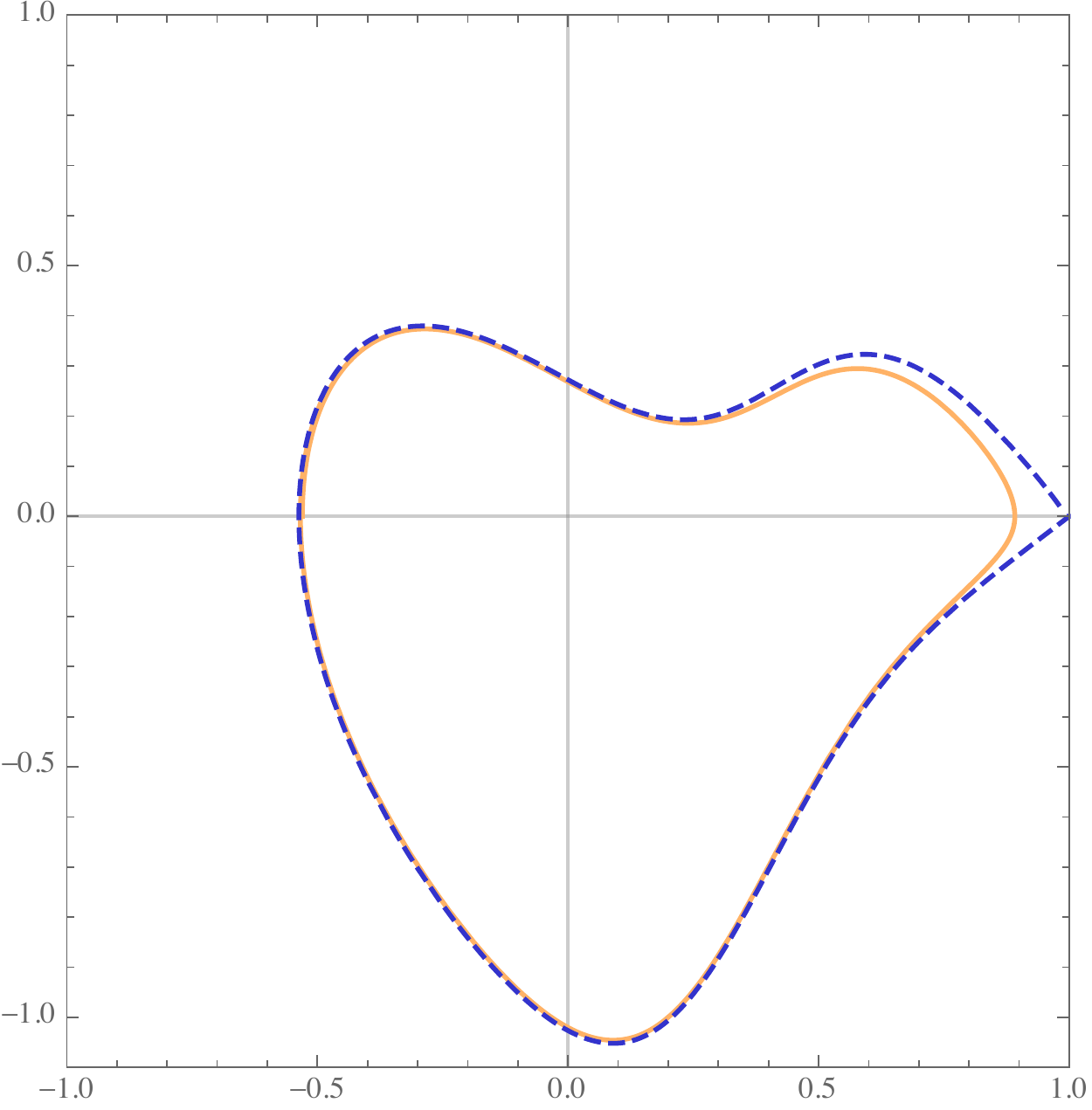}
\end{center}
\caption{\small{Numerical approximation of the homoclinic loop for the logistic Buckley-Leverett equation \eqref{LogBLmodel} with speed value $c_1 \approx 0.5891$ (in blue, dashed line) and the periodic wave nearby with speed value $c_1  + \epsilon$, $\epsilon \approx 0.025$ (solid, orange line; color online).}}\label{FigHomoBL}
\end{figure}

\subsection{Modified generalized Burgers-Fisher equation}
\label{secmBF}

The family of equations
\begin{equation}
\label{famgBF}
u_t + a u^m u_x = b u_{xx} + k u(1-u^m),
\end{equation}
where $a,b,k \in \R$ and $m \in \N$ are constants, is known in the literature as the \emph{generalized Burgers-Fisher equation} \cite{ChZh04,KaEl-S04,Vall18a}. The family underlies many types of traveling wave solutions: pulses, fronts, periodic wavetrains, both bounded or unbounded (see \cite{ZLZ13,LuYJ07,Vall18a} and the many references therein). 

As a final example, let us consider the following viscous balance law
\begin{equation}
\label{mBF}
u_t + \partial_x \Big( \tfrac{1}{4}u^4 - \tfrac{1}{3} u^3\Big) = u_{xx} + u-u^4, \qquad x \in \R, \; t >0,
\end{equation}
which is a modification of the generalized Burgers-Fisher equation with parameter values $a=1$, $b=1$, $k=1$, $m = 3$. We call it a \emph{modified generalized Burgers-Fisher equation} and it corresponds to nonlinear flux and reaction functions given by
\begin{equation}
\label{mBFf}
f(u) =  \frac{1}{4}u^4 - \frac{1}{3} u^3,
\end{equation}
and
\begin{equation}
\label{mBFg}
g(u) = u-u^4,
\end{equation}
respectively. Clearly, this pair satisfies assumptions \eqref{H1}, \eqref{H2} and \eqref{H3}, where the unique value $u_* \approx -0.72212$ such that \eqref{H3} holds is approximated numerically.
Upon calculation of the derivatives, one finds that
\[
\overline{a}_0 = -2 \neq 0,
\]
which means that hypothesis \eqref{H4} holds and the family of small amplitude waves occur for \emph{negative} speed values $c(\ep) = - \ep < 0 = c_0 = f'(0)$, sufficiently small. From Theorem \ref{thmexbded} and from Andronov-Hopf theory a \emph{supercritical} Hopf bifurcation occurs and the small amplitude periodic orbits are stable as solutions to the dynamical system \eqref{firstos} with speed value $c(\ep) = - \ep$. Figure \ref{figHopfmBF} illustrates the emergence of small-amplitude waves for the modified generalized Burgers-Fisher equation \eqref{mBF}. As before, we present the phase portraits of system \eqref{firstos} for different speed values; the numerically approximated solutions are shown in light purple color. Figure \ref{figHopfmBFPos} shows the phase plane  for the speed value $c = 0.05$; the origin is an attractive node and all nearby solutions converge at the origin. Figure \ref{figHopfmBF0} shows the case when $c = 0$, the parameter value where the supercritical Hopf bifurcation occurs; the origin is a center and solutions move away if they start sufficiently far from the origin and rotate locally around a linearized center otherwise. Figure \ref{figHopfmBFNeg} shows the case where $c = -0.005$: the orbit in red is a numerical approximation of the unique small amplitude periodic wave for this fixed speed, the origin is a repulsive node and nearby solutions both inside and outside the periodic orbit approach the periodic wave because it is stable as a solution to system \eqref{firstos}. Figure \ref{figHopfmBFwave} shows the graph (in red) of the periodic wave $\varphi$ as a function of the Galilean variable $z = x -ct$.

\begin{figure}[ht]
\begin{center}
\subfigure[$c = 0.05$]{\label{figHopfmBFPos}\includegraphics[scale=.45, clip=true]{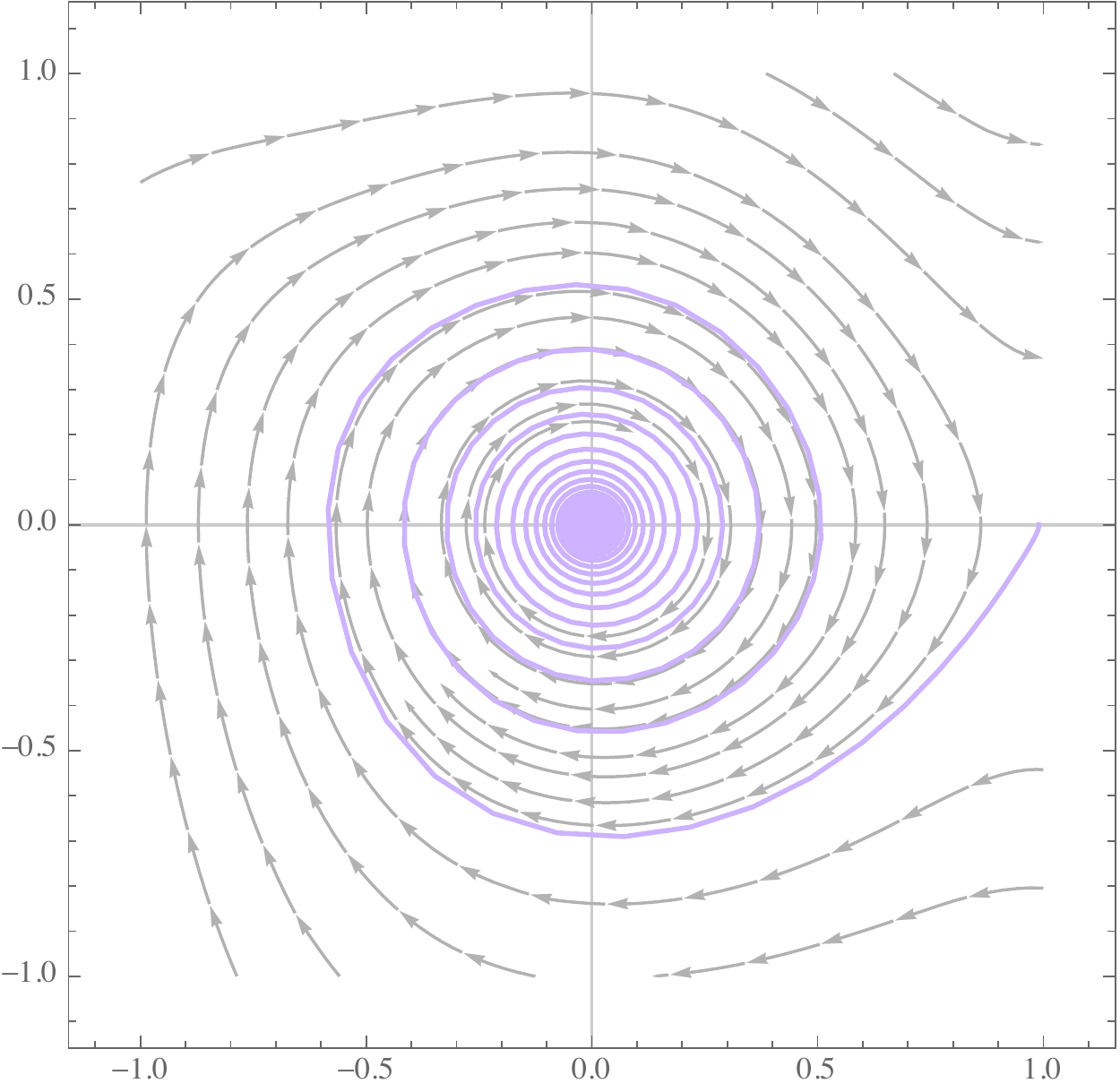}}
\subfigure[$c=0$]{\label{figHopfmBF0}\includegraphics[scale=.45, clip=true]{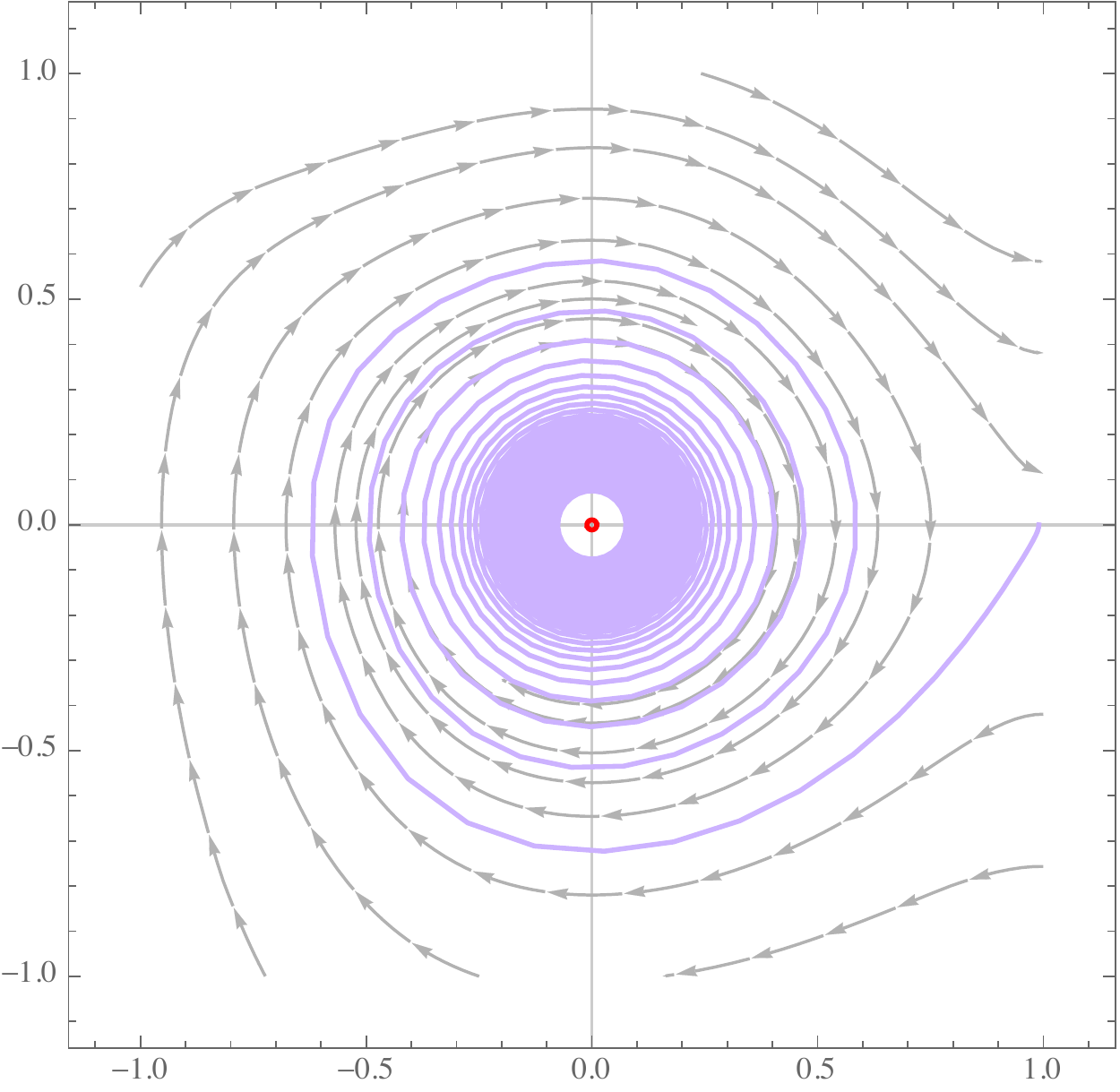}}
\subfigure[$c= -0.005$]{\label{figHopfmBFNeg}\includegraphics[scale=.45, clip=true]{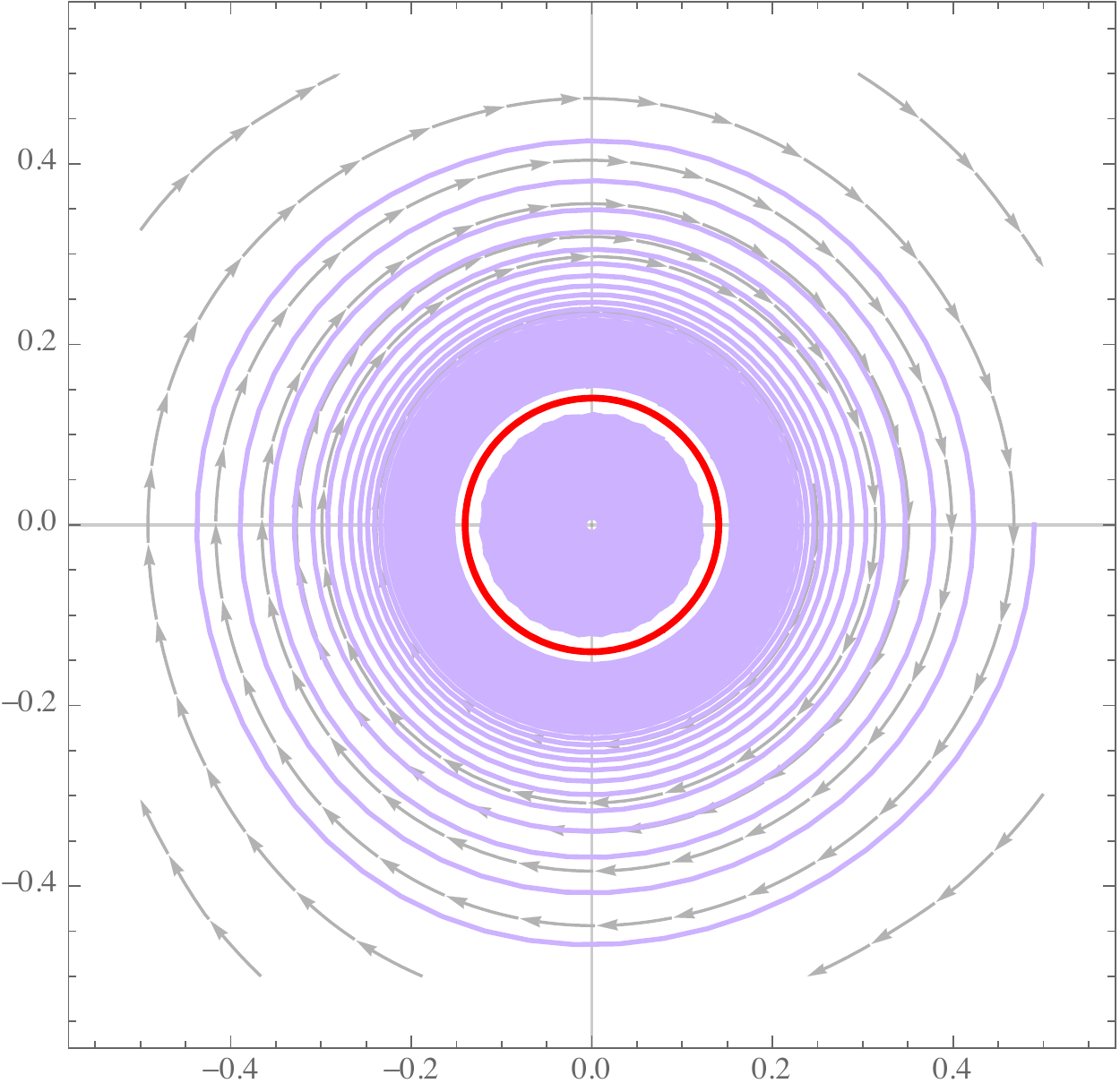}}
\subfigure[$c= -0.005$]{\label{figHopfmBFwave}\includegraphics[scale=.45, clip=true]{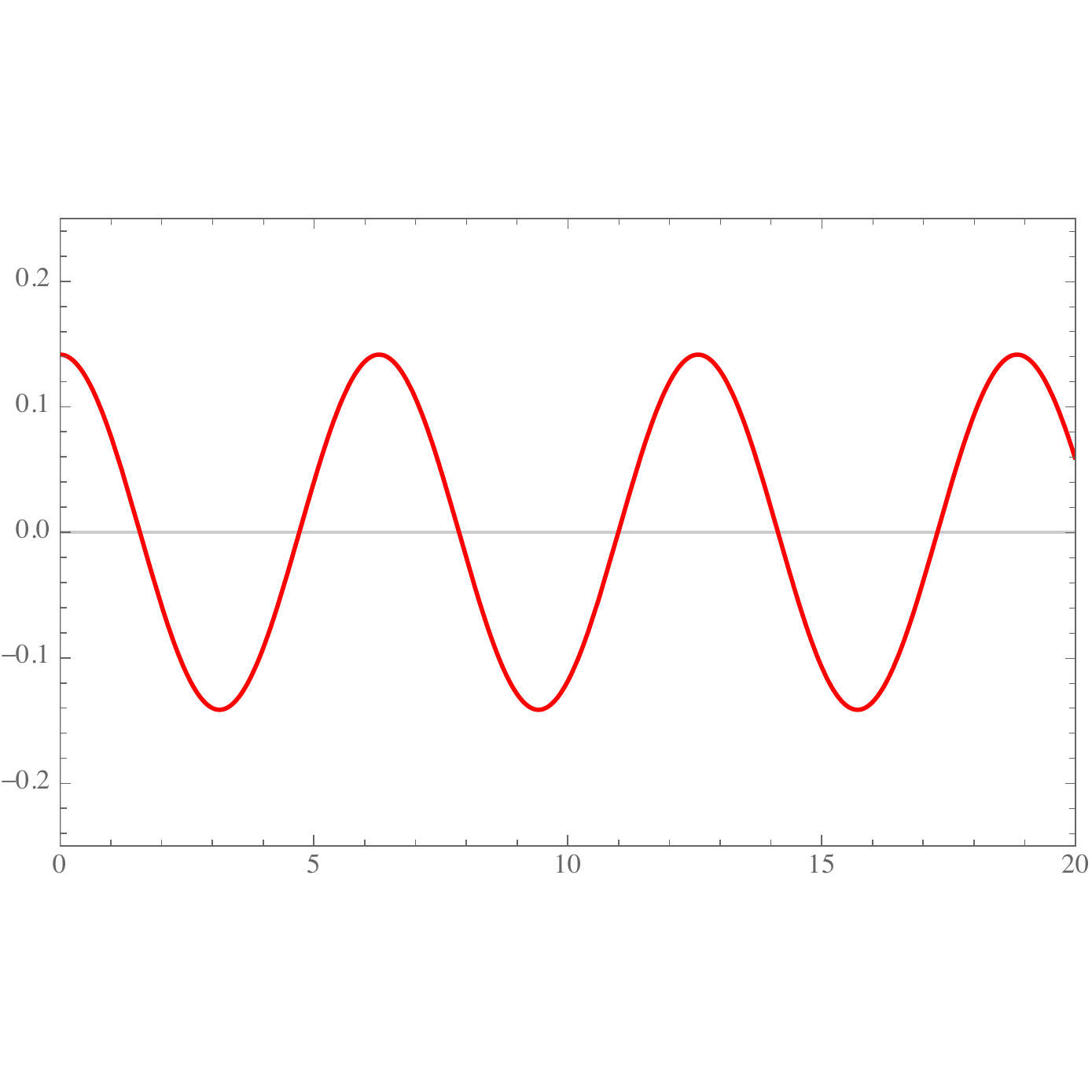}}
\end{center}
\caption{\small{Emergence of small-amplitude waves for the modified generalized Burgers-Fisher equation \eqref{mBF}. The bifurcation value for this case is $c_0 = f'(0) = 0$. Panel (a) shows the phase portrait of system \eqref{firstos} for the speed value $c = 0.05$. Panel (b) shows the case when $c = 0$, the parameter value where the supercritical Hopf bifurcation occurs. Panel (c) shows the case where $c = -0.005$: the orbit in red is a numerical approximation of the unique small amplitude periodic wave for this speed value. Panel (d) shows the graph (in red) of the periodic wave $\varphi$ as a function of the Galilean variable $z = x -ct$ (color online).}}\label{figHopfmBF}
\end{figure}

Like in the previous examples, one can numerically approximate the integrals in \eqref{lasIs}, namely
\[
\begin{aligned}
I_0 &= \frac{1}{\sqrt{5}} \int_{u_*}^1 \sqrt{3-5s^2+2s^5} \; ds \approx 0.979027, \\
I_1 &= \frac{1}{\sqrt{5}} \int_{u_*}^1 (s^3 - s^2)\sqrt{3-5s^2+2u^5}   \; ds \approx -0.129571,
\end{aligned}
\]
and,
\[
\begin{aligned}
L &= 2 \int_{u_*}^1 \sqrt{\frac{3-8s^5+5s^8}{3-5s^2+2s^5}}\; ds \approx 5.02904,\\
J&= 2 \int_{u_*}^1 (s^3 - s^2) \sqrt{\frac{3-8s^5+5s^8}{3-5s^2+2s^5}}\; ds \approx -1.27529,
\end{aligned}
\]
yielding the critical value of the homoclinic speed,
\[
c_1 = \frac{I_1}{I_0} \approx -0.132347,
\]
and, in turn, the verification of hypotheses \eqref{H5} and \eqref{H6}: $I_0 J \approx -1.24854 \neq LI_1 \approx -0.651619$ and $c_1 \neq f'(1) = 0$. Therefore, we conclude that the modified generalized Burgers-Fisher equation \eqref{mBF} satisfies hypotheses \eqref{H1} thru \eqref{H6} of this paper. Finally, observe that since $c_1 < 0 = f'(1)$ then Theorem \ref{theoexistlp} implies that the family of large period waves emerge for speed values below $c_1 \approx -0.1323$, that is, for $c \in (c_1 - \ep_1, c_1)$ with $\ep_1 > 0$ small. Figure \ref{FigHomomBF} shows a numerical approximation of the homoclinic loop to system \eqref{firstos} with speed $c_1$ (dashed line in blue) and a large-period wave from the family with speed $c \approx c_1 - 0.05$ (continuous line in orange). Once again, this is a family of spectrally unstable periodic waves in view of Theorem \ref{theolargei}.

\begin{figure}[t]
\begin{center}
\includegraphics[scale=.5, clip=true]{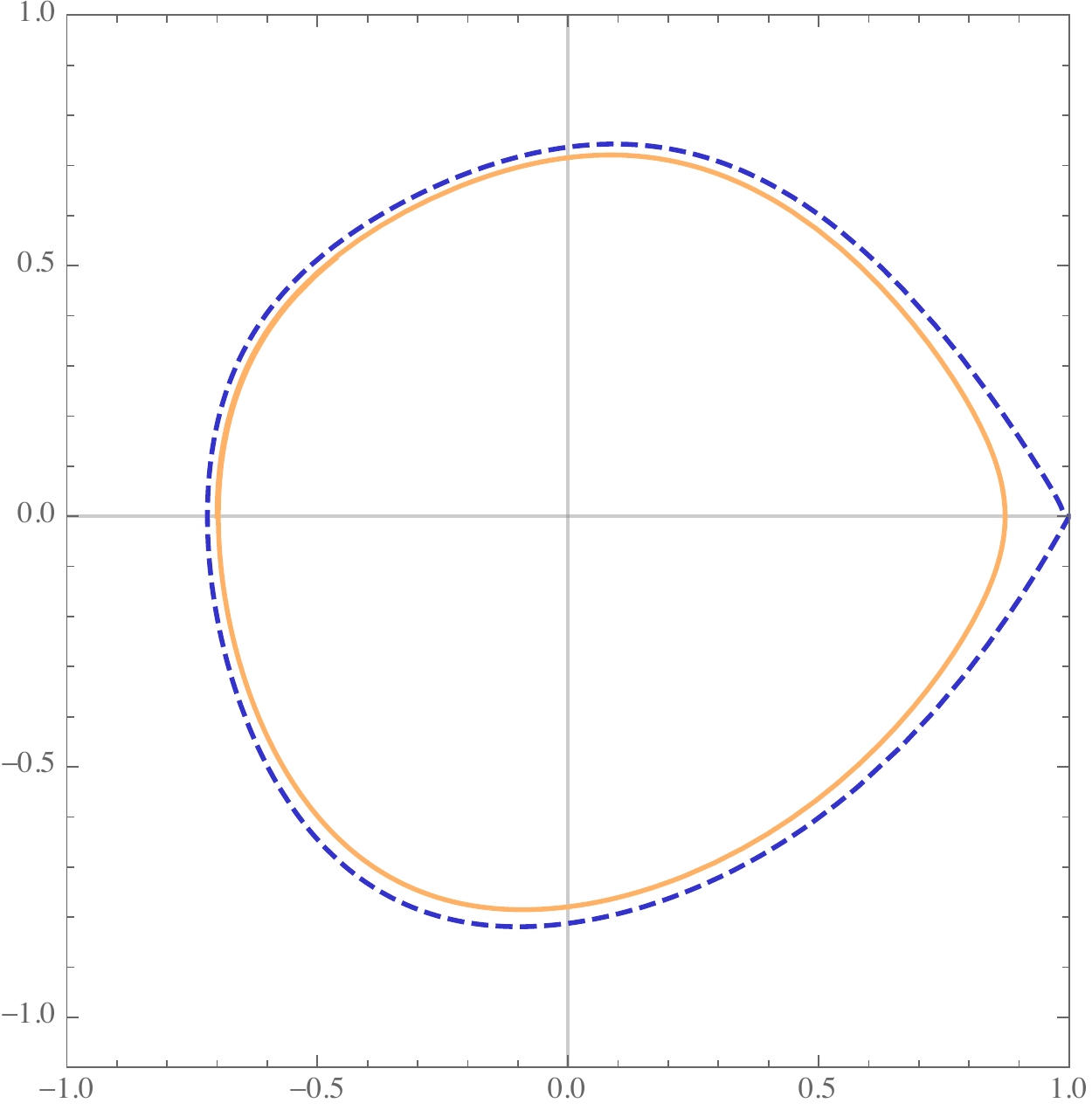}
\end{center}
\caption{\small{Numerical approximation of the homoclinic loop for the modified Burgers-Fisher equation \eqref{mBF} with speed value $c_1 \approx -0.1323$ (in blue, dashed line) and the periodic wave nearby with speed value $c_1  - \epsilon$, $\epsilon \approx 0.05$ (solid, orange line; color online).}}\label{FigHomomBF}
\end{figure}

\begin{remark}
The generalized Burgers-Fisher equation \eqref{famgBF} with the above parameter values ($a=1$, $b=1$, $k=1$, $m = 3$), namely,
\[
u_t + u^3 u_x = u_{xx} + u-u^4,
\]
\emph{does not} satisfy the genericity condition \eqref{H4} as the reader may easily verify. Hence, we are not able to apply the existence Theorem \ref{thmexbded}. This does not mean, of course, that small amplitude periodic waves may not emerge from a higher order (degenerate) Hopf bifurcation, a calculation that we do not pursue here.
\end{remark}

\section{Discussion}
\label{secdisc}


In this paper we have shown that, for a relatively simple (but large) class of scalar viscous balance laws, there exist two families of periodic traveling wave solutions. Both families emerge from bifurcation analyses: small-amplitude, finite period waves are generated via a Hopf bifurcation around a critical wave speed; on the other hand, and as it happens in many circumstances, traveling pulses on the real line are accompanied by periodic waves with arbitrarily large periods, emerging from what is called a homoclinic bifurcation.

We also examine the Floquet spectrum of the linearization around the two families of periodic wavetrains. In the small-amplitude case, it is shown that all the waves belonging to the family are spectrally unstable by the application of standard perturbation theory of linear operators. The instability is due to a structural assumption on the model equations: the instability of the origin as an equilibrium point of the reaction generates an unstable eigenvalue of an associated constant coefficient operator (the linearization around the zero solution), from which the linearization of a small-amplitude wave represents a perturbation. This approximation technique has been recently investigated by K\'{o}llar \emph{et al.} \cite{KDT19} to study the spectral stability of small-amplitude waves for scalar equations with Hamiltonian structure. It is remarkable that the same method can be applied to a family of evolution equations like \eqref{vbl} lacking any special structure whatsoever. In the case of large period waves, we verify the conditions under which the seminal result by Gardner \cite{Grd2} (of convergence of periodic spectra in the infinite-period limit to that of the underlying homoclinic wave) applies. The typical instability of the traveling pulse then produces unstable spectrum curves for the linearized operator around the periodic wave, proving in this fashion, spectral instability of each member of the family. We also present some examples which satisfy the hypotheses of this paper, among which the well-known (viscous) Burgers-Fisher equation stands out.

A natural question following our analysis is whether these periodic wavetrains are orbitally unstable as solutions to the nonlinear PDE.  For equations with specific structures, it is widely known that the spectral instability of a traveling wave solution is a key prerequisite to show their nonlinear (orbital) instability (see, e.g., \cite{GrilSha90, Lop02, ShaStr00}). In view that viscous balance laws of the form \eqref{vbl} lack special structures (such as symmetries, Hamiltonian form or complete integrability) the study of orbital instability of these periodic waves warrants further investigations. It would be also of interest to explore the relation of our spectral instability result with the more common modulational stability analysis, which is specialized to long wavelength perturbations. The modulational stability of both families of periodic waves for the Burgers-Fisher equation will be addressed in a companion paper \cite{AlPl1}.

\section*{Acknowledgements}

The authors thank Peter D. Miller, Robby Marangell, Jaime Angulo Pava and L. Miguel Rodrigues for useful 
conversations during a workshop at Casa Matem\'{a}tica Oaxaca. The work of E. \'{A}lvarez was partially supported by 
CONACyT (Mexico) through a scholarship for doctoral studies, grant no. 307909. The work of R. G. Plaza was partially supported by DGAPA-UNAM, program PAPIIT, grant IN-100318.

\appendix
\section{Non-degeneracy of the homoclinic orbit}
\label{apenon}

In this appendix we prove that the homoclinic orbit from Theorem \ref{theoexisthomo} (or Corollary \ref{corpulse}) satisfies a nondegeneracy condition in the sense established by Beyn \cite{Beyn90b}. Consider general (parametrized) dynamical systems of the form
\begin{equation}
\label{ds}
\frac{dy}{dz} = \bF (y, \mu), \qquad y \in \R^m, \;\; \mu \in \R^p, \, \, z \in \R,
\end{equation}
with $m,p \in \N$, $f \in C^1(\R^m\times \R^p;\R^m)$. Beyn \cite{Beyn90b} calls any pair $(y(z), \mu_*)$ a \emph{connecting orbit pair} if $y = y(z)$ is a solution to \eqref{ds} at $\mu = \mu_*$ for all $z \in \R$ and the limits
\[
y_\pm = \lim_{z \to \pm \infty} y(z),
\]
exist. Since $\bF$ is continuous, necessarily $\bF(y_\pm,\mu_*) = 0$. If $y_+ = y_-$ then the orbit is called homoclinic, and if $y_+ \neq y_-$ then it is called heteroclinic. For any $m \in \N$, $k \in \Z$, $k \geq 0$, let us denote the Banach spaces
\[
X_m^k = \left\{ \phi \in C^k(\R; \R^m) \, : \, \lim_{z \to \pm \infty} \frac{d^j \phi(z)}{dz^j} \, \text{exists for } \, j = 0,1, \ldots, k\right\},
\]
\[
\| \phi \|_{X_m^k} = \sum_{j=0}^k \sup_{z \in \R} \left|\frac{d^j \phi(z)}{dz^j}\right|.
\]
Note that for any $\phi \in X_m^k$ with $k \geq 1$, by the mean value theorem, there holds
\[
\frac{d^j \phi(z)}{dz^j} \to 0,  \qquad \text{as } \, z\to \pm \infty, \;\; \text{all } \, j = 1, \ldots, k.
\]

\begin{definition}[Beyn \cite{Beyn90a,Beyn90b}]
\label{defnondeg}
A connecting orbit pair $(y(z),\mu_*) \in X_m^1 \times \R^p$ of \eqref{ds} is called \emph{non-degenerate} if the following conditions hold:
\begin{itemize}
\item[(a)] The matrices
\[
\bA_\pm = \lim_{z \to \pm \infty} D_y\bF(y(z),\mu_*)
\]
are hyperbolic with stable dimensions $m^s_\pm$.
\item[(b)] $p = m^s_+ + m^s_- -1$.
\item[(c)] The only solutions $(w,\mu) \in X_m^1 \times \R^p$ to the variational system
\[
\frac{dw}{dz} = D_y \bF(y(z),\mu_*) w + D_\mu \bF (y(z),\mu_*) \mu
\]
are $w = k (dy/dz)$ and $\mu = 0$ for some constant $k \in \R$.
\end{itemize}
\end{definition}

Under assumptions \eqref{H1}, \eqref{H2}, \eqref{H3} and \eqref{H5}, let $(\bU,\bV)(z) = (\psi, \psi')(z)$, $z \in \R$, be the homoclinic loop of system \eqref{firstos} from Theorem \ref{theoexisthomo} and Corollary \ref{corpulse} with speed value $c = c_1$. Let us denote
\begin{equation}
\label{coefsbs}
\begin{aligned}
\hat{b}_1(z) &:= c_1 - f'(\psi(z)),\\
\hat{b}_0(z) &:= g'(\psi(z)) - f'(\psi(z)) \psi'(z).
\end{aligned}
\end{equation}
These coefficients are functions of class $C^2$ and uniformly bounded (in view that $\psi(z) \in [u_*,1]$, compact, for all $z \in \R$). Moreover, clearly,
\[
\hat{b}_1(z) \to c_1 - f'(1), \qquad \hat{b}_0(z) \to g'(1),
\]
exponentially as $z \to \pm \infty$. First we need the following auxiliary
\begin{lemma}
\label{lemaux}
Suppose there exist solutions $\zeta \in X_1^2$ and $\eta \in X_1^2$ to
\[
\cB \zeta := \zeta'' + \hat{b}_1(z) \zeta' + \hat{b}_0(z) \zeta = 0,
\]
and to
\[
\cB^* \eta := \eta'' - \hat{b}_1(z) \eta' + (\hat{b}_0(z) - \hat{b}_1'(z)) \eta = 0,
\]
respectively, such that
\[
\zeta, \eta \to 0 \qquad \text{as } \; z \to \pm \infty.
\]
Then, all other solutions $u,v \in X_1^2$ to $\cB u = 0$ and to $\cB^* v = 0$ are multiples of $\zeta$ and $\eta$, respectively.
\end{lemma}
\begin{proof}
Suppose $u \in X_1^2$ is a solution to $\cB u = 0$. Since $\lim_{z \to \pm \infty} u$ exists and $\lim_{z \to \pm \infty} d^ju / dz^j = 0$, $ j = 1,2$, it is then clear that the Wronskian
\[
\bw (z) := u \zeta' - u' \zeta,
\]
satisfies
\[
\bw' = - \hat{b}_1(z) \bw, \qquad z\in \R,
\]
\[
\bw \to 0, \qquad \text{as } \; z \to \pm \infty.
\]
Hence,
\[
\bw(z) = C_0 \exp \left( - \int_0^z \hat{b}_1(s) \, ds \right),
\]
for some constant $C_0 \in \R$. Now, since $\hat{b}_1(z) \to c_1 - f'(1)$ as $z \to \pm \infty$, it is easy to verify that
\[
\bw(z) = C_0 \exp \left( - \int_0^z \hat{b}_1(s) \, ds \right) \sim \overline{C}_0 \exp \big( (f'(1) - c_1) z\big),
\]
as $z \to \pm \infty$ for some other constant $\overline{C}_0 \in \R$. In view that
\[
\exp \big( (f'(1) - c_1) z\big) \rightarrow \begin{cases} 
\infty, & \text{if } \; f'(1) > c_1 \;\; \text{as } \, z \to \infty,\\ \infty, & \text{if } \; f'(1) < c_1 \;\; \text{as } \, z \to -\infty,
\end{cases}
\]
we conclude that $C_0 = 0$ (otherwise we contradict $\bw \to 0$ as $z \to \pm \infty$) and, hence, the Wronskian vanishes everywhere. This implies that $u = k \zeta$ for some constant $k \in \R$. The proof for the solution $v$ to $\cB^* v = 0$ is analogous.
\end{proof}

\begin{lemma}
\label{lemnondefgen}
The (homoclinic) connecting orbit pair $(\psi, \psi', c_1) \in X_2^1 \times \R$ for system \eqref{firstos} is non-degenerate in the sense of Definition \ref{defnondeg}.
\end{lemma}
\begin{proof}
We apply Proposition 2.1 of Beyn \cite{Beyn90b}, which states that any connecting orbit pair $(y(z),\mu_*) \in X_m^1 \times \R^p$ for a generic system of the form \eqref{ds} is non-degenerate if and only if the matrices $\bA_\pm$ are hyperbolic and the linear operator
\[
\left\{
\begin{aligned}
\cA \, &: \, X_m^1 \to X_m^0,\\
\cA &:= \frac{d}{dz} - \bA(z),
\end{aligned}
\right.
\]
with $\bA(z) := D_y \bF(y(z),\mu_*)$, has the following properties:
\begin{itemize}
\item[(i)] $\dim \ker \cA = 1$, $\dim \ker \cA^* = p$; and,
\item[(ii)] the $p \times p$ matrix 
\begin{equation}
\label{intE}
E = \int_{-\infty}^{\infty} \Phi(z)^\top D_\mu \bF (y(z),\mu_*) \, dz,
\end{equation}
is non-singular, where the $p$ columns $\Phi_i \in X_m^1$, $i =1, \ldots, p$ of $\Phi$ form a basis of $\ker \cA^*$.
\end{itemize}
Here the operator $\cA^* : X_m^1 \to X_m^0$ is given by
\[
\cA^* = \frac{d}{dz} + \bA(z)^\top.
\]
(See Proposition 2.1, p. 383 in \cite{Beyn90b}, for further details.) In our case, for system \eqref{firstos} the matrices $\bA$ are given by
\[
\bA(z) = D_{(U,V)} \begin{pmatrix} F(U,V) \\ G(U,V) \end{pmatrix}_{|(U,V,c) = (\psi, \psi', c_1)} 
= \begin{pmatrix}
0 & 1 \\ f''(\psi(z))\psi'(z) - g'(\psi(z)) & -c_1 + f'(\psi(z))
\end{pmatrix},
\]
with asymptotic limits
\[
\bA_{\pm} = \lim_{z \to \pm \infty} \bA(z) = \begin{pmatrix} 0 & 1 \\ - g'(1) & -c_1 + f'(1) \end{pmatrix} = (A_1)_{|c=c_1},
\]
and with eigenvalues \eqref{evaluessaddle}. Hence, they are clearly hyperbolic with stable dimension $m^s_+ = m^s_- = 1$. Now, from Corollary \ref{corpulse} we know that $\psi \in C^3(\R)$ and hence
\[
\Phi(z) := \begin{pmatrix}
\psi' \\ \psi''
\end{pmatrix} \, \in \, X_2^1,
\]
is a solution to $\cA \Phi = 0$, because \eqref{psi3} can be written as
\[
\frac{d}{dz} \begin{pmatrix} \psi' \\ \psi'' \end{pmatrix} - \bA(z) \begin{pmatrix} \psi' \\ \psi'' \end{pmatrix}  = 0.
\]

If we define $\zeta(z) := \psi'(z)$, $z \in \R$, then $\zeta \in X_1^2$ and it is a solution to 
\[
\cB \zeta = \zeta'' + \hat{b}_1(z) \zeta' + \hat{b}_0(z) \zeta = 0,
\]
where the coefficients $\hat{b}_j(z)$, $j = 0,1$, are defined in \eqref{coefsbs}. Moreover, $\zeta \to 0$ as $z \to \pm \infty$. Therefore, by Lemma \ref{lemaux} we have that any other solution $u \in X_1^2$ to $\cB u = 0$ is a multiple of $\zeta$. This implies, in turn, that
\[
\dim \ker \cA = 1, \qquad \ker \cA = \Span \{ \Phi \} \subset X_2^1.
\]

Let us now define
\[
\eta(z) := \chi(z) \zeta(z), \qquad \chi(z) := \exp \left( \int_0^z \hat{b}_1(s) \, ds \right).
\]
Then upon differentiation
\[
\begin{aligned}
\chi' &= \hat{b}_1 \chi, & \chi'' &= \hat{b}_1' \chi + \hat{b}_1 \chi',\\
\eta' &= \chi' \zeta + \chi \zeta', & \eta'' &= \chi'' \zeta + 2 \chi' \zeta' + \chi \zeta'',
\end{aligned}
\]
yielding
\[
\cB^* \eta = \eta'' - \hat{b}_1 \eta' + (\hat{b}_0 - \hat{b}_1') \eta = \chi \big( \zeta'' + \hat{b}_1 \zeta' + \hat{b}_0 \zeta \big) = \chi \cB \zeta = 0.
\]
Moreover, $\eta = \chi \zeta \in X_1^2$. Indeed, it is clear that $\eta \in C^2$, inasmuch as $\zeta, \chi \in C^2$. In addition, 
\[
\chi(z) \sim \exp \big( (c_1 - f'(1)) z\big),
\]
as $z \to \pm \infty$. But since $\zeta$ decays as $\exp(\lambda_1(c_1)z)$ as $z \to \infty$ and as $\exp(\lambda_2(c_1) z)$ as $z \to -\infty$, where $\lambda_1$ and $\lambda_2$ are given by \eqref{evaluessaddle}, it is easy to verify that $\eta \to 0$ as $z \to \pm \infty$. For example, if $z \to \infty$ we have
\[
\eta = \chi \zeta \sim \exp \left( \Big[ \tfrac{1}{2}(c_1 - f'(1)) - \tfrac{1}{2} \sqrt{(c_1-f'(1))^2 - 4g'(1)} \Big] z\right) \to 0,
\]
independently of the sign of $c_1 - f'(1)$. In the same fashion, $\eta \to 0$ as $z \to - \infty$. Finally, it is clear that $\eta', \eta'' \to 0$ as $z \to \pm \infty$. Whence, we have that $\eta \in X_1^2$ and Lemma \ref{lemaux} implies that any other solution $v \in X_1^2$ of $\cB^* v = 0$ is a multiple of $\eta$. If we further define
\[
\xi(z) := \int_0^z \hat{b}_0(s) \eta(s) \, ds,
\]
then from the exponential decay of $\eta$ and boundedness of $\hat{b}_0$ it is clear that $\xi$ has finite limits as $z \to \pm \infty$. Also, $\xi' = \eta \hat{b}_0 \to 0$ and $\xi'' = \eta' \hat{b}_0 + \eta \hat{b}_0' \to 0$ as $z \to \pm \infty$. We conclude that
\[
\Psi(z) := \begin{pmatrix} \xi \\ \eta \end{pmatrix} \in X_2^1
\]
is the only (up to constants) solution to 
\[
\cA^* \Psi = \frac{d}{dz} \Psi + \bA(z)^\top \Psi = 0.
\]
This yields $\dim \ker \cA^*  =1$ and $\ker \cA^* = \Span \{ \Psi \} \subset X_2^1$. Condition (i) is therefore verified inasmuch as we have one bifurcation parameter $c$ and $p = 1$.

Finally, the integral in \eqref{intE} reduces to
\[
\begin{aligned}
E = \int_{-\infty}^\infty \begin{pmatrix} \xi \\ \eta \end{pmatrix}^\top \!\!\partial_c \begin{pmatrix} F \\ G \end{pmatrix}_{|(\psi,\psi',c_1)} \! dz 
&= \int_{-\infty}^\infty \begin{pmatrix} \xi \\ \eta \end{pmatrix}^\top \!\!\begin{pmatrix} 0 \\ -\psi'(z) \end{pmatrix} \, dz  \\
&= - \int_{-\infty}^\infty \chi(z) \psi'(z)^2 \, dz\\
&= - \int_{-\infty}^{\infty} \exp \left( \int_0^z \hat{b}_1(s) \, ds \right) \psi'(z)^2 \, dz \neq 0.
\end{aligned}
\]
verifying, in this fashion, condition (ii). The lemma is proved.
\end{proof}

\def\cprime{$'\!\!$} \def\cprimel{$'\!$}





\end{document}